\documentclass[reqno]{amsart}

\title[Sampling inverse subordinators and subdiffusions]{Sampling inverse subordinators and subdiffusions}

\author{Ivan Bio\v{c}i\'{c}}
\author{Daniel E. Cedeño-Girón}
\author{Bruno Toaldo}

	\date{\today}
\thanks{The authors acknowledge financial support under the National Recovery and Resilience Plan (NRRP), Mission 4, Component 2, Investment 1.1, Call for tender No. 104 published on 2.2.2022 by the Italian Ministry of University and Research (MUR), funded by the European Union – NextGenerationEU– Project Title “Non–Markovian Dynamics and Non-local Equations” – 202277N5H9 - CUP: D53D23005670006 - Grant Assignment Decree No. 973 adopted on June 30, 2023, by the Italian Ministry of University and Research (MUR)}
\thanks{The author Bruno Toaldo are partially supported by Gruppo Nazionale per l’Analisi Matematica, la Probabilità e le loro Applicazioni (GNAMPA-INdAM)}
\thanks{The authors would like to thank the Isaac Newton Institute for Mathematical Sciences, Cambridge, for support and hospitality during the programme Stochastic systems for anomalous diffusion, where work on this paper was undertaken. This work was supported by EPSRC grant EP/Z000580/1}
\thanks{The authors would like to thank Prof. Aleksandar Mijatović for fruitful discussions on the
topic that improved a previous version of this manuscript.}

\thanks{The authors would like to thank Prof. Eli Barkai for very useful remarks on weak ergodicity breaking}

\thanks{The authors would like to thank two anonymous referees whose remarks and suggestions considerably improved a previous version of this manuscript.}

\usepackage[top=30mm, bottom=30mm, left=30mm, right=30mm]{geometry}
\usepackage{amssymb, amsmath, amsfonts}
\usepackage{natbib}
\usepackage{lscape, comment}
\usepackage{color}
\usepackage{mathtools}
\usepackage{dsfont}
\usepackage{mathrsfs}
\usepackage{bm}
\usepackage{graphicx}
\usepackage{wrapfig}
\usepackage{caption}
\usepackage{tikz}
\usetikzlibrary{arrows, decorations.markings, arrows.meta}
\usetikzlibrary{calc}

\usepackage{cancel}

\usepackage{soul}

\DeclareMathAlphabet{\mathpzc}{OT1}{pzc}{m}{it}

\bibpunct{[}{]}{;}{n}{,}{,}
\usepackage[linesnumbered]{algorithm2e}
\usepackage{algpseudocode}
\RestyleAlgo{ruled}

\usepackage{bbm} 
\newtheorem{theorem}{Theorem}[section] 
\newtheorem{lemma}[theorem]{Lemma}
\newtheorem{corollary}[theorem]{Corollary}
\newtheorem{proposition}[theorem]{Proposition}
\theoremstyle{remark}
\newtheorem{remark}[theorem]{Remark}
\newtheorem{example}[theorem]{Example}
\numberwithin{equation}{section} 
\usepackage{authblk}
\usepackage{natbib}
\usepackage{hyperref}
\hypersetup{
	colorlinks=true,
}
\usepackage{caption}
\usepackage{subcaption}
\usepackage{graphicx}
\usepackage[pagewise]{lineno}
\usepackage{float} 

\newcommand{\R}{\mathbb{R}}
\newcommand{\E}{\mathds{E}}
\newcommand{\var}{\mathds{V}\text{ar}}
\newcommand{\pr}{\mathds{P}}
\newcommand{\1}{\mathds{1}}
\newcommand{\N}{\mathbb{N}}
\def \ls { \left[ }
\def \rs {\right] }

\DeclareMathAlphabet{\mathpzc}{OT1}{pzc}{m}{it}

\def \l { \left( }
\def \r {\right) }

\allowdisplaybreaks

\newtheorem*{assumption*}{\assumptionnumber}
\providecommand{\assumptionnumber}{}
\makeatletter
\newenvironment{assumption}[2]
{%
	\renewcommand{\assumptionnumber}{(\textbf{#1#2})}%
	\begin{assumption*}%
		\protected@edef\@currentlabel{(\textbf{#1#2})}%
	}
	{%
	\end{assumption*}
}
\makeatother

\begin{document}
\maketitle

\begin{abstract}
	In this paper, a method to exactly sample the trajectories of inverse subordinators (in the sense of the finite-dimensional distributions), jointly with the undershooting or overshooting process, is provided. The method applies to general {strictly increasing} subordinators. The (random) running times of these algorithms have finite moments and explicit bounds for the expectations are provided. Additionally, the Monte Carlo approximation of functionals of subdiffusive processes (in the form of time-changed Feller processes) is considered where a central limit theorem and the Berry-Esseen bounds are proved. The approximation of time-changed It\^o diffusions is also studied. The strong error, as a function of the time step, is explicitly evaluated demonstrating the strong convergence, and the algorithm's complexity is provided. The Monte Carlo approximation of functionals and its properties for the approximate method is studied as well. An application of our algorithms in the context of weak ergodicity breaking of subdiffusion is also discussed.
\end{abstract}

\tableofcontents

\section{Introduction}
{The main results of this paper include exact simulation algorithms for the finite-dimensional distributions of inverse subordinators, jointly with their undershooting and overshooting processes, as well as applications of these algorithms to time-changed Feller processes. The algorithms have running times with finite moments, and explicit bounds for the expectations are derived. Then, we also establish convergence results for Monte Carlo approximations of functionals of these time-changed processes together with central limit theorems, Berry–Esseen estimates, and strong error analysis. When It\^o diffusions cannot be sampled exactly we also provide a Euler-Maruyama scheme for approximation of the time-changed process, together with the Monte Carlo analysis of functionals. Algorithmic complexity is always quantified. All of this is motivated by the interplay between continuous time random walks (CTRWs), time-changed Markov processes and fractional-type non-local equations, that gained considerable popularity in the last decades.} This interplay has the central role in modeling anomalous diffusions, especially in the context of Hamiltonian chaos \cite{zaslavsky3, zaslavsky4, zaslavsky5}, Hamiltonian systems describing cellular flows \cite{hairer2, hairer1}, trapping models \cite{benarous, savtoa}, anomalous heat conduction \cite{FALCINI2019584, povstenko2015fractional}, option pricing \cite{Torricelli2020}, neuronal modeling \cite{APT20}, and many others, see e.g. \cite{barkai0, barkai1, fedotov1, fedotov3, fedotov2, gianni,  barkai2, METZLER20001, giannisim}, and references therein for several possible applications and modeling aspects of the theory.

One of the milestones in this context is the seminal paper \citep{Montroll1965} from which the CTRWs were rapidly adapted and tested for their suitability to describe anomalous diffusions in carrier transport in amorphous solids \citep{Scher1973a, Scher1973b, Scher1975}.
A CTRW in $\R^d$ describes the motion of a particle performing independent random jumps paced by independent residence times. In formulae: If $\l J_i, i \in\N \r$ is a sequence of non-negative absolutely continuous independent identically distributed random variables (residence times), then a CTRW is defined as the process $X=(X_t,\, t\geq0)$ given by
\begin{align}
    X_t = X_0 + \sum_{i=1}^{N(t)} Y_i\,. 
\end{align}
Here $(Y_i,\, i\in\N)$ is a sequence of independent identically distributed random variables in $\R^d$ and $N(t)\coloneqq \inf\{ n\in\N:\, \sum_{i=1}^nJ_i \leq t \}$.
CTRWs lead, after a scaling limit, to fractional partial differential equations (PDEs) under appropriate assumptions on the non-exponential residence times and the jumps. In practice, one can introduce a scale parameter to get a scaled process $X^c_t$ that converges as $c \to \infty$ (in a suitable sense) to another process $(X^\infty_t,\, t\geq0)$ such that $\E^x[u(X^\infty_t)]$ solves the fractional equation
\begin{align}
    \partial_t^\alpha q(x,t) = G q(x,t), \qquad x \in \R^d,\, t>0,
    \label{frac}
\end{align}
with $q(x,0)=u(x)$. Here $\partial_t^\alpha$ is the fractional derivative in the Caputo sense, i.e.
\begin{align}
    \partial_t^\alpha q(x,t) \, = \, \frac{1}{\Gamma(1-\alpha)} \partial_t \int_0^t (q(x,s)-q(x,0)) (t-s)^{-\alpha} ds,
    \label{airport}
\end{align}
and $G$ is a linear operator connected to $X^\infty$, {see e.g. \cite[Chapter 4]{meerschaert2019stochastic} for the foundation of this theory, and a more detailed explanation in Subsection \ref{sec:MCA}}.
Hence, since very often CTRWs can be simulated exactly, a lot of efforts have been dedicated to approximations of solutions to fractional PDEs using Monte Carlo methods based on CTRWs, see e.g. \citep{Fulger2008, Gorenflo2007, Uchaikin2003} and the references therein. The recent work \citep{Kolokoltsov2023} also provides convergence rates for functional limit theorems for CTRW approximations to the fractional evolution. We also note that CTRW approximations have an effective alternative in numerical methods, see \cite{Leonenko2022a}.

It turns out that the CTRWs limit processes $X^\infty$ are time-changed Markov processes, {even in a very general setting. This has been made finally clear in \cite{straka2011lagging} and the theory is as follows. Define} $(S_n^c, T_n^c)\coloneqq \sum_{i=1}^n (Y_i^c, J_i^c)$ and {assume that the sequence $(S_n^c, T_n^c), n \in\N$, is a discrete time Markov chain, i.e., the jumps and the waiting times are not necessarily independent.} Suppose that $( S_{[cu]}^c, T_{[cu]}^c ) \to (M_u, \sigma_u)$, as $c \to +\infty$, weakly in the space of c\`adl\`ag functions endowed with the Skorohod $J_1$ topology, where $( (M_u, \sigma_u),\, u\geq0 )$ is a Feller process. Then $X^\infty_t = (M_{L_t-})^+ \coloneqq \lim_{\delta_1 \searrow 0}\lim_{\delta_2 \nearrow 0}M_{L_{t+\delta_1}+\delta_2}$ where $L_t \coloneqq \inf\{s\geq0:\, \sigma_s >t \}$, see \cite[Section 2]{Meerschaert2014}. Note that when $M$ and $\sigma$ are independent, one has that $M_{L_t-} = M_{L_t}$ a.s. {(by \cite[Lemma 3.9]{straka2011lagging})} and the processes $X^\infty$ are {often called subdiffusions} in the sense that the process $L = (L_t,\, t\geq0)$, used as a time-change, induces intervals of constancy whose duration is arbitrarily distributed (depending on the jumps of $\sigma$), and as such are suitable for modeling sticking or trapping effects in an evolution. The mean square displacement for these processes has been studied in several cases (see, e.g., \cite{KOCHUBEI2008252, toaldo2015}) and shown to be sublinear.

The stochastic representation of the solution to \eqref{frac} is obtained by assuming that $M$ is a Markov process with the infinitesimal generator $G$, and by assuming that $\sigma$ is a positively skewed stable process (i.e. a stable subordinator) independent of $M$, see \cite{Baeumer2001}. If $\sigma$ is a general subordinator independent from $M$, then the governing equation is as in \eqref{frac} {where the non-local time operator has the same form of \eqref{airport} but with a different convolution kernel} (see e.g. \cite{chen, hernandez2017, kochubei, kolokoltsov2015, Meerschaert2004}). The process $L= (L_t,\, t\geq0)$ is then called an inverse subordinator (or local time), see more on these processes in \cite{tlms, bertoin1996}. If the random variables ${J_i^c}$ and ${Y_i^c}$ {(residence times and jumps)} are not independent, then the processes $(M_u,\, u\geq0)$ and $(\sigma_u,\, u\geq0)$ are not independent. In this case, the governing equation is non-local in both variables $t$ and $x$ in a coupled way, i.e. the non-local operator acts jointly on both variables $t$ and $x$, see \cite{ascione2024,Baeumer2005, harmonic}.

In light of this relationship between fractional PDEs and subordinators, simulating and inverting subordinator paths raised a new strategy to solve non-local PDEs numerically, see e.g. \citep{Kobayashi2016, Lv2020, Magdziarz2007}. The papers \cite{Kobayashi2012,  Kobayashi2019, Kobayashi2016, Kobayashi2011} provided a framework for connecting It\^o diffusions time-changed with inverse subordinators and their associated transition densities, where approximations of these processes and the convergence rate were also studied. 
In these papers, however, the paths of inverse subordinators were always simulated approximately, from the trajectories of the corresponding subordinator. 
The exact sampling of the inverse subordinator has been done exactly, only for a fixed single time and in the stable case, in the paper \citep{kolokoltsov2021}, where the Monte Carlo estimator for the stochastic representation of the solution to the fractional Cauchy problem was studied. Also, very recently, in \citep{DiGregorio2024} the authors show a paths's approximation for L\'evy processes time-changed by inverse subordinators using a martingale approach (see Section 7 there).

Clearly, a method to make exact samples of trajectories of an inverse subordinator (i.e. a method for exact sampling from the finite-dimensional distribution) would be crucial in the context of modeling diffusions subject to sticking or trapping, and for approximating the associated non-local kinetic equations. With the so far developed techniques it is impossible to perform this sample exactly, but only using approximations from the paths of the subordinator, and in some specific cases. However, the single time exact sampling of the inverse subordinator is already possible using different techniques. This has been done in \citep{Chi2016, Dassios2020} for stable subordinators and their truncated and tempered versions, and in \citep{Jorge2023b, Jorge2023a} which are the more general and up-to-date contributions. The latter papers, provide an exact and fast algorithm to sample a quite general class of inverse subordinators, at a fixed single time. This will be the starting point for the algorithms developed in this paper.

\subsection{Main results and structure of the paper}
In this work we develop a method to make exact samples from the finite-dimensional distributions of the processes 
\begin{align}
( (L_t, \gamma_t),\, t\geq0 ), \enskip  ( (L_t, \Gamma_t),\, t\geq0 ) \enskip\text{ and }\enskip ((L_t, \gamma_t, \Gamma_t), t \geq 0 ).
\label{intro_vector}
\end{align}
Here $L_t$ is the inverse of the strictly increasing subordinator $\sigma = (\sigma_t,\, t\geq0)$, $\gamma_t = t-\sigma_{L_t-}$ where $\sigma_{L_t-} \coloneqq \lim_{\delta\downarrow 0}\sigma_{L_t-\delta}$ is called the age process, and $\Gamma_t = \sigma_{L_t}-t$ is called the remaining lifetime process. The processes $\sigma_{L_t-}$ and $\sigma_{L_t}$ are called the undershooting and overshooting of the subordinator, respectively. Our method can be applied to the inverse of any subordinator with strictly increasing trajectories. {The algorithms require exact sampling from the single time distribution of the processes in \eqref{intro_vector}.} This is currently possible for a fairly general class of inverse subordinators, using the method described in \cite{Jorge2023b, Jorge2023a}. {This method, described in Section \ref{sec:prelim}, gives the possibility to make exact simulations of the first passage time event of a subordinator through a quite general curve, jointly with its undershooting and overshooting, for a broad class of subordinators. However, since subordinators start at zero, one has that $L_0=H_0=D_0=0$ almost surely, and thus the method above does not allow to sample $(L_t, H_t, D_t)$ conditionally on different values of $H_0$ and/or $D_0$, which is needed in our algorithms. Hence we will simplify the method in \cite{Jorge2023b, Jorge2023a} to a constant boundary while we extend it to include conditioning on different starting points. Furthermore, the algorithm for the finite-dimensional distributions of the second process in \eqref{intro_vector} does not require to know the undershooting: this gives us the possibility to use also methods from \cite{Dassios2020} which are easier but still efficient.} The sampling of trajectories is detailed in Section \ref{secpathinv}, and the procedures are given in  Algorithms \ref{alg:2}, \ref{alg:3} and \ref{alg:triplet}. The computational complexity is {discussed, together with the methods to sample the single time distributions, in Section \ref{sec:compcomp}.}

With this at hand, we then study some properties of a Monte Carlo estimator for the quantity $\E^x [u( M_{L_{t_1}}, \cdots, M_{L_{t_n}} )]$, where $M=(M_t,\, t\geq0)$ is an arbitrary Feller process in $\R^d$, $\E^x$ denotes the expectation when $M_0=x$ a.s., and $u : \R^{d \times n} \mapsto \R$ are suitable functionals, for arbitrary $n\in\N$ and any choice of times $t_1, \cdots, t_n$. For the Monte Carlo estimator we prove, in particular, a Central Limit Theorem (CLT) and the Berry-Esseen bounds, see Theorem \ref{thm1541}. This is detailed in Section \ref{secmonte}.

The previously discussed Monte Carlo method is applicable whenever the Feller process can be sampled exactly. This is of course not always possible. In Section \ref{secapproxdiff}, we consider an Euler-Maruyama scheme for the approximation of the trajectories of time-changed It\^o diffusions. In particular, we provide an estimate for the strong error as a function of the time step of the scheme, and as a consequence the strong convergence order is determined. Our method improves classical methods in the literature, e.g. \cite{Kobayashi2011}, as we avoid the error in the scheme introduced by approximation of the paths of the inverse subordinator. 
With this at hand, we then consider a Monte Carlo estimator for functionals of the approximate process: properties, such as a CLT and computational complexity, are then studied. {Finally, in Section \ref{sec:web}, we discuss a possible use of our methods in the context of weak ergodicity breaking for subdiffusions.}

\section{Exact simulation of paths of the inverse, undershoot, and overshoot of a subordinator}\label{secpathinv}
{The goal of this section is to develop algorithms that exactly sample the trajectory (in the sense of finite-dimensional distributions) of an inverse subordinator jointly with its undershoot or the overshoot process. These are provided as Algorithm \ref{alg:2} and Algorithm \ref{alg:3}, and come as a consequence of Markovianity of such processes, proved here in Theorems \ref{markovtransprob} and  \ref{markovtransprob2}, respectively. But first, we begin by formalizing the framework of this paper.}

Let $\sigma = (\sigma_t,\, t\geq0)$ be a subordinator, i.e. an increasing L\'evy process with $\sigma_0=0$, with the Laplace exponent $[0,+\infty)\ni\lambda\mapsto\phi(\lambda)$ given by
\begin{align}
    \phi(\lambda) &= b\lambda + \int_{(0, +\infty)} (1-e^{-\lambda s}) \nu(ds)
    \label{bernstinsimul}\\
    &{=b\lambda + \lambda \int_{(0,+\infty)}e^{-\lambda s}\overline \nu(s)ds}.\label{bernstein-eq0}
\end{align}
{Here $b\ge0$ is called the drift, and $\nu(ds)$ the L\'evy measure of the subordinator $\sigma$, which satisfies $\int_{(0,+\infty)}(1\wedge s)\nu(ds)<+\infty$, and where $\overline \nu(s)=\nu(s,+\infty)$ is the tail of $\nu$.} There is one-to-one correspondence between functions of type \eqref{bernstinsimul} and subordinators, see \cite{Schilling2012}. {We assume that $b>0$ or $\nu(0,+\infty)=+\infty$, which implies that $\sigma$ is strictly increasing, i.e. it is not a compound Poisson process. Note also that $\nu(0,+\infty)=+\infty$ is usually called infinite activity of subordinator and it means that the subordinator has a countable infinity of jumps in any given finite time interval, while if $\nu(0,+\infty)<+\infty$, it has finite number of jumps in any finite time interval.}

We denote the inverse of $\sigma$ by $L= (L_t,\, t\geq0)$, where $L_t = \inf\{ s>0:\, \sigma_s >t \}$, the undershooting of $\sigma$ by $H= (H_t,\, t\geq0)$ where $H_t= \sigma_{L_t-}\coloneqq \lim_{\delta\downarrow 0}\sigma_{L_t-\delta}$, and the overshooting by $D= (D_t,\, t\geq0)$, where $D_t = \sigma_{L_t}$.
Denote by $\gamma_t= t-H_t$ and $\Gamma_t = D_t -t$. In this section, we provide a method of exact simulation of the trajectories $[0, +\infty) \ni t \mapsto (L_t, \gamma_t)$ and $[0, +\infty) \ni t  \mapsto (L_t, \Gamma_t)$, in a finite (but {arbitrarily} big) number of points, when the process $\sigma$ can start at $v^\prime \leq0$. In other words, we can assume any starting point $(L_0, \gamma_0)= (x,v)$ and $(L_0, \Gamma_0)= (x,r)$, with $x \in\R$, $v\geq0$, $r\geq0$. In practice, we furnish a method to exactly sample the vectors
\begin{align}
    ( (L_{t_1}, \gamma_{t_1}), \cdots, (L_{t_n}, \gamma_{t_n}) ),
    \label{vectors}
\end{align}
and
\begin{align}
     ( (L_{t_1}, \Gamma_{t_1}), \cdots, (L_{t_n}, \Gamma_{t_n}) ),
     \label{vectors2}
\end{align}
for any choice of $0 < t_1 < \cdots < t_n$, $n\in\N$, and without assuming $(L_0=0, \gamma_0 =0)$ or $(L_0=0, \Gamma_0=0)$. Since the starting point of the subordinator is $\sigma_0=0$, a.s., one has that $L_0=0$ and $\gamma_0=0$, hence in order to give a meaning to the event $\gamma_0>0$, we perform the following construction of the mentioned processes.

Take the Feller semigroup $T_t$, $t\geq0$, acting on functions $h \in C_0(\R^2)$, given by
\begin{align}
    T_t u (x,v) \, = \, \int_{\R^2} u(x+y,v+w) \delta_t(dy) \mu_t(dw),
\end{align}
where $\delta_t(\cdot)$ denotes the Dirac point mass at $t$, and $\mu_t(dw)$ is the one-dimensional distribution of the subordinator with the Laplace exponent \eqref{bernstinsimul}. In particular, this is a convolution semigroup induced by a L\'evy process, see e.g. \cite[Examples 1.3 and 1.17]{schillinglevy}, denoted by $(A_t, \sigma_t),\, t\geq0$, where $A_t$ is a pure drift, and $\sigma_t$ is an independent subordinator, both not necessarily started at zero. We construct this process as the canonical one in the sense of \cite[Chapter III.1]{revuz2013continuous}, and thus on the filtered probability spaces $\l \Omega, \mathcal{F}_\infty, (\mathcal{F}_t)_{t\geq0}, \pr^{(x,v)} \r$, where $\mathcal{F}_t$ is the natural filtration, $\mathcal{F}_\infty \coloneqq \bigvee_t \mathcal{F}_t$, and $\pr^{(x,v)}$ is the (unique) probability measure such that $\pr^{(x,v)} (A_0=x, \sigma_0 = v)=1$. 

Now, we can define, as above, the processes $L = (L_t,\, t\geq0)$, $H= (H_t,\, t\geq0)$ and $D= (D_t,\, t\geq0)$, paying attention that under $\pr^{(x,0)}$ one has that $L_0 = 0$, $H_0=0$ and $D_0=0$, while under $\pr^{(x,v)}$ for $v\in\R$, one could have different starting points and thus $L_t$ is the inverse of a subordinator started at $\sigma_0 \in\R$ and $H_t$ and $D_t$ are the corresponding undershooting and overshooting.

It turns out that the process $A_{L_t}$ has a certain kind of semi-Markov property. In particular, it turns out that $(A_{L_t}, \gamma_t), t\geq0$, has the simple Markov property and can be inserted in the theory developed in \cite[Section 4]{Meerschaert2014}. For a general discussion of definitions of semi-Markov properties, see \cite[Chapter III]{harlamov}. We can apply the theory of \cite[Section 4]{Meerschaert2014} also to the process $(A_{L_t}, \Gamma_t)$, and it turns out that $({A_{L_t}}, \Gamma_t)$ is a Hunt process, i.e. it is a strong Markov c\`adl\`ag process which is quasi-left-continuous. In {Theorems} \ref{markovtransprob} and \ref{markovtransprob2}, we formalize the previous assertions.

It is clear that under $\pr^{(x,0)}$ the processes $L$, $\gamma$, and $\Gamma$ are just the inverse of a subordinator, the so-called customary age, and the so-called remaining lifetime process, respectively, while $A_{L_t}$ is just the process $x+L_t$. First recall that $L_t$ is the inverse of a strictly increasing subordinator (started at an arbitrary point), so it has continuous trajectories (and so does $A_{L_t}$). It is easy to see that the process $\Gamma_t = D_t-t$ is right-continuous, while the age process $\gamma_t = t-H_t$ is left-continuous.
The discontinuity points of the process $\gamma_t,\, t\geq0$, occur at times in the range of $\sigma$ that are isolated on their left, i.e. at times $t\in \{ \sigma_s,\, s\in \mathcal{J} \}$, where $\mathcal{J}= \{ s\in [0, +\infty):\, \sigma_s-\sigma_{s-}>0 \}$. Hence, $\gamma_t$ has no fixed discontinuities and $  \gamma_{t+} = \gamma_t $ a.s. for any $t\geq0$, provided that $\sigma_0=0$. In practice, this implies that if we consider the right-continuous version
\begin{align}
   (L_t, \gamma_t)^+ \, \coloneqq \, \lim_{\delta\to0+} (L_{t+\delta}, \gamma_{t+\delta}),
\end{align}
then we have $({L_t}, \gamma_t) = (L_t, \gamma_t)^+$, $\pr^{(x,0)}$-a.s., for all fixed $t\geq0$. However, conditioning on $\gamma_0=v$, under $\pr^{(x,y)}$ with $y<0$, there could be a positive probability that $t$ is a discontinuity point of $\gamma$, e.g. if the L\'evy measure $\nu$ has {atoms}. {We will consider here the (customary) undershooting process $H_t=\sigma_{L_t-}$ as it is defined in classical references (e.g. \cite{bertoin1996}) and not its right-continuous version, and thus also the left-continuous process $\gamma_t=t-H_t$. This is helpful as we first write an algorithm for sampling $(L_t, \gamma_t)$, $t \geq 0$, by using the fact that it is a Markov process. The Markov property here comes from the application of a theorem in \cite{Meerschaert2014} that applies to the left-continuous processes; although we are convinced that the generalization of this Markov property to the right-continuous version is not hard, but out of the scope of this paper.}

First, we formalize the (simple) Markov property of the process $(A_{L_t}, \gamma_t)$, $t\geq0$.

\begin{theorem}\label{markovtransprob}
    The process $(A_{L_t}, \gamma_t),\, t\geq0$, is a homogeneous Markov process with the transition probabilities satisfying
    \begin{align}
        &p_{t}^\gamma((x,0);dy,dw) = \pr^{(x,0)} (A_{L_t} \in dy, \gamma_t \in dw), \\
        \begin{split}\label{trans21}
            &p_t^\gamma((x,v); dy,dw) \, = \, \delta_x(dy) \delta_{v+t}(dw) \frac{\nu[v+t, +\infty) }{\nu[v, +\infty)}  + \int_{[v,v+t)} p_{v+t-s}^\gamma((x,0);dy, dw) \, \frac{\nu(ds)}{\nu[v, +\infty)},
        \end{split}
    \end{align}
    for all $x\in \mathbb{R}$ and $v>0$.
\end{theorem}
\begin{proof}
    To prove the result we resort to \cite[Theorem 4.1]{Meerschaert2014}. 
    From the Fourier symbol
    \begin{align}
        \int_\R \int_\R e^{i\xi_1 y+i\xi_2 w} \delta_1(dy) \mu_1(dw) \, = \, e^{i\xi_1-\phi(-i\xi_2)}, 
    \end{align}
    where $\phi$ is the Bernstein function \eqref{bernstinsimul} associated with $\sigma$, and by using \cite[Theorem 31.5]{sato1999}, it is plain that the semigroup $T_t$, $t\geq0$, is generated by $(G, \text{Dom}(G))$ such that
    \begin{align}
        & G\mid_{C_0^2 (\R^2)} h(x,v) 
        = \,  \frac{\partial}{\partial x} h(x,v) + b\frac{\partial}{\partial v} h(x,v)+ \int_{\R^2} (h(x+y,v+w) - h(x,w)) \delta_0(dy)\nu(dw),
    \label{generat}
    \end{align}
    where $\nu(\cdot)$ is the L\'evy measure of the subordinator $\sigma$ (i.e., it is supported on $(0, +\infty)$). It follows that our process is (also) a jump diffusion in the sense of \cite{applebaum} and also in the sense of \cite{Meerschaert2014}, i.e. the generator \eqref{generat} has the form of \cite[Formula (2.5)]{Meerschaert2014} with the jump kernel
    \begin{align}
        K(x,v;dy,dw) = \delta_0(dy) \nu(dw).
    \label{jumpkernel}
    \end{align}
    {We want to apply \cite[Theorem 4.1]{Meerschaert2014} and we should check that the process associated with $T_t$, $t\geq0$, is Markov additive in the sense of \cite{cinlarma}, i.e. that the future depends only on the current state of $A_t$; but in our case this is clear since $(A_t, \sigma_t)$ is a L\'evy process started at $(x,v)$.}
    Recall that the process $(A_t, \sigma_t), t\geq0$, generates the semigroup $T_t$, $t\geq0$, where the first coordinate $A_t$ is a pure drift, and that $L_t$ is the inverse of the second coordinate $\sigma_t$. Thus, $A_{L_t}$ satisfies $A_{L_t}=x+L_t$. The result then follows by an application of \cite[Theorem 4.1]{Meerschaert2014}. Indeed, this theorem says that $(A_{L_t}, \gamma_t)$ is a homogeneous simple Markov process associated with a semigroup of operators $P_t$, $t\geq0$, on the space of bounded Borel functions, such that for $v>0$ it holds that
    \begin{align}
            & P_t h(x,0) \, = \, \E^{(x,0)}[h(A_{L_t}, \gamma_t)], \\
        \begin{split}
            & P_t h(x, v) \, = \, h(x,v+t)\frac{K(x,t;\R, [v+t, +\infty))}{K(x,t; \R, [v, +\infty))}  + \int_\R \int_{[v,{v+t})} P_{v+t-s}h(x+y,0) \frac{K(x,t;ds, dy)}{K(x,t; \R, [v, +\infty))},
        \end{split}
    \end{align}
     from which {the claim follows}. 
\end{proof}
{The transition probabilities in \eqref{trans21} clearly describe the form of the trajectories of the process, where the intervals of constancy are induced by the jumps of the subordinator. Indeed, \eqref{trans21} has the form of a renewal equation. The first term accounts for the transition probabilities when the process remains within an interval of constancy, i.e., the ratio $\nu[v+t,+\infty)/\nu[v,+\infty)$ gives the probability that the process will be stuck for at least $v+t$ given that it has already persisted for $v$; while the second (integral) term instead captures transitions starting from a renewal point — namely, the moment when the process is once again free to evolve.}

Now we formalize the strong Markov property of $(A_{L_t}, \Gamma_t)$, $t\geq0$.

\begin{theorem}\label{markovtransprob2}
     The process $(A_{L_t}, \Gamma_t)$, $t\geq0$, is a Hunt process with the transition probabilities satisfying
     \begin{align}
        & p_t^\Gamma ((x,0);dy, dw) \, = \, \pr^{(x,0)} (A_{L_t} \in dt, \Gamma_t \in dw) \notag, \\
        & p_t^\Gamma ((x,r);dy, dw) \, = \, \1_{[r>t]}\delta_x(dy)\delta_{r-t}(dw) + \1_{[r \leq t]} p_{t-r}^\Gamma ((x,0);dy,dw),
        \label{trans1129}
     \end{align}
     for all $x\in\R$ and $r>0$.
\end{theorem}
\begin{proof}
    The proof follows similarly as the one of Theorem \ref{markovtransprob}. Indeed, \cite[Theorem 3.1 \& Theorem 4.1]{Meerschaert2014} imply that $(A_{L_t}, \Gamma_t)$ is a Hunt process associated with a semigroup of operators $Q_t$, $t\geq0$, on the space of bounded Borel measurable functions such that
    \begin{align}
        &Q_t h(x,0) \, = \, \E^{(x,0)} [h(A_{L_t}, \Gamma_t)],\label{semrepr0}\\
        &Q_t h(x,r) \, = \, \1_{[t <r]} h(x,r-t) + \1_{[t\geq r]}  Q_{t-r}h(x,0),
    \label{semrepr}
    \end{align}
    for $x\in\R$ and $r>0$. It follows from \eqref{semrepr0} and \eqref{semrepr} that
     \begin{align}
         & p_t^\Gamma ((x,0); dy, dw) \, = \, \pr^{(x,0)} (A_{L_t} \in dy, \Gamma_t \in dw), \\
         &p_t^\Gamma ((x,r);dy, dw) \, = \, \1_{[r>t]}\delta_x(dy)\delta_{r-t}(dw) + \1_{[r\leq t]} p_{t-r}^\Gamma ((x,0);dy,dw),
     \end{align}
    for $x\in\R$ and $r>0$, which concludes the proof.
\end{proof}
{The transition probabilities \eqref{trans1129} have a clear heuristic interpretation. The first term represents the transitions when the remaining time in the current position is $r$, i.e., the process at time $t$ is still stuck in the same position. The second term represents the transitions from the next renewal point, i.e., after time $r$ when the process is left free to move.}

From Theorems \ref{markovtransprob} and \ref{markovtransprob2} it is clear that we can use the simple Markov property to obtain our algorithm for the exact sampling of paths of $(L,\gamma)$ and $(L,\Gamma)$. Indeed, denote by
\begin{align}
    &  p_{t_1, \cdots, t_n}^\gamma((x,v);B_1, \cdots, B_n), \\
    & p_{t_1, \cdots, t_n}^\Gamma((x,v);B_1, \cdots, B_n),
\end{align}
the distribution of vectors \eqref{vectors} and \eqref{vectors2} for $B_i \in \mathcal{B} (\R^2)$, conditioning on $L_0=x_0$ and $\gamma_0=v_0$.
Then, by Theorem \ref{markovtransprob} we have that, for $ y_0=x_0$, $w_0=v_0$,
\begin{align}
    p_{t_1, \cdots, t_n}^\gamma ((x_0,v_0);B_1, \cdots, B_n) \, = \,  \int_{B_1 \times \cdots \times B_n} \prod_{i=1}^n p_{t_i-t_{i-1}}^\gamma ((y_{i-1}, w_{i-1}); dy_i, dw_i),
    \label{jointvector}
\end{align}
and by Theorem \ref{markovtransprob2} we get, for $y_0=x_0$, $w_0=v_0$,
\begin{align}
    p_{t_1, \cdots, t_n}^\Gamma ((x_0,v_0);B_1, \cdots, B_n) \, = \,  \int_{B_1 \times \cdots \times B_n} \prod_{i=1}^n p_{t_i-t_{i-1}}^\Gamma ((y_{i-1}, w_{i-1}); dy_i, dw_i).
    \label{jointvector2}
\end{align}
Here we show that we can sample from the transition probabilities appearing in \eqref{jointvector} and \eqref{jointvector2}. We do this by showing that there exists r.v.'s whose distributions under $\pr^{(x,0)}$ are given by such probabilities, with an appropriate choice of parameters.
\begin{proposition}\label{propoX}
    Consider a family of random variables {$\{\mathcal{S}_v:\, v>0\}$ independent from $\sigma$ such that
    \begin{align}
        \pr^{(x,0)}(\mathcal{S}_v\in ds)\coloneqq \frac{\nu(ds)}{\nu[v, +\infty)}\1_{[s\geq v]},\quad v>0,\label{densityS}
    \end{align}
    }and a random vector
    \begin{align}
        \mathcal{X}_{v,x, t}^s\coloneqq (x, v+t) \1_{[s \geq v+t]}  + (x+ L_{v+t-s}, \gamma_{v+t-s}) \1_{[s < v+t]}
    \end{align}
    with the parameters $x\in\R$, $t>0$, $v>0$, $s>0$. Let $p_t^\gamma ((x,v);dy,dw)$ be the transition probabilities defined in Theorem \ref{markovtransprob}. Then it holds that 
    \begin{align}
        p_{t}^\gamma((x,v);dy,dw) \, = \, \pr^{(x,0)} \l \mathcal{X}^{\mathcal{S}_v}_{v,x,t} \in (dy, dw) \r.
    \end{align}
    Furthermore, let
    \begin{align}
        \mathcal{Y}_{r,x,t} \coloneqq (x,r-t) \1_{[r>t]} + \1_{[r \leq t]} (x+L_{t-r}, \Gamma_{t-r}), 
    \end{align}
    then
    \begin{align}
        p_t^\Gamma ((x,r); dy, dw) \, = \, \pr^{(x,0)} ( \mathcal{Y}_{r,x,t} \in (dy, dw) ).
    \end{align}
\end{proposition}

\begin{proof}
    Under $\pr^{(x,0)}$, the  process $\sigma$ is a subordinator with $\sigma_0=0$ a.s., and $L_t$ is the inverse of $\sigma$, so $L_0=0$ and $\gamma_0=0$, a.s. It follows that
    \begin{align}
        \pr^{(x,0)} & ( \mathcal{X}_{x,v,t}^{\mathcal{S}_v} \in (dy, dw) ) \, = \, \delta_x(dy) \delta_{v+t}(dw) \pr^{(x,0)} (\mathcal{S}_v \geq v+t) \notag\\
        & + \int_{[v,v+t)} \pr^{(x,0)} ( (x+L_{v+t-s}, \gamma_{v+t-s})\in(dy, dw) \mid \mathcal{S}_v=s ) \pr^{(x,0)} (\mathcal{S}_v \in ds) \notag\\
        = \, & \delta_x(dy) \delta_{v+t}(dw) \frac{\nu[v+t, +\infty)}{\nu[v, +\infty)} + \int_{[{v},{v+t})} p_{v+t-s}((x,0); dy, dw) \frac{\nu(ds)}{\nu[v, +\infty)},
    \label{315}
    \end{align}
    where we used the independence between $\mathcal{S}_v$ and $\sigma$, the form of the distribution of $\mathcal{S}_v$, and the fact that $\gamma_0=0$ as explained above. The result follows by observing that \eqref{315} coincides with the transition probabilities obtained in Theorem \ref{markovtransprob}.
    
    The second result follows immediately by Theorem \ref{markovtransprob2}. 
\end{proof}

The previous proposition, together with a simple conditioning argument, suggest the following result.

\begin{proposition}\label{propojoint}
    Let $(\Omega, \mathcal{G}, P)$ be an arbitrary probability space supporting $\mathcal{S}_{v}^i$, $^i\mathcal{X}_{v_i,x_i,t_i}^{s_i}$ and $^i\mathcal{Y}_{r_i,x_i,t_i}$, $i=1, \cdots, n$, that are independent copies of the random variables introduced in Proposition \ref{propoX}. 
    It is true that
    \begin{align}
          p_{t_1, \cdots, t_n}^\gamma ((x_0,v_0);B_1, \cdots, B_n) \, = \,  \int_{B_1 \times \cdots \times B_n} \prod_{i=1}^n f_{t_i-t_{i-1}}^\gamma ((y_{i-1}, w_{i-1}); dy_i, dw_i)
    \end{align}
    where $y_0=x_0$, $w_0=v_0$, and, for $i \geq 2$,
    \begin{align}
       & f_{t_i-t_{i-1}}^\gamma ((y_{i-1}, w_{i-1}); dy_i, dw_i) \notag\\
       = \,& P \l \, ^i\mathcal{X}^{\mathcal{S}^i_{w_{i-1}}}_{w_{i-1}, y_{i-1}, t_i-t_{i-1}} \in (dy_i, dw_i) \mid \, ^{i-1}\mathcal{X}^{\mathcal{S}^{i-1}_{w_{i-2}}}_{w_{i-2}, y_{i-2}, t_{i-1}-t_{i-2}} = (y_{i-1}, w_{i-1})  \r \notag \\
        = \, & P \l \, ^i\mathcal{X}^{\mathcal{S}^i_{w_{i-1}}}_{w_{i-1}, y_{i-1}, t_i-t_{i-1}} \in (dy_i, dw_i) \r.
    \end{align}
    Further, it is true that
    \begin{align}
          p_{t_1, \cdots, t_n}^\Gamma ((x_0,v_0);B_1, \cdots, B_n) \, = \,  \int_{B_1 \times \cdots \times B_n} \prod_{i=1}^n f_{t_i-t_{i-1}}^\Gamma ((y_{i-1}, w_{i-1}); dy_i, dw_i),
    \end{align}
    where $y_0=x_0$, $w_0=v_0$, and, for $i \geq 2$,
    \begin{align}
       & f_{t_i-t_{i-1}}^\Gamma ((y_{i-1}, w_{i-1}); dy_i, dw_i) \notag \\ = \,& P \l \, ^i\mathcal{Y}_{w_{i-1}, y_{i-1}, t_i-t_{i-1}} \in (dy_i, dw_i) \mid \, ^{i-1}\mathcal{Y}_{r_{i-2}, x_{i-2}, t_{i-1}-t_{i-2}} = (y_{i-1}, w_{i-1})  \r\notag \\
       = \, &P \l \, ^i\mathcal{Y}_{w_{i-1}, y_{i-1}, t_i-t_{i-1}} \in (dy_i, dw_i)  \r.
    \end{align}
\end{proposition}

\begin{proof}
    Since the probability space here is inessential, the proof follows by a simple conditioning argument and Proposition \ref{propoX}.
\end{proof}

The random variables $^i\mathcal{X}_{x,v,t}^s$ and $^i\mathcal{Y}_{x,v,t}$ appearing in Propositions \ref{propoX} and \ref{propojoint} are easily exactly simulated provided that $\sigma$ falls in the class of subordinators for which we can exactly sample {r.v.'s $\mathcal{S}_v$, and the couple $(L_t, H_t)$ or the couple $(L_t,D_t)$ for arbitrary $t>0$, conditionally on $\sigma_0=0$. In Section \ref{sec:prelim}, we discuss in detail the algorithms which indeed provide such samples. At this point, we briefly note that \cite[Algorithm 1]{Jorge2023b} gives an exact sample of the triplet $(L_t,H_t,D_t)$, under $\sigma_0=0$, for a fairly large class of subordinators. These observations inspire} the Algorithms \ref{alg:2} and \ref{alg:3} which give a method to exactly sample the random vectors
\begin{align}
   & ( (A_{L_{t_1}}, \gamma_{t_1} ), \cdots, (A_{L_{t_n}}, \gamma_{t_n}) ), \label{vectorsfin} \\
   & ( (A_{L_{t_1}}, \Gamma_{t_1} ), \cdots, (A_{L_{t_n}}, \Gamma_{t_n}) ), \label{vectors2fin}
\end{align}
for any choice of $0 < t_1 < \cdots < t_n$, $n \in \mathbb{N}$. Note again that, since $A_t$ is pure drift started at $x$ under $\pr^{(x,v)}$, the processes $A_{L_t}$ is just the inverse $L_t$ of the second coordinate $\sigma_t$, translated by $x$. {It is also important to emphasize that instead of e.g. \cite[Algorithm 1]{Jorge2023b} which can be used in Algorithms \ref{alg:2} and \ref{alg:3}, we can use any algorithm that provides exact samples of $(L_t, H_t)$ or  $(L_t,D_t)$, respectively.} {In particular, since \cite[Algorithm 1]{Jorge2023b} gives a sample of the first passage time, the undershooting and the overshooting to a quite general curve, we will develop an easier algorithm to sample $(L_t, D_t)$ alone (without the coordinate $H_t)$ by mixing techniques from \cite{Jorge2023b} and \cite{Dassios2020}, see Algorithm \ref{alg:simplif} as well as Algorithm \ref{alg:simplif_2}.}

\begin{algorithm}[h]
    \caption{generating vectors \eqref{vectorsfin} conditionally on $A_{L_{t_0}}\in\R$, $\gamma_{t_0}\ge 0$}\label{alg:2}
    \KwData{
    $A_{L_{t_0}}\in\R$, $\gamma_{t_0}\geq0$, and $t_0 <t_1 <\cdots <t_{n-1} <t_n$, $n\in\N$.
    }
    \For{$i\leftarrow 1$ \KwTo $n$}{
        $(x,v) \gets (A_{L_{t_{i-1}}}, \gamma_{t_{i-1}})$\\
		\eIf{$v=0$}{
                Generate $(L_{t_i -t_{i-1}}^i, \gamma_{t_i-t_{i-1}}^i)$ from Algorithm e.g. \ref{alg:simplif_2}\\
			$(A_{L_{t_i}}, \gamma_{t_i}) \gets (x+ L_{t_i - t_{i-1}}^i, \gamma_{t_i-t_{i-1}}^i)$
		}{
			Generate $\mathcal{S}^{i}_v \sim \frac{\nu(\cdot)}{\nu[v,+\infty)}\1_{\cdot > v}$\\
			\eIf{$\mathcal{S}^{i}_v \geq v + t_{i}-t_{i-1} $}{
				$(A_{L_{t_i}}, \gamma_{t_i}) \gets (x, v + t_{i}-t_{i-1})$
			}{
				Generate $( L_{t_i -t_{i-1} +v -\mathcal{S}^{i}_v}^i, \gamma_{t_i -t_{i-1} +v -\mathcal{S}^{i}_v}^i )$ from e.g. Algorithm \ref{alg:simplif_2}\\
				$(A_{L_{t_i}}, \gamma_{t_i}) \gets ( x + L_{t_i -t_{i-1} +v -\mathcal{S}^{i}_v}^i, \gamma_{t_i -t_{i-1} +v -\mathcal{S}^{i}_v}^i )$
			}
		}
	}
\end{algorithm}

\begin{algorithm}[ht]
    \caption{generating vectors \eqref{vectors2fin} conditionally on $A_{L_{t_0}}\in\R$, $\Gamma_{t_0}\geq0$}\label{alg:3}
	\KwData{
    $A_{L_{t_0}}\in\R$, $\Gamma_{t_0}\geq0$, and $t_0 <t_1 <\cdots <t_{n-1} <t_n$, $n\in\N$.
	}
    \For{$i\leftarrow 1$ \KwTo $n$}{
		$(x,r) \gets ( A_{L_{t_{i-1}}}, \Gamma_{t_{i-1}} )$ \\
		\eIf{$r > t_{i}-t_{i-1}$}{
			$(A_{L_{t_i}}, \Gamma_{t_i}) \gets ( x, r -(t_i - t_{i-1}) )$
			}{
			Generate $( L_{t_i -t_{i-1} -r}^i , \Gamma_{t_i -t_{i-1} -r}^i )$ from e.g. Algorithm \ref{alg:simplif} or Algorithm \ref{alg:simplif_2}\\
			$(A_{L_{t_i}}, \Gamma_{t_i}) \gets ( x + L_{t_i -t_{i-1} -r}^i, \Gamma_{t_i -t_{i-1} -r}^i )$
			}
	}
\end{algorithm}

\begin{remark}
    {Algorithm \ref{alg:3} can be generalized to sample the finite-dimensional distributions of the triplet $(\gamma_t, A_{L_t}, \Gamma_t)$ even if we do not prove that this process is Markovian. Indeed, subordinators enjoy the regenerative property (\cite[Chapter 2]{bertoin1999}), and this permits to sample from the conditional probabilities
    \begin{align}
        \mathds{P}^{(x,0)} \l \gamma_{t_2} \in dx_1, A_{L_{t_2}} \in dy_2, \Gamma_{t_2} \in dz_2 \mid \gamma_{t_1} = x_1, A_{L_{t_1}} = y_1, \Gamma_{t_1} = z_1 \r,
    \end{align}
    where $t_1<t_2$, as we show below.
    
    The regenerative property states that the shifted range $\overline{\{ \sigma_{L_t+y}-\sigma_{L_t}: y \geq 0 \}}$ is independent of $\mathcal{F}_{L_t}$, where $(\mathcal{F}_s, \,s \geq 0)$ denotes the natural filtration of $\sigma$, and identically distributed as the original range. Applying this, under $\mathds{P}^{(x,0)}$ one has that
    \begin{align*}
         \mathds{P}^{(x,0)}& \l \gamma_{t_2} \in dx_1, A_{L_{t_2}} \in dy_2, \Gamma_{t_2} \in dz_2 \mid \gamma_{t_1} , A_{L_{t_1}} , \Gamma_{t_1}  \r  \\ = \,& \mathds{E}^{(x,0)} \left[ \mathds{1}_{[\sigma_{L_{t_1}}\ge   t_2]}\mathds{1}_{[\gamma_{t_2} \in dx_2, A_{L_{t_2}} \in dy_2, \Gamma_{t_2} \in dz_2]}\mid \gamma_{t_1} , A_{L_{t_1}} , \Gamma_{t_1}\right]\\& \quad + \mathds{E}^{(x,0)} \left[\mathds{1}_{[\sigma_{L_{t_1}} < t_2]} \mathds{1}_{[\gamma_{t_2} \in dx_2, A_{L_{t_2}} \in dy_2, \Gamma_{t_2} \in dz_2]} \mid \gamma_{t_1} , A_{L_{t_1}} , \Gamma_{t_1}\right]  \\
      = \, &  \mathds{1}_{[\Gamma_{t_1}\ge  t_2-t_1]} \delta_{\gamma_{t_1}+t_2-t_1}(dx_2) \delta_{A_{L_{t_1}}}(dy_2) 
      \delta_{\Gamma_{t_1}-(t_2-t_1)}(dz_2) \\ &\quad+ \mathds{1}_{[\Gamma_{t_1}< t_2-t_1]} \mathds{P}^{(A_{L_{t_1}},0)} \l \gamma_{t_2-t_1- \Gamma_{t_1}} \in dx_2, A_{L_{t_2-t_1- \Gamma_{t_1}}} \in dy_2, \Gamma_{t_2-t_1- \Gamma_{t_1}} \in dz_2  \r, 
    \end{align*}
    and the latter can be sampled by Algorithm \ref{alg:simplif_2} or \cite[Algorithm 1]{Jorge2023b}. This makes possible to write an algorithm for the finite-dimensional distribution in the same spirit as Algorithm \ref{alg:3}. The Algorithms \ref{alg:2} and \ref{alg:3} require, however, to sample only from the vectors $(L_t, \gamma_t)$ or $(L_t, \Gamma_t)$ and thus they have (potentially) less complexity of the next Algorithm \ref{alg:triplet} that we present for the vectors
    \begin{align}
       ((\gamma_{t_0}, A_{L_{t_0}}, \Gamma_{L_{t_0}}), \cdots, (\gamma_{t_n}, A_{L_{t_n}}, \Gamma_{L_{t_n}})),
        \label{vectorstrip}
    \end{align}
    under $\mathds{P}^{(x,0)}$ and for any $t_0<t_1< \cdots< t_n$, $n \in \mathbb{N}$.
    On the other hand, using Algorithm \ref{alg:simplif_2} or \cite[Algorithm 1]{Jorge2023b} in Algorithm \ref{alg:3} and Algorithm \ref{alg:triplet}, makes Algorithm \ref{alg:triplet} equivalent in terms of complexity to Algorithm \ref{alg:3}.
\begin{algorithm}[ht]
    \caption{generating vectors \eqref{vectorstrip} conditionally on $A_{L_{t_0}}\in\R$, $\Gamma_{t_0}\geq0$}\label{alg:triplet}
	\KwData{
    $\gamma_{t_0}=0,A_{L_{t_0}}=x\in\R$, $\Gamma_{t_0}=0$ and $t_0 <t_1 <\cdots <t_{n-1} <t_n$, $n\in\N$.
	}
    \For{$i\leftarrow 1$ \KwTo $n$}{
		$(r,x,R) \gets (\gamma_{t_{i-1}} ,A_{L_{t_{i-1}}}, \Gamma_{t_{i-1}} )$ \\
		\eIf{$R > t_{i}-t_{i-1}$}{
			$(\gamma_{t_{i}},A_{L_{t_i}}, \Gamma_{t_i}) \gets ( r+t_i-t_{i-1},x, R -(t_i - t_{i-1}) )$
			}{
			Generate $( \gamma_{{t_i-t_{i-1}-R}}^i, L_{t_i -t_{i-1} -R}^i , \Gamma_{t_i -t_{i-1} -R}^i )$ from e.g. Algorithm \ref{alg:simplif_2}\\
			$(\gamma_{L_{t_i}},A_{L_{t_i}}, \Gamma_{t_i}) \gets ( \gamma_{{t_i-t_{i-1}-R}}^i, x + L_{t_i -t_{i-1} -R}^i, \Gamma_{t_i -t_{i-1} -R}^i )$
			}
	}
\end{algorithm}}
\end{remark}

\begin{remark}
    It is clear that if one only wants to sample the trajectory of $t\mapsto L_t$, it is more convenient to use Algorithm \ref{alg:3} as it does not require sampling the random variables $\mathcal{S}_v$. We remark also that the idea of using the overshooting to make a process Markovian and simulate its trajectories has been already used, in a different context, i.e., for a discrete time process, in \cite{aleksunder}.
\end{remark}
\begin{remark}[On sampling paths of {general} tempered subordinators]\label{rem:stable-gen}
    Let $\nu(dx)=\nu(x)dx$ be an absolutely continuous L\'evy measure and $\nu_q(dx)$ its tempered version, i.e. $\nu_q(dx)=\nu_q(x)dx={q(x)}\nu(x)dx$, {where $q(x)$ is a function $0<q(x)<1$}. If samples of the random variable $\mathcal{S}_v$ with the distribution as in \eqref{densityS} are available, then samples of the tempered version, denoted here by $\mathcal{S}_{v,q}$, are also available through the rejection sampling method.

    Recall that the rejection sampling method consists of sampling random variables from a target distribution $\mathcal{S}_{v,q}$ with a probability density function $g(\cdot)$ using a proposal distribution $\mathcal{S}_{v}$ with a probability density function $f(\cdot)$. The idea is to generate random variables from the distribution $\mathcal{S}_{v}$ and accept them with probability $g(\cdot)/Mf(\cdot)$. The constant $M$ represents a finite value that serves as a bound of the likelihood ratio $g(\cdot)/f(\cdot)$ over the support of $\mathcal{S}_{v}$, see \citep[p. 39]{Asmussen2007}. In our case, $M$ is
    \begin{align}
        \frac{g(s)}{f(s)} = \frac{\nu_q(s)/\nu_q[v,+\infty)}{\nu(s)/\nu[v,+\infty)} = \frac{{q(s)}\nu[v,+\infty)}{\nu_q[v,+\infty)} \leq \frac{\nu[v,+\infty)}{\nu_q[v,+\infty)}\eqqcolon M.
    \end{align}
    In other words, to obtain samples of $\mathcal{S}_{v,q}$, first we generate samples $s$ and $u$ from  $\mathcal{S}_v$ and $U(0,1)$, respectively, and accept the value $s$ if $u<g(s)/Mf(s)={q(s)}$.
	
    An example is the tempered stable L\'evy measure. In this case,
    \begin{align}
        \nu(s) = \frac{\alpha}{\Gamma(1-\alpha)}s^{-\alpha-1}, \quad \nu_q(s) = e^{-qs}\nu(s), \quad\text{and}\quad \mathcal{S}_v \sim \alpha v^\alpha s^{-\alpha-1},
    \end{align}
    where $\alpha\in(0,1)$, $q>0$, and $\Gamma$ represents the Gamma function. Furthermore, the cumulative distribution function of $\mathcal{S}_v$, denoted here by $F(s)$, and its inverse are available. Indeed, we have $F^{-1}(s)=v(1-s)^{-1/\alpha}$. Therefore, samples of random variables with distribution $\mathcal{S}_v$ can be obtained through the inverse transform sampling method \citep[p. 37]{Asmussen2007}. This allows us to sample the paths of the $\alpha$-stable inverse subordinator along with the undershoot and overshoot processes and their respective tempered versions using Algorithms \ref{alg:2} and \ref{alg:3}.
\end{remark}

\section{Exact sampling and computational complexity}\label{sec:prelim}
{In this section, we start by explaining what we mean by an exact simulation, and then proceed by calculating the computational complexity of our Algorithms \ref{alg:2} and \ref{alg:3}. In the latter part, we focus on a general algorithms for sampling $(L_t, H_t)$ or $(L_t,D_t)$ for fixed time $t>0$ and under $\sigma_0=0$, which are independent of the current state of the art, as well as on recent methods such as those developed in \cite{Dassios2020,Jorge2023b, Jorge2023a}. {Moreover, in Subsection \ref{sec:compcomp} we give two new simple but fast Algorithms \ref{alg:simplif} and \ref{alg:simplif_2} for sampling $(L_t,D_t)$ and $(L_t, H_t)$, for fixed time $t>0$ and under $\sigma_0=0$, obtained by mixing techniques of  \cite{Dassios2020} with the ones of \cite{Jorge2023b, Jorge2023a}, and we finish by comparing actual computational times of our algorithms to the known ones.}}

For the notion of an exact simulation, we suggest the instructive discussion in \cite[Appendix D]{Jorge2023a} to which we refer to the notion of exact (or practically exact) simulation adopted in this paper. For the reader's convenience, here we summarize the main parts.

Since computers dispose of finite resources (bits of precision) the law of the output of any algorithm will differ from the target law. The more intuitive notion of exact simulation, useful in practice, is the following: A simulation is exact if the distance (associated with a suitable metric) between the desired and simulated law is controlled with machine precision, i.e. with $n\in\N$ bits of precision. In other words, if $I$ is the ideal simulated quantity and $A$ denotes the output of the simulated law, determining if the simulation is exact reduces to check a suitable distance $d( \mathcal{L}(I), \mathcal{L}(A) )$, where $\mathcal{L} (\cdot)$ denotes the law of the corresponding quantity. The choice of the metric $d (\cdot, \cdot)$ is clearly crucial in this context. Several distances might be taken into account but the obvious one is the Wasserstein distance which, however, shows its limitations even in the context of a rejection sampling, see, in particular, \cite[Item I, Appendix D]{Jorge2023a}. Here, we will consider the L\'evy-Prokhorov metric. The latter metricizes weak convergence and in this sense it is a very natural choice; its definition is as follows. For two given probability measures $\mu$ and $\nu$ on a metric space $(\mathbb{S},d)$, define the distance $d_P(\mu,\nu)\coloneqq\inf\{ \varepsilon>0:\, \mu(F)\leq\nu(F^\varepsilon)+\varepsilon\quad\text{for all closed}\quad F\subset\mathbb{S} \}$, where $F^\varepsilon\coloneqq\{ x:\, d(x,F)<\varepsilon \}$, which is called the L\'evy-Prokhorov distance between $\mu$ and $\nu$. Let $\{\pr_\lambda:\lambda\in\Lambda\}$ be the family of probability measures on $\mathbb{S}\times \mathbb{S}$ such that each $\pr_\lambda$ has marginals $\mathcal{L}(I)$ and $\mathcal{L}(A)$, respectively. Then, it is true that  (see \citep{Dudley1968} for details)
\begin{align}
    d_P( \mathcal{L}(A),\mathcal{L}(I) ) \, = \, \inf\{  \varepsilon>0:\, \text{ there exists $\lambda$ with } \pr_\lambda(d(I,A) > \varepsilon) \leq \varepsilon \},
\end{align}
i.e. the distance is the smallest $\varepsilon$ such that the distance between $I$ and $A$ is no more than $\varepsilon$ under a suitable law $\pr_\lambda$, with probability at least $1-\varepsilon$.
Then, in the case of an algorithm that employs $k_i$ times numerical methods of the $i$-th kind, each one controlling the error in the quantity $\varepsilon_0$, the L\'evy-Prokhorov distance is bounded by $\max\{ k_1,\dots,k_i\}\varepsilon_0$. In this paper we adopt this notion of exact simulation (or, practically exact according to \cite{Jorge2023a}): we say that an algorithm is exact if the above distance between the output and the ideal variable can be controlled with the machine precision, i.e. if $d_P( \mathcal{L}(I), \mathcal{L}(A) ) \leq C 2^{-N}$, where $C>0$ is a constant (that can be determined explicitly) and $N$ is the number of bits of precision.

{As we have seen, our Algorithms \ref{alg:2} and \ref{alg:3} do not use approximations}, but their implementation might require to call others that do so, {since} our algorithms require to sample from (the single time) distribution of $(L_t, H_t)$ and $(L_t, D_t)$ under $\sigma_0=0$.  In \cite{Jorge2023a}, the authors provided a method to {sample the triplet $(L_t, H_t,D_t)$} in the case when $\sigma_t$ is a tempered stable subordinator, i.e. its L\'evy density is
\begin{align}
    \nu(s) \, = \, e^{-qs}\nu_\alpha(s),    
\end{align}
for some $q\geq0$ and $\alpha\in(0,1)$, where  $\nu_\alpha(s) \, = \, \frac{\alpha}{\Gamma (1-\alpha)}s^{-\alpha -1}$ is the L\'evy density of the $\alpha$-stable subordinator. 

The method from \cite{Jorge2023a}, called the TSFFP-Algorithm, returns exact samples of the first passage event (i.e. the first-passage time, undershoot, and overshoot) across an absolutely continuous non-increasing function, called the barrier. Hence, it can be used to get an exact sample of the vector $(L_t, H_t, D_t)$ at fixed time $t>0$. The provided method is exact in the sense above, despite employing numerical inversion and integration, see Appendix D.1 in \citep{Jorge2023a} where the authors prove exactness in the L\'evy-Prokhorov metric. 

{The method from  \cite{Jorge2023a} is then generalized to a larger class of subordinators by the same authors in \citep[{Algorithm 1}]{Jorge2023b}.} The available class of subordinators is such that their L\'evy density $\nu$ satisfies $\nu(t)t^{\alpha +1}=1+\mathcal{O}(t)$ as $t\downarrow 0$, where $\alpha\in(0,1)$. Here $f(t)=\mathcal{O}(g(t))$ as $t\downarrow 0$ stands for $\lim\sup_{t\downarrow 0} f(t)/g(t)<+\infty$, where $f$ and $g$ are positive functions. In such cases, the L\'evy density $\nu$ can be written as
\begin{align}
    \nu(t) = \theta e^{-qt}t^{-\alpha -1}\1_{\{ t\leq r\} } + \zeta(t),\quad t>0,
    \label{train}
\end{align}
for some constants $q\geq0$, $\theta \in(0,\infty)$, $r\in(0,\infty]$, and $\zeta$ a density of a finite measure on $(0,+\infty)$. For further details see \citep[Appendix A]{Jorge2023b}. Here we also mention that $\nu(t)t^{\alpha +1}=1+\mathcal{O}(t)$ can be changed to the condition $\nu(t)t^{\alpha +1}=c+\mathcal{O}(t)$, for any constant $c>0$. Indeed, the multiplication of the L\'evy measure by a constant only changes the time scale of the subordinator. At the end of {this} section, one can find examples of such subordinators as well as the simulations of their trajectories.
To the best of our knowledge, the two algorithms mentioned above are the only {peer-reviewed} ones that return exact (in the sense above) samples of the first passage event of a subordinator. {In fact, \cite{Chi2016} also provides the algorithm for the simulation of the triplet but the algorithm has infinite expected running time, see \cite[Section 2.7]{Jorge2023a}. We also mention the recent preprint \cite{chi2025complexity} which provides an exact method to sample the first passage events of a stable subordinator across a barrier, but without using numerical inversion or integration.}

{In the case of sampling only the couple $(L_t,D_t)$ which we may use in Algorithm \ref{alg:3}, we mention the paper \cite{Dassios2020} which gives exact samples of $(L_t,D_t)$ under the assumption that the L\'evy measure of $\sigma$ is truncated stable subordinator. In the following subsection, we will build on this result and give an algorithm which samples $(L_t,D_t)$ where the L\'evy measure can be supported on the whole $(0,+\infty)$.}

The implementation of {Algorithms \ref{alg:2} and \ref{alg:3} in principle requires} running the algorithms above a finite (explicit) number of times and adding some elementary arithmetic operations or rejection sampling. Therefore, the error propagation is easily controlled as, at the end of our algorithm (end of the trajectory), the error is bounded by the error made in {e.g. \cite[TSFFP-algorithm]{Jorge2023a} or \cite[Algorithm 1]{Jorge2023b}}, which can be controlled with machine precision times a finite number (that can be determined by the number of iterations and elementary operations).

To be precise, the L\'evy-Prokhorov distance between the desired and simulated random variable is bounded by the maximum number of iterations between the numerical methods employed, multiplied by the desired error tolerance. {E.g. the closed form of the bound for such distance in the case of TSFFP-Algorithm can be found in \cite[Appendix D.1]{Jorge2023a}.} 

\subsection{Sampling $L_t$, $\gamma_t$ and $\Gamma_t$ at single times, and computational complexity}\label{sec:compcomp}

\subsubsection{Parsimonious algorithms for $(A_{L_t}, \Gamma_{t})$ and $(A_{L_t}, \gamma_{t})$}

In order to run Algorithms \ref{alg:2} and \ref{alg:3}, we only need to sample the couple $(A_{L_t}, \Gamma_t)$ or $(A_{L_t}, \gamma_t, \Gamma_t)$ for a fixed time and under $\sigma_0=0$, so using e.g. \cite[Algorithm 1]{Jorge2023b} to sample these marginals can be considered too general. Indeed, this algorithm samples the first passage time jointly with the undershooting and the overshooting through a quite general curve, while we only need it for a fixed boundary, and in e.g. Algorithm \ref{alg:3} we do not need the undershooting. Hence, we provide here a couple of easier and efficient algorithms to sample $(A_{L_t}, \Gamma_t)$ and $(A_{L_t}, \gamma_t)$, obtained by mixing \cite[Algorithm 4.1]{Dassios2020} with the techniques from \cite{Jorge2023a} and \cite{Jorge2023b}.

In the following, we assume that the subordinator $\sigma$ is a driftless subordinator with the L\'evy measure $\nu$ (as in \eqref{train}) decomposed as
\begin{align}\label{decomp}
    \nu(ds)=\nu^{\alpha,\theta, q}_{r}(s)ds+\zeta(ds)\coloneqq\theta \frac{\alpha e^{-qs}}{\Gamma(1-\alpha)}s^{-\alpha-1}\1_{\{0<s\le r\}}ds+\zeta(ds),\quad s>0,
\end{align}
where $r\in(0,+\infty]$, $\zeta$ is a finite measure on $(0,+\infty)$, and $\alpha\in (0,1)$, $\theta>0$ and $q\in[0,+\infty)$. This restriction arises from the fact that \cite[Algorithm 4.1]{Dassios2020} enables the joint sampling of the inverse and the overshoot of a truncated stable subordinator, whereas the approach developed in \cite{Jorge2023b, Jorge2023a} allow us to extend this method to incorporate tempering. Therefore, we discuss two algorithms separately: the first one is based on \cite{Dassios2020} and applied to \eqref{decomp} for $q=0$ to sample $(A_{L_t}, \Gamma_t)$, while the second one is a simplified version of \cite[Algorithm 1]{Jorge2023b} and permits to include $q>0$ in \eqref{decomp}. We remark that adding a positive drift $b>0$ is possible in both cases with the techniques of \cite{Jorge2023b} (since it is just a boundary translation). For the sake of simplicity we do not deal with this parameter in the simplified Algorithms discussed in this section.

In the following, it is useful to decompose $\nu$ again but with the truncation level exactly $t\in (0,r]$, i.e.:
\begin{align}
    \nu(ds)&=\theta \frac{\alpha e^{-qs}}{\Gamma(1-\alpha)}s^{-\alpha-1}\1_{\{0<s\le t\}}ds+\big(\theta \frac{\alpha e^{-qs}}{\Gamma(1-\alpha)}s^{-\alpha-1}\1_{\{t<s\le r\}}ds+\zeta(ds)\big)\notag\\
    &\eqqcolon \nu^{\alpha,\theta,q}_{t}(s)ds+\widetilde\zeta_t(ds).\label{decomp2}
\end{align}

Denote by $\sigma^\zeta$ a compound Poisson subordinator with the L\'evy measure $\widetilde\zeta_t(ds)$ and by $\sigma^{\alpha,q}$ an (independent) truncated stable subordinator with the L\'evy density $\nu^{\alpha,\theta,q}_{t}(\cdot)$, so that $\sigma=\sigma^{\alpha,q}+\sigma^\zeta$ in law.
The algorithm to sample $(A_{L_t}, \Gamma_t)$, with $q =0$, is as follows. Assume first that $t\le r$. Sample an exponential r.v. $\mathcal{E}$ with mean $1/\Upsilon_t$, where
\begin{align*}
\Upsilon_t := \int_0^{+\infty} \widetilde\zeta_t (ds) \in [0, +\infty),
\end{align*}
where $\Upsilon_t=0$ means that $\sigma^\zeta(t)=0$ a.e., and thus the subordinator $\sigma$ considered here is just a truncated stable subordinator. This exponential represents the first jump of the compound Poisson component. We sample the inverse and the overshoot $(L_t^{\alpha,q}, \Gamma_t^{\alpha,q})$ of the truncated 
subordinator $\sigma^{\alpha,q}$ by using \cite[Algorithm 4.1]{Dassios2020} (since $q=0)$. If $\mathcal{E}>L^{\alpha,q}_t$, then the first passage time is given by the truncated stable subordinator $\sigma^{\alpha,q}$. Otherwise, i.e. if $\mathcal{E}\le L^{\alpha,q}_t$, we generate $\sigma^{\alpha,q
}$ conditionally to be less than $t$ (by \cite[Algorithm 3]{Jorge2023b}) and the jump of $\sigma^\zeta$ from the measure $\widetilde\zeta_t(\cdot) / \Upsilon_t$ to obtain the position of the process $\sigma$ at the jump time $\mathcal{E}$. If the process $\sigma$ did not cross the level $t$ at time $\mathcal{E}$, by using the strong Markov property, we may now repeat the procedure on a reduced level. 

Assume now that $t> r$. Create the grid $0=t_0<t_1<\dots<t_n=t$ such that $t_i=\frac{i}{n}t$, $i=0,\dots,n$, where $n=\lceil t/r\rceil$. This implies that $t_{i+1}-t_i\le r$ so to get $(L_t,\Gamma_t)$ we can apply Algorithm \ref{alg:3} where in its line 6 we use the first part of the algorithm described above.

This whole algorithm is given as Algorithm \ref{alg:simplif}. In it, we assume that it is possible to sample from $\zeta(\cdot)/\zeta(0,+\infty)$ with a constant complexity (e.g. by inversion of the distribution function), from which easily follows that we can also sample from $\widetilde \zeta_t(\cdot)/\Upsilon_t$ with a constant complexity.

        

\begin{algorithm}[ht]
    \caption{Generating $(A_{L_t}, \Gamma_t)$ for fixed $t>0$, under the decomposition \eqref{decomp} with $q=0$, and $A_{L_0}=0$, $\sigma_{L_0}=0$.}\label{alg:simplif}
	\KwData{Fix the parameters in \eqref{decomp}.}
    $x\gets 0$, $v\gets 0$, $\mathfrak{b} \gets t$ \;
    \eIf{$\mathfrak{b} \le r$}{
        Decompose $\nu$ as in \eqref{decomp2} (with a new truncation level $\mathfrak{b}$, i.e. into $\nu_\mathfrak{b}^{\alpha,\theta,0} +\widetilde \zeta_\mathfrak{b}$)  and generate an exponential r.v. $\mathcal{E}$ with mean $1/\Upsilon_\mathfrak{b}$\;
        Generate $(x^\prime, v^\prime) \sim (L_\mathfrak{b}^{\alpha,\theta,0}, D^{\alpha,\theta,0}_{\mathfrak{b}})$ from \cite[Algorithm~4.1]{Dassios2020}\;
    \eIf{$\mathcal{E} > x^\prime$}{
            $(x,v) \gets (x+x^\prime, v+v^\prime)$\;
            $\mathfrak{b}\gets t-v$\;
        }{
            Generate $W \sim \sigma_{\mathcal{E}}^{\alpha,\theta,0} \mid \sigma_{\mathcal{E}}^{\alpha,\theta,0} < \mathfrak{b}$ from \cite[Algorithm~3]{Jorge2023b}\;
            Generate a r.v. $\mathcal{J} \sim \widetilde\zeta_\mathfrak{b}(\cdot) / \Upsilon_\mathfrak{b}$\;
           $(x,v) \gets (x+\mathcal{E}, v+\mathcal{J}+W)$\;
           $\mathfrak{b} \gets t-v$\;
        }
        \eIf{$\mathfrak{b}\leq 0$}{$(A_{L_t},\Gamma_{t}) \gets (x,v-t)$\;} {Repeat from line 3}}
        {
        Create a grid $0=t_0<t_1<\dots<t_n=t$ with $t_i=\tfrac{i}{n}t$, $n=\lceil t/r\rceil$\;
        Get $(A_{L_t},\Gamma_t)$ by applying Algorithm~\ref{alg:3} on the grid $t_0<\dots<t_n$, while using lines 2--18 of this algorithm in line 6 of Algorithm~\ref{alg:3}\;
        }
   \end{algorithm}

In order to include $q>0$ and/or sample the undershooting, i.e., the vector $(A_{L_t}, \gamma_t)$, one can adapt the structure of Algorithm \ref{alg:simplif}, together with integrating the tools developed in \cite{Jorge2023b, Jorge2023a}. This yields a slightly simplified version of \cite[Algorithm 1]{Jorge2023b}, which we present as Algorithm \ref{alg:simplif_2}.

    
        

\begin{algorithm}[ht]
    \caption{Generating $(A_{L_t}, \Gamma_t, \gamma_t)$ for fixed $t>0$, under the decomposition \eqref{decomp}, and $A_{L_0}=0$, $\sigma_{L_0}=0$.}
    \label{alg:simplif_2}
    \KwData{Fix the parameters in \eqref{decomp}.}
    $(x,z,v)\gets (0,0,0)$, $\mathfrak{b} \gets \min\{t,r\}$ \;
    Generate an exponential r.v. $\mathcal{E}$ with mean $1/\Upsilon_r$\;
    \While{$\mathfrak{b}\ge 0$}{
        Generate $(x^\prime, z^\prime, v^\prime) \sim (L_\mathfrak{b}^{\alpha,\theta,q}, H^{\alpha,\theta,q}_{\mathfrak{b}}, D^{\alpha,\theta,q}_{\mathfrak{b}})$ from \cite[Algorithm 2]{Jorge2023b}\;
        \eIf{$\mathcal{E} > x^\prime$}{
        $(x,z,v)\gets (x+x^\prime,v+z^\prime,v+v^\prime)$\;
            $\mathcal{E}\gets \mathcal{E}-x^\prime$\;
            $\mathfrak{b} \gets \min\{t-v,r\}$ \;
        }{
            Generate $W \sim \sigma_{\mathcal{E}}^{\alpha,\theta,q} \mid \sigma_{\mathcal{E}}^{\alpha,\theta,q} < \mathfrak{b}$ from \cite[Algorithm~4]{Jorge2023b}\;
            Generate a r.v. $\mathcal{J} \sim \widetilde\zeta_r(\cdot) / \Upsilon_r$\;
            $(x,z,v)\gets (x+\mathcal{E}, v+W,v+\mathcal{J}+W)$\;
            $\mathfrak{b} \gets \min\{t-v,r\}$ \;
            Generate an exponential r.v. $\mathcal{E}$ with mean $1/\Upsilon_r$\;
        }
        
    }
    $(A_{L_t},\gamma_t,\Gamma_t)\gets (x,t-z,v-t)$\;
\end{algorithm}

\subsubsection{Computational complexity}

We remark that the two mentioned algorithms  {\cite[TSFFP-Alg]{Jorge2023a} and \cite[Algorithm 1]{Jorge2023b}} are fast, see the discussion in \cite[Section 2.2]{Jorge2023a}, {so by using them, our algorithms are fast in the same sense as well, since} we perform a finite number of iterations and elementary operations. {We will make this assertion more precise in the following text, where we will also show that our Algorithms \ref{alg:simplif} and \ref{alg:simplif_2} are fast as well.}

Computational complexity refers to the resources needed by an algorithm to complete its task and is usually expressed as a function of the inputs of the algorithm, see \citep{cormen}. Here, we focus on the computation time, also known as running time, which encompasses the number of instructions executed and data accesses performed during the algorithm's execution. Additionally, if an algorithm incorporates an element of randomness, its computational complexity becomes random as well,  so it is also meaningful to study the moments of the (random) running time. 

{To exemplify, the running time of \citep[Algorithm 1]{Jorge2023b} which samples $(L_t,H_t,D_t)$ has moments of all orders and the bound for its expected running time is given in \cite[Theorem 2.1]{Jorge2023b} in closed form. 

As it is commonly taken, we assume that that elementary operations (such as addition, multiplication, definition of objects, and evaluation of functions, etc.) have the same cost of one. 

We will also analyze our algorithms under the assumption that we use general (arbitrary) algorithms to sample from $(L_t,\gamma_t)$ and $(L_t,\Gamma_t)$. Here we assume that the expected running time of such algorithm used in Algorithm \ref{alg:2} to sample $(L_t,\gamma_t)$ is denoted by $\mathcal{K}^\gamma(t)$ and that it is bounded in bounded intervals as a function of $t$. In the same manner, we introduce the notation $\mathcal{K}^\Gamma(t)$. By considering e.g. $\sup_{h\in[0,t]} \mathcal{K}^\gamma(h)$, without a big loss of generality we assume that $t\mapsto\mathcal{K}^\gamma(t)$ and $t\mapsto\mathcal{K}^\Gamma(t)$ are also non-decreasing.
}

\begin{proposition}\label{complex_alg12}
    Let  $\mathcal{K}^\gamma(\cdot)$ and $\mathcal{K}^\Gamma(\cdot)$ be the expected running times of the algorithms called in Algorithm \ref{alg:2} (lines 5 and 11) and Algorithm \ref{alg:3} (line 6), respectively, under the assumptions above. Let $n\in\N$ be the number of time points $t_0<t_1<\cdots<t_n$ with the maximum step $h=\max\{t_k-t_{k-1}:k=1,\dots,n\}$. Assume that the inverse of the cumulative distribution functions of the random variables $\mathcal{S}_v$, $v>0$, are known explicitly. Then Algorithm \ref{alg:2} has finite expected running time bounded by $n(\mathcal{K}^{\gamma}(h)+7)$, while Algorithm  \ref{alg:3} has finite expected running time bounded by $n(\mathcal{K}^{\Gamma}(h)+4)$.
\end{proposition}
\begin{proof}
    Let us analyze Algorithm \ref{alg:3} first. In it, lines 1, 2, 3, 6, and 7 correspond to the worst case, since they contain the largest number of computations and also perform the call to {the algorithm with the expected running time at most $\mathcal{K}^{\Gamma}(h)$}. Lines 1, 2, and 7 have the cost of one {by the assumption}, and are executed $n$ times. Line 6 has {the expected running time at most $\mathcal{K}^{\Gamma}(h)$}, and is executed $n$ times. The claim for Algorithm \ref{alg:3} now easily follows.

    {The case of the Algorithm \ref{alg:2} follows similarly, where we note that its line 7 has the cost of two by the assumption that the inverse of the cumulative distribution function is known (so  first a uniform random variable is generated and then the inverse cumulative distribution is evaluated in it to get the sample of $\mathcal{S}_v$).}
\end{proof}

Now we want to show practicality of our algorithms by calculating their complexity in the cases when known algorithms for sampling $(L_t,D_t)$ and $(L_t,H_t)$, for fixed time $t>0$ and under $\sigma_0=0$, are used. {In Remark \ref{1644} we also compare and describe their actual computation times.} First we collect known complexities of algorithms from \cite{Dassios2020,Jorge2023a,Jorge2023b} which are called in our algorithms.

\begin{remark}[{Complexity of \cite[Algorithm 4.1]{Dassios2020}}]\label{rem:Dassios}
    In Algorithm \ref{alg:simplif} we use \cite[Algorithm 4.1]{Dassios2020} in line 4. Since \cite[Algorithm 4.1]{Dassios2020} conducts rejection sampling procedure together with operations of cost of one, its complexity has finite exponential moments. It is easily observed that the expected running time of \cite[Algorithm 4.1]{Dassios2020} is in essence bounded by the minimum value of the function
    \begin{align*}
        C(\lambda)=\frac{\alpha  \Gamma(2 - \alpha) A_0}{\Gamma(1-\alpha)(1 - \alpha)(A_0 - \lambda)^{2 - \alpha}}
e^{\Gamma(1-\alpha)^{-\frac{1}{\alpha}} \lambda^{1 - \frac{1}{\alpha}} \alpha (1 - \alpha)^{\frac{1}{\alpha} - 1}}
,\quad \lambda \in (0,A_0),
    \end{align*}
    where $A_0=(1-\alpha)\alpha^{\frac{\alpha}{1-\alpha}}$, up to some constant complexity due to elementary calculations. We call this expected complexity $\mathfrak{C}(\alpha)$, and note that the minimum can be easily numerically obtained. Also, we note that $\lim_{\alpha\to0}\mathfrak{C}(\alpha)=+\infty$  and $\lim_{\alpha\to1}\mathfrak{C}(\alpha)=1$. Hence, the algorithm is practically useful for not too small $\alpha$, e.g. $\mathfrak{C}(0.2)\approx 730$, $\mathfrak{C}(0.1)\approx 94500$, which we demonstrate later on, see Figure \ref{fig:speed comparison}. Furthermore, the complexity is independent of the barrier $t$ by the scaling properties of the stable subordinator. Note also that there is a small typing error in the pseudocode of \cite[Algorithm 4.1]{Dassios2020} where the function $C$ is missing some constants from its true definition in \cite[Eq. (4.10)]{Dassios2020}.
\end{remark}

\begin{remark}[{Complexity of \cite[Algorithm 2]{Jorge2023b}}]\label{rem:CLM}
    The complexity of \cite[Algorithm 2]{Jorge2023b} which we use in Algorithm \ref{alg:simplif_2} in line 4, is given by the complexity of \cite[TSFFP-Alg]{Jorge2023a} multiplied by the number of rejections in line 2 of  \cite[Algorithm 2]{Jorge2023b}. The rejection happens in the case when for a tempered stable subordinator $\sigma$, i.e. the one with the L\'evy density $\nu(s)=\alpha \Gamma(1-\alpha)^{-1}e^{-qs}s^{-1-\alpha}$, happens that $\{D_t-H_t>t\}$. It is known that
    \begin{align}\label{eq1818}
        \pr^{(x,0)}(D_t-H_t>t)=\overline \nu(t)u((0,t]), \quad t>0,
    \end{align}
    see \cite[Lemma 1.10]{bertoin1999}, where $u(t)=e^{-qt}t^{-1+\alpha}E_{\alpha,\alpha}\big( (qt)^\alpha\big)$ denotes the potential density of $\sigma$ (see \cite[Section 1.3]{bertoin1999}), and where $E_{\beta,\delta}\big( t\big)=\sum_{n=0}^\infty\frac{t^n}{\Gamma(\beta+\delta n)}$ is the two-parameter Mittag-Leffler function. 
    
    Thus, since \cite[TSFFP-Alg]{Jorge2023a} has finite exponential moments and the rejection follows an (independent) geometric distribution, \cite[Algorithm 2]{Jorge2023b} has finite moments of all orders with the expected running time bounded by
    \begin{align}\label{cross-big}
        \mathcal{K}_{CLM}(t)\coloneqq\kappa \left( 1 + \frac{q t}{\alpha} \right) 
\frac{(1 - \alpha)^{-3} + |\log \alpha| + \log N }\alpha \frac{1}{1-\sup_{\{0< h\le t\}}\overline \nu(h)u((0,h])},
    \end{align}
    where $N$ denotes the desired number of bits of precision, while the constant $\kappa$ is independent of other parameters. Here the first three terms come from the expected running time of \cite[TSFFP-Alg]{Jorge2023a} as given in \cite[Corollary 1]{Jorge2023a}. The last part comes by taking supremum in \eqref{eq1818} over $\{0<h\le t\}$, where we note that it is easy to check (by using Tauberian theorem) that $\overline \nu(t)u((0,t])\to 1/(\alpha \Gamma(\alpha)\Gamma(1-\alpha))<1$ as $t\to 0$, so the second term in \eqref{cross-big} is well defined, and thus making the whole expression in \eqref{cross-big} monotone in $t$.
\end{remark}

\begin{corollary}\label{cor:complexDassios}
    Let $\mathfrak{C}(\alpha)$ denote the expected running time of \cite[Algorithm 4.1]{Dassios2020} as in Remark \ref{rem:Dassios}. Then Algorithm \ref{alg:simplif} has finite moments of all orders with the expected running time of 
    \begin{align}\label{eq1251}
        \lceil t/r\rceil(\mathfrak{C}(\alpha)+9)\textrm{const}(\zeta),
    \end{align}
    while Algorithm \ref{alg:3} when 
    Algorithm \ref{alg:simplif} is used inside of it has finite moments of all orders with the expected running time bounded by 
    $$4n+\textrm{const}(\zeta)\big(\mathfrak{C}(\alpha)+9\big)\sum_{i=1}^n\left\lceil \frac{t_i-t_{i-1}}r\right\rceil,$$
    where in both expressions the term $\textrm{const}(\zeta)$ represents (the same) positive constant which depends only on the measure $\zeta$.
\end{corollary}
\begin{proof}
    To prove \eqref{eq1251} and finiteness of moments of all orders, assume first $t\le r$. We can repeat the reasoning of the proof of \cite[Theorem 2.1]{Jorge2023b}. Here, instead of the bound for the complexity of \cite[FPTS-Alg]{Jorge2023b}, we use the bound for \cite[Algorithm 4.1]{Dassios2020} (used in line 4) obtained in Remark \ref{rem:Dassios}. Line 9 has the expected running time of 5, see \cite[Proposition 3.1]{Jorge2023b}, and we have additional 4 operations of cost of one.  Thus, in the case $t\le r$, by mimicking the proof of \cite[Theorem 2.1]{Jorge2023b}, we get that the expected running time of Algorithm \ref{alg:simplif} is bounded by $(\mathfrak{C}(\alpha)+9)\textrm{const}(\zeta)$, where the constant comes as a bound related to the second parenthesis in \cite[Eq. (2.4)]{Jorge2023b}, see also \cite[Eq. (5.5)]{Jorge2023b}. If on the other hand $t> r$, we have to repeat the first mentioned scenario $\lceil t/r\rceil$ times, so we get \eqref{eq1251}.

    Now it is clear that Algorithm \ref{alg:3} has finite moments of all orders since Algorithm \ref{alg:simplif} has. Here at each iteration we have the expected running time of $\lceil (t_i-t_{i-1})/r\rceil(\mathfrak{C}(\alpha)+5)\textrm{const}(\zeta)$ so the claim follows.
\end{proof}

\begin{corollary}\label{cor:complexCLM}
Let $\mathcal{K}_{CLM}(\cdot)$ denote the expected running time of \cite[Algorithm 2]{Jorge2023b} as in Remark \ref{rem:CLM}. Then Algorithm \ref{alg:simplif_2} has finite moments of all orders, with the expected running time bounded by 
\begin{align}\label{eq1415}
    \lceil t/r\rceil(\mathcal{K}_{CLM}(r)+5e^{qr}+4)\textrm{const}(\zeta),
\end{align}
while Algorithm \ref{alg:2} when Algorithm \ref{alg:simplif_2} is used inside it has finite moments of all orders, with the expected running time bounded by 
    $$7n+\textrm{const}(\zeta)\big(\mathcal{K}_{CLM}(r)+5e^{qr}+4\big)\sum_{i=1}^n\left\lceil \frac{t_i-t_{i-1}}r\right\rceil.$$
\end{corollary}
\begin{proof}
The proof is similar to the proof of Corollary \ref{cor:complexDassios}. Again, we can follow the same reasoning of the proof of \cite[Theorem 2.1]{Jorge2023a}. Here, for the bound of the expected running time of \cite[Algorithm 2]{Jorge2023b} (used in line 4), we use the bound \eqref{cross-big}, while for the bound of the expected running time of \cite[Algorithm 2]{Jorge2023b} (used in line 10), we use \cite[Proposition 3.1 and Remark on p.9]{Jorge2023b}.
\end{proof}


\begin{example}
    As already noted in Remark \ref{rem:stable-gen}, our algorithms for sampling the trajectories allow us to deal with stable subordinators and their tempered versions. In addition to those, Section 4.2 of \citep{Jorge2023b} provides an illustrative example which falls into the class given by \eqref{decomp}. Another example can be found in \citep[p. 307]{Schilling2012}, where the L\'evy measure and the Laplace exponent of the subordinator are $\nu(t)=\sin{(\alpha\pi)}\Gamma(1-\alpha)e^{-at}t^{\alpha -2}(at+1-\alpha)/\pi$, $\alpha\in(0,1)$, $a>0$, and $\phi(\lambda)=\lambda/(\lambda+\alpha)^\alpha$, respectively. Note that by setting $\beta = 1-\alpha$ we have
    \begin{align}
    \nu(s) = \frac{1}{\beta} e^{-as}\nu_\beta (s) (as+\beta),\quad \nu_\beta (s) = \frac{\beta}{\Gamma(1-\beta)}s^{-\beta -1},
    \end{align} where we use the fact $\sin{(\pi \alpha)}\Gamma(\alpha)\Gamma(1-\alpha)=\pi$. Then it is clear that $\nu(s)/\nu_\beta (s) = 1+\mathcal{O}(s)$ as $s\downarrow 0$. 
    Thus, the assumptions of \cite[Algorithm-1]{Jorge2023b} are satisfied, see the discussion in Section \ref{sec:prelim}. In order to apply our algorithm, one should also have the distribution of the random variable $\mathcal{S}_v$. However, the tail of $\nu$ is known explicitly $\overline\nu(t)=\Gamma(1-\beta)^{-1}e^{-at}t^{-\beta}$.
\end{example}

\begin{remark}\label{1644}
    In Figure \ref{fig:speed comparison}, in the upper subfigure, we compare the real running times of Algorithms \ref{alg:simplif}, \ref{alg:simplif_2} and \cite[Algorithm 1]{Jorge2023b} (with its parameter $\rho=0.5$)
    on $\alpha$-stable subordinators, when $\alpha$ ranges from 0.1 to 0.9. For each $\alpha$, we make $M=100$ simulations of the first passage event through level $t=1$, where $\nu(s)$ is decomposed as
    \begin{align}
        \nu(s)=\frac{\alpha}{\Gamma(1-\alpha)}s^{-\alpha-1}\1_{\{s\le 1\}}+\frac{\alpha}{\Gamma(1-\alpha)}s^{-\alpha-1}\1_{\{s> 1\}},
    \end{align}
    and we measure the real running time (in seconds) of such simulation for each of the mentioned algorithms.  We repeat this sampling $N=1000$ times to obtain the data which we plot as boxplots in the upper subfigure of Figure \ref{fig:speed comparison}. We note that for $\alpha=0.1$, Algorithm \ref{alg:simplif} does not perform well as already said in Remark \ref{rem:Dassios} and this boxplot is outside the figure since it is centered around the value of $8$ seconds. However, for bigger $\alpha$, Algorithm \ref{alg:simplif} performs best. 

    In the lower subfigure, we compare the real running times of Algorithm \ref{alg:simplif_2} and \cite[Algorithm 1]{Jorge2023b} (with its parameter $\rho=0.5$).  We evaluate the running times on subordinators with the L\'evy measure decomposed as
    \begin{align}
        \nu(s)=\frac{0.65}{\Gamma(1-0.65)}e^{-qs}s^{-1.65}\1_{\{s\le 1\}}+s^{-5}\1_{\{s> 1\}},
    \end{align}
    where $q$ is ranging from $1$ to $7.75$. For each $q$, we make $M=10$ simulations of the first passage event through level $t=1$, measure the real running time (in seconds), and repeat such sampling $N=100$ times to obtain the data which we plot as boxplots in the lower subfigure of Figure \ref{fig:speed comparison}.

   {We observe that \cite[Algorithm 1]{Jorge2023b} and Algorithm \ref{alg:simplif} have comparable real running times; indeed the second is just a simplified version of the first where the free parameter $\rho$ (which slightly optimizes the running time of \cite[Algorithm 1]{Jorge2023b}) is avoided.} 

    The sampling was seeded, and conducted on Microsoft Surface Pro 8 with Intel i7-1185G7 CPU and 16 GB of RAM (no GPU was used). The code is available in \cite{Python}.
\begin{figure}[ht]
    \centering
    \begin{subfigure}{.85\textwidth}
        \centering
    \includegraphics[width=1\linewidth]{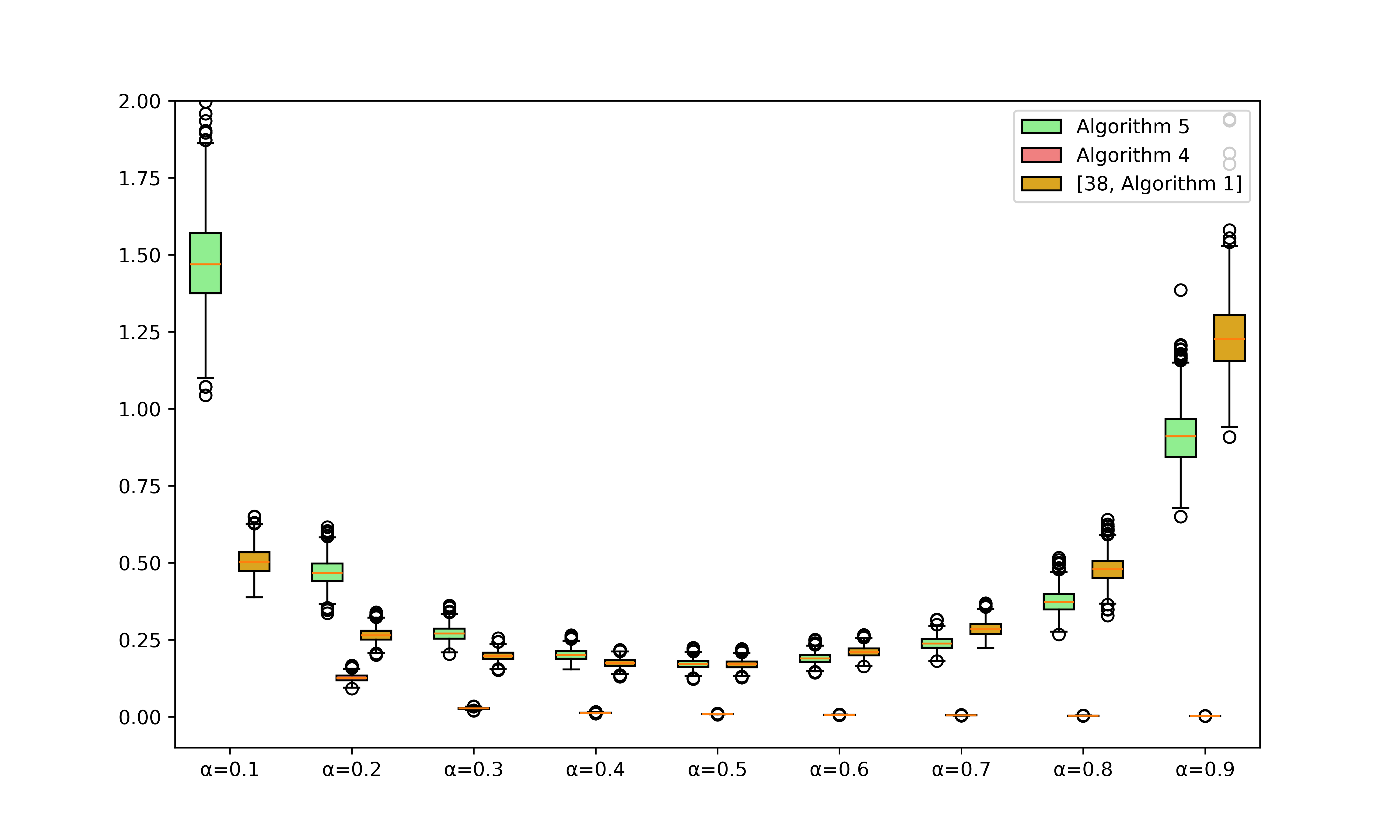}
    \end{subfigure}
    \\
    \begin{subfigure}{.85\textwidth}
        \centering
    \includegraphics[width=1\linewidth]{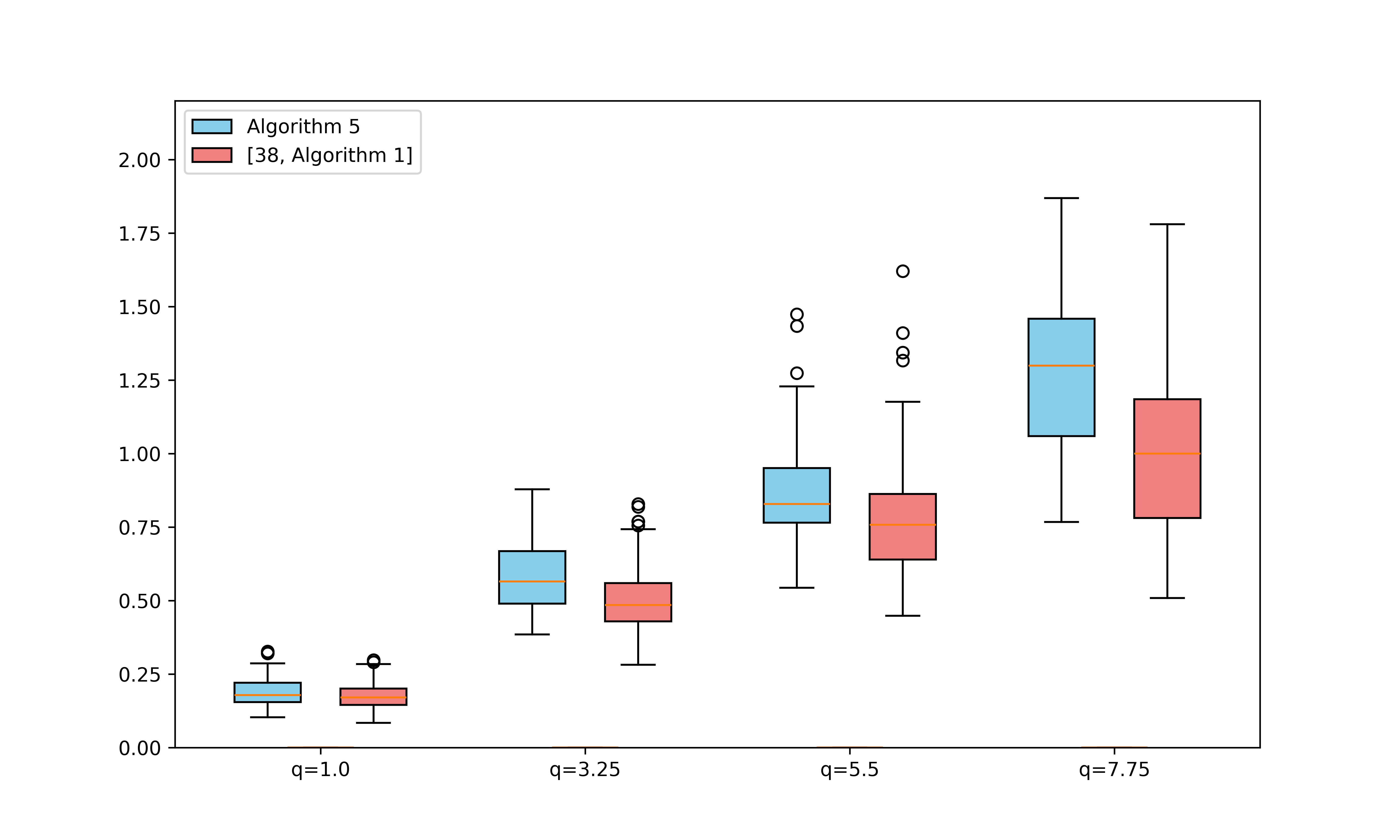}
    \end{subfigure}
    \caption{See Remark \ref{1644} for details of these boxplot comparisons.}
\label{fig:speed comparison}
\end{figure}
\end{remark}

{To conclude this section, for illustrative purposes, we examine the $\alpha$-stable subordinator and its tempered version and provide Figures \ref{Fig1} and \ref{FigAlg4} of sampled trajectories obtained using Algorithms \ref{alg:2}, \ref{alg:3} and \ref{alg:triplet}. We also attach the GitHub repository \cite{Python} with the corresponding Python code used to sample figures for the reader's convenience.}

\begin{figure}[ht]
    \centering
    \begin{subfigure}{.85\textwidth}
        \centering
        \includegraphics[width=.975\linewidth]{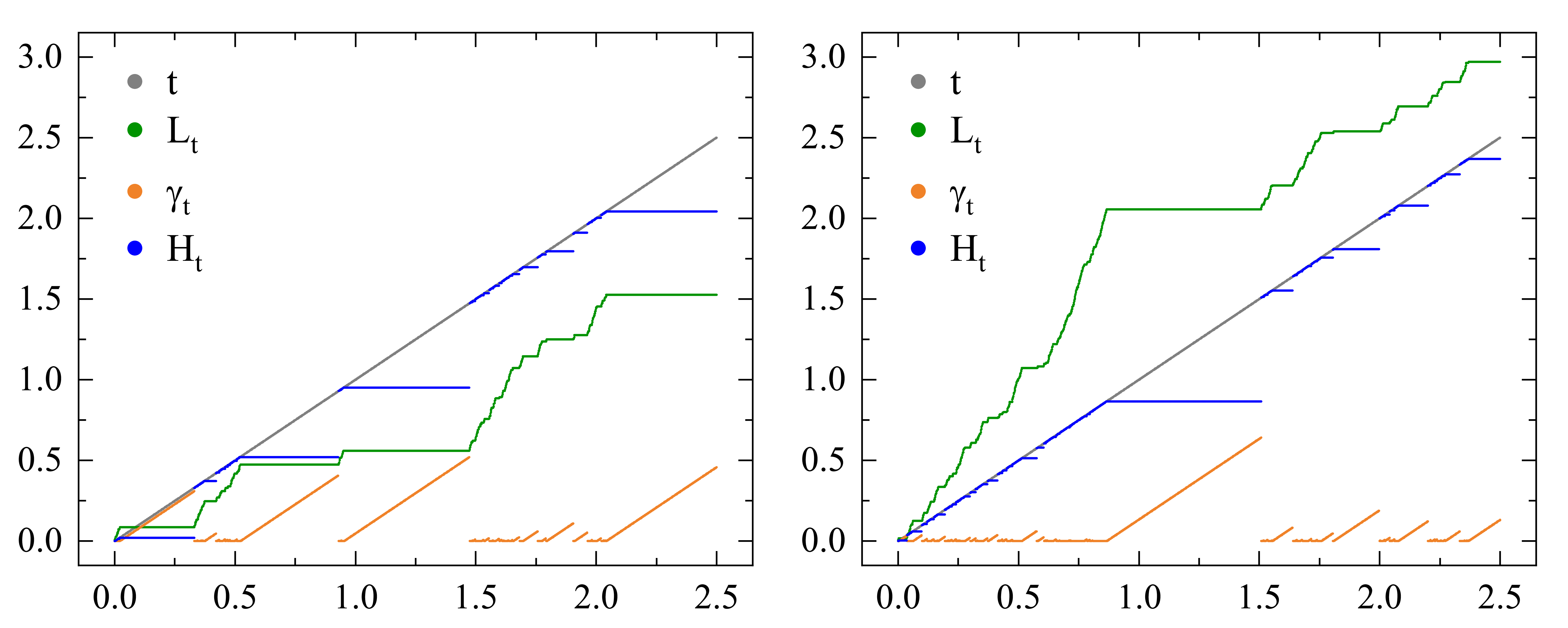}
        \caption{Sample of the $\alpha$-stable inverse subordinator, age process, and undershoot process through Algorithm \ref{alg:2}, left figure, while the right one stands for the tempered version.}
    \label{Top}
    \end{subfigure}
    \\
    \begin{subfigure}{.85\textwidth}
        \centering
        \includegraphics[width=.975\linewidth]{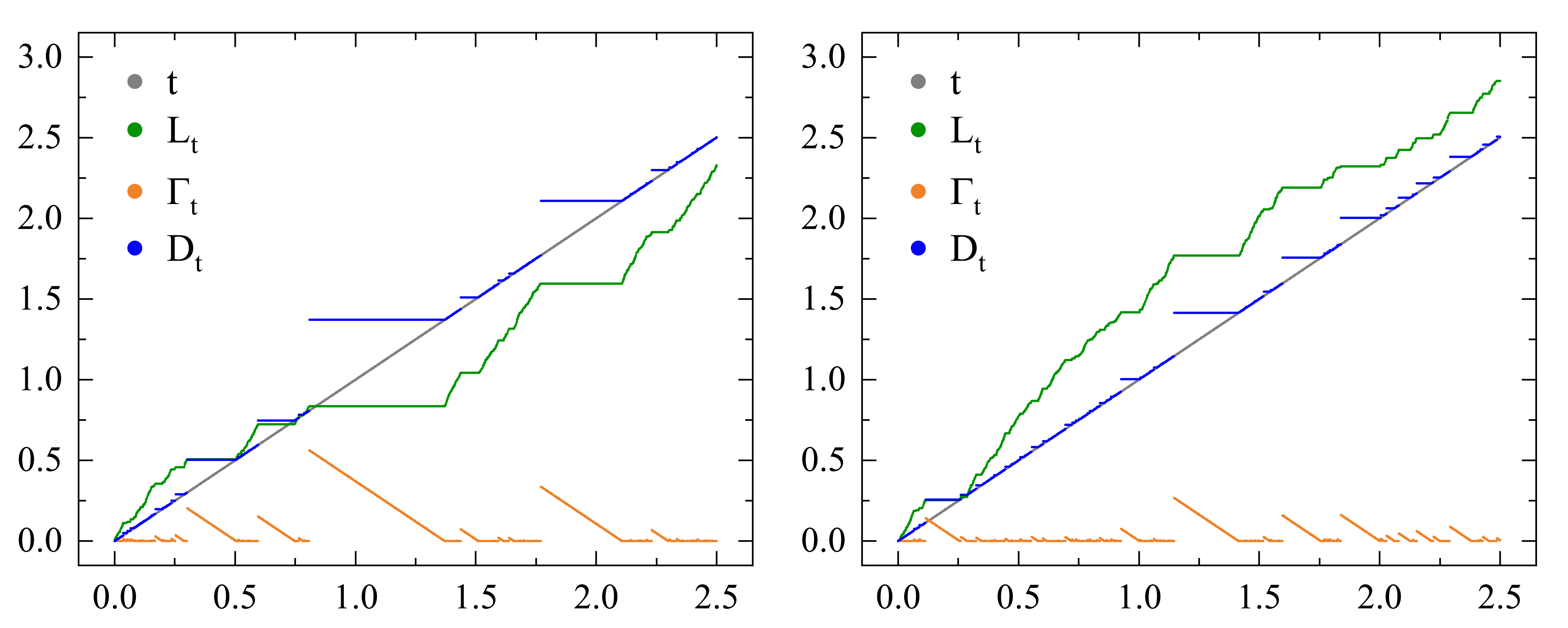}
        \caption{Sample of the $\alpha$-stable inverse subordinator, remaining lifetime process, and overshoot process through Algorithm \ref{alg:3}, left figure, while the right one stands for the tempered version.}
    \label{Bottom}
    \end{subfigure}

    \caption{Paths for the $\alpha$-stable inverse subordinator and its tempered version in a time interval from $0$ to $2.5$ with $10^3$ equidistant steps with $\alpha =0.75$, $q=1$, and $x=v=r=0$. The green, orange, and blue colors stand for the inverse subordinators, age process (remaining lifetime), and  undershooting (overshooting) process, respectively. Time is displayed in black.}
\label{Fig1}
\end{figure}

\begin{figure}[ht]
     \centering
     \includegraphics[width=0.95\linewidth]{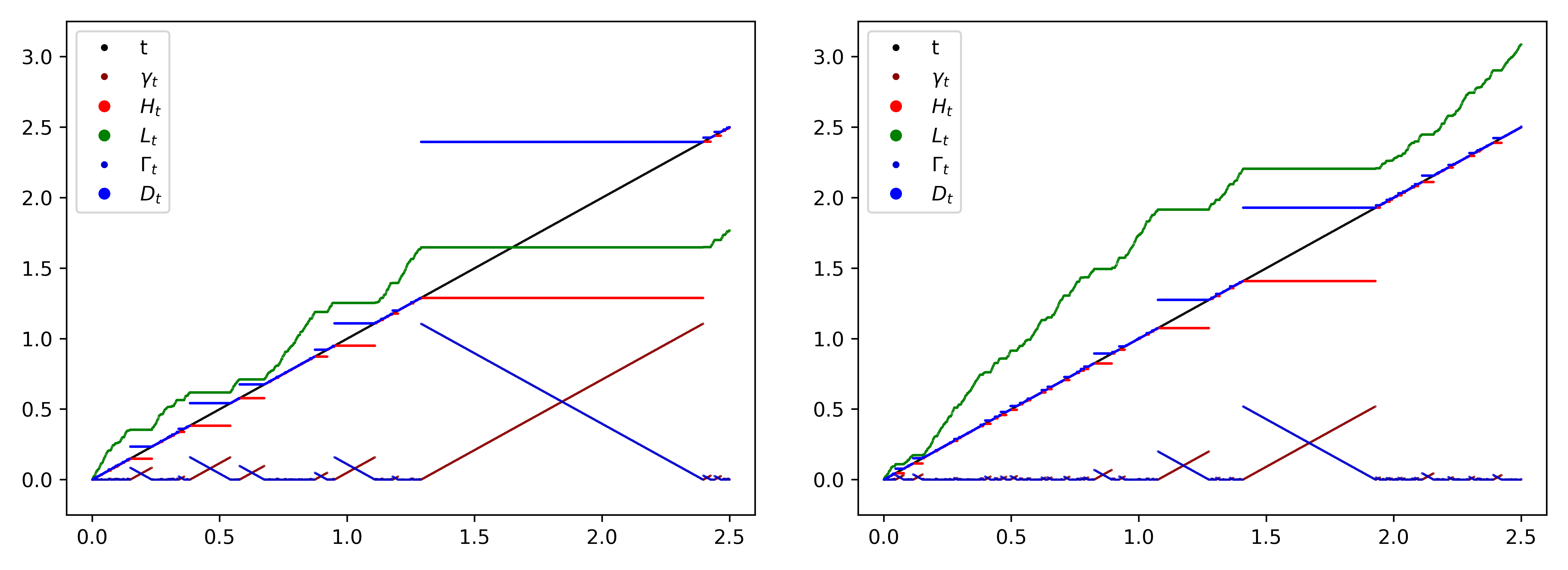}
     \caption{Sample of the $\alpha$-stable inverse subordinator, the remaining lifetime, the overshoot, the age and the undershoot processes through Algorithm \ref{alg:triplet}, left figure, while the right one stands for the tempered version; the parameters are $\alpha=0.75$ and $q=1$.}
     \label{FigAlg4}
 \end{figure}

\section{Exact simulation of time-changed processes and Monte Carlo method}\label{secmonte}

The ability to make simulations of the vectors \eqref{vectors} and \eqref{vectors2} yields the possibility of exact sampling of the trajectories of suitable time-changed processes.
Namely, consider a Feller process $M_t$, $t\geq0$, on $\R^d$. As long as it can be sampled exactly, its composition with an independent inverse subordinator, its undershoot, or its overshoot (sampled as shown in the previous section) can be also sampled exactly.
{In this section, the goal is to study such time-changed Feller processes, where the highlight is Theorem \ref{thm1541} - the central limit theorem for its functionals. We show the significance of this theorem by, first, giving the connection between time-changed Feller processes and non-local-in-time parabolic equations (see Subsection \ref{sec:MCA}), and second, its applicability to a broad class of L\'evy processes.}

Here is the construction of the process we are going to sample. 
Consider the canonical construction of the process $(M, \sigma) \, = \, ((M_t, \sigma_t),\, t\geq 0)$, where $(M_t,\, t\geq0)$ is a Feller process and $\sigma = (\sigma_t,\, t\geq0)$ is an independent strictly increasing c\`adl\`ag process with stationary and independent increments. Hence, we have the (family of) filtered probability spaces $\l \Omega, \mathcal{F}_\infty, (\mathcal{F}_t)_{t\geq0}, \pr^{(x,v)} \r$ where $\mathcal{F}_t$ is the natural filtration, $\mathcal{F}_\infty = \bigvee_t \mathcal{F}_t$, and $\pr^{(x,v)}$ is the unique measure such that $\pr^{(x,v)} (M_0=x, \sigma_0=v)=1$.  In particular, we have that $\sigma^0= (\sigma^0_t,\, t\geq0)$ defined as $\sigma^0_t=\sigma_t-\sigma_0$ is a subordinator under any $\pr^{(x,v)}$ while $\sigma =(\sigma_t,\, t\geq0)$ is a subordinator under $\pr^{(x,0)}$. Define now the first-passage processes
\begin{align*}
    L_t\coloneqq\inf\{u\geq0:\, \sigma_u>t\}.
\end{align*}
Recall that since $\sigma$ is assumed to be strictly increasing then the inverse is, a.s., continuous. We define
\begin{align*}
    H_t\coloneqq\sigma^0_{L_t-}, \qquad D_t\coloneqq\sigma^0_{L_t}.
\end{align*}
In particular, under $\pr^{(x,0)}$, $\sigma_t$ is a subordinator, $H_t$ is the undershooting and $D_t$ the overshooting. 
Following \cite[Section 4]{ascione2024} it is possible to prove that the processes
\begin{align}
    & (M, \sigma) \, = \, ( (M_t, \sigma_t),\, t\geq0 ) \label{42},\\
    &  (M_{\sigma^0}, \sigma) \, = \, (  (M_{\sigma^0_t}, \sigma_t),\, t\geq 0 ) \label{43}
\end{align}
are Markov additive processes in the sense of Cinlar \cite{cinlarma} on the probability spaces above. Therefore, it is noteworthy that whenever the Feller process \eqref{42} is a jump diffusion it can be inserted in the theory developed in \cite[Section 4]{Meerschaert2014}, i.e. the left-continuous process $(M_{L_t}, \gamma_t)$ has the simple Markov property, while $(M_{L_t}, \Gamma_t)$ is a Hunt process. Equivalently, if the process \eqref{43} is a jump diffusion, then the left-continuous process $(M_{H_t}, \gamma_t)$ has the simple Markov property while $(M_{D_t}, \Gamma_t)$ is a Hunt process.

To summarize, by using the results of the previous section we can make exact sampling of the vectors
\begin{align}
    & (M_{L_{t_1}}, \cdots, M_{L_{t_n}}), \label{vectorstc1} \\
    & (M_{H_{t_1}}, \cdots, M_{H_{t_n}}), \label{vectorstc2}\\
    & (M_{D_{t_1}}, \cdots, M_{D_{t_n}}), \label{vectorstc3}
\end{align}
under $\pr^{(x,0)}$, for any $n\in\N$, $0 < t_1 < \cdots < t_n$. Here $\sigma_t$ is a subordinator with the Laplace exponent \eqref{bernstinsimul}, $L_t$ is its inverse process, $H_t$ is the undershooting of the subordinator, while $D_t$ is the overshooting of the subordinator $\sigma_t$. This is possible whenever an exact algorithm to make simulations of the Feller process $M_t$ is available, as shown in Algorithm \ref{alg:4}.  We use { $T=(T_t,\, t\geq0)$} for the inverse subordinator $L_t$ or the undershoot process $H_t$, or the overshoot process $D_t$, of a given subordinator $\sigma$, where no distinction between the three is needed. 

\begin{algorithm}[H]
    \caption{Generating vectors \eqref{vectorstc1}, \eqref{vectorstc2} and \eqref{vectorstc3} under $\pr^{(x,0)}$}\label{alg:4}
    \KwData{$n\in\N$, $t_1 < t_2 < \cdots <t_n$, $x \in \mathbb{R}^d$}
    Generate $(T_{t_1}, \cdots ,T_{t_n})$ from Algorithm \ref{alg:2} or \ref{alg:3}\\
    $(s_1, \cdots, s_n) \gets (T_{t_1}, \cdots, T_{t_n})$\\
    Generate $(M_{s_1}, \cdots, M_{s_n})$\\
    $(M_{T_{t_1}}, \cdots, M_{T_{t_n}}) \gets (M_{s_1}, \cdots, M_{s_n})$\\
\end{algorithm}

Examples of Feller processes whose trajectories can be sampled exactly include for example the following processes. Diffusions for which an analytical solution is known constitute one example, see e.g.  \citep[Section 4.4]{kloeden1992} which offers many such examples. In the context of diffusion processes, we highlight the class of one-dimensional diffusions that under specific conditions of the drift and the diffusion coefficients can be transformed into the Wiener process, as shown by Ricciardi's research \cite[see Theorem 1]{Ricciardi1976}.  Similarly,  \citep[Section 4.3]{kloeden1992} provides an appropriate transformation such that a certain one-dimensional non-linear stochastic differential {equation} is reduced to a linear one, for which an analytical solution can be obtained. In higher dimensions this kind of transformation is not straightforward, but in \citep[Section 4.8]{ kloeden1992} there are some cases in which this is possible. Additionally, the authors in \citep{Casella2011} provide an algorithm for the exact simulation of jump diffusion processes {which belong to a class of solutions to the following SDE
\begin{align}
    dV_t = \varrho (V_{t-}) dt + \Sigma(V_{t-}) dW_t + \int_E g(z, V_{t-}) m(dz, dt), \quad V_0 = v_0,
\end{align}
where $\varrho$, $\Sigma$ and $g$ are coefficients satisfying the classical Lipschitz conditions,  and $m(dz, dt)$ is a random counting measure on the product space $E \times [0, T]$  with $E\subset \R$ and $T>0$, see \cite[Section 2]{Casella2011} for details.} First passage times for diffusion processes can be sampled exactly in some situations, see e.g. \cite{Sacerdote2001,Zucca2019}.

\begin{remark}\label{complex_alg3}
    The complexity {of Algorithm \ref{alg:4}} above depends both on the complexity of sampling the Feller process  {$M$ (denoted by, say, $\mathfrak{M}$ and depending on parameters of $M$ and possibly on the time points $s_1,\dots,s_n$ as well)} and the complexity of sampling the time-change. In this generality, the value of $\mathfrak{M}$ is not clear, but for space homogeneous Feller processes with well-behaved densities, it is clear that expected time of generating $(M_{s_1}, \cdots, M_{s_n})$ is bounded by $\mathcal{C}\cdot n$, {for the constant $\mathcal{C}$ independent of times $s_1,\dots,s_n$. Hence,} the total expected running time of sampling time-changed Feller processes by Algorithm \ref{alg:4}, in this case, is {equal to $n(\mathcal{K}^{\gamma}(h)+7+\mathcal{C})+2$ by using Algorithm \ref{alg:2}, or equal to $n(\mathcal{K}^{\Gamma}(h)+4+\mathcal{C})+2$ by using Algorithm \ref{alg:3}, where we recall that $\mathcal{K}^{\gamma}$ and $\mathcal{K}^{\Gamma}$ denote the expected running times of the algorithms used inside of Algorithm \ref{alg:2} (lines 4 and 11) and Algorithm \ref{alg:3} (line 6), respectively.}

    {For example, in the context of diffusion processes, with time-space deterministic transformations as provided in \cite{Ricciardi1976}, one can determine $\mathcal{C}$. Indeed, one can sample exactly the vector in line 3 of Algorithm \ref{alg:4} just by sampling Brownian motion. This means that the complexity of such sampling is given by the complexity of sampling Brownian motion.}
\end{remark}

For illustrative purposes, we look at the Brownian motion and time-change it with the tempered $\alpha$-stable inverse subordinator and its undershooting and overshooting processes; see Figure \ref{Fig2}.

 \begin{figure}[ht]
     \centering
     \includegraphics[width=0.95\linewidth]{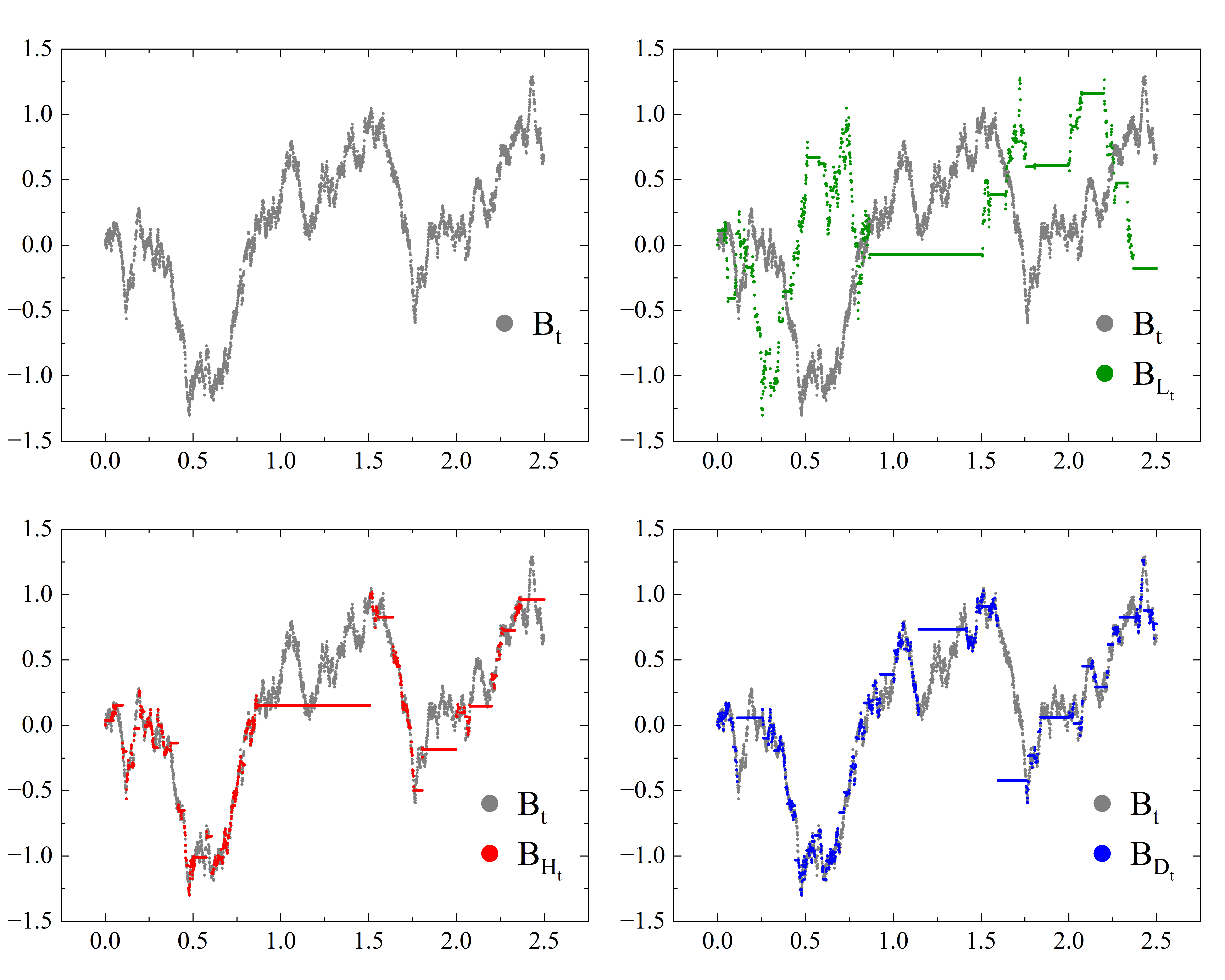}
     \caption{Time-changed Brownian motion. The Brownian motion $B_t$ shown on the top-left is time-changed with the tempered $\alpha$-stable inverse subordinator of the previous figure, i.e. Figure \ref{Fig1} (A) right-hand side, and shown in the top-right. The bottom figures show the same Brownian motion time-changed with the undershooting and overshooting processes (left and right side, respectively) of Figure \ref{Fig1} (A) and (B) on the right-hand side, respectively.}
     \label{Fig2}
 \end{figure}

\subsection{Monte Carlo approximations}\label{sec:MCA}
In this section, which draws inspiration from the recent research \citep{kolokoltsov2021}, we want to pose our attention to the Monte Carlo approximation of functionals of the process $(M_{T_t}, t\geq0)$,
\begin{align}\label{eq4.1}
	q(t_1, \cdots, t_n,\, x)\coloneqq \E^{(x,0)}[u( M_{T_{t_1}}, M_{T_{t_2}}, \dots, M_{T_{t_{n}}} )],
\end{align}
where $0\le t_1<t_2<\dots<t_n$, $u:\R^{d\times n}\to\R$,  $M=(M_t,\, t\geq0)$ is a $d$-dimensional Feller process, and, as before, {$T=(T_t,\, t\geq0)$} stands for the inverse subordinator $L_t$ or the undershoot process $H_t$, or the overshoot process $D_t$, of a given subordinator $\sigma$, where no distinction between the three is needed. 

It is noteworthy that our study builds upon and generalizes the findings of previous research in this field in the following aspects. Any inverse subordinator and undershooting processes are considered, as well as any overshooting process that satisfies \eqref{1610}.  We consider any $d$-dimensional Feller process instead of $d$-dimensional isotropic stable process. The availability of exact sample paths of inverse, undershoot, and overshoot processes allows us to consider general functionals of the time-changed process instead of only the single time position.

Note that the distribution of $T_t$ under $\pr^{(x,0)}$ does not depend on $x$ and so for this distribution we use the symbol 
\begin{align}
    \mu_t(dw)\coloneqq \pr^{(x,0)} (T_t\in dw),
\end{align}
whenever there is no need to distinguish between the inverse subordinator, undershooting, or overshooting.

In particular, if $T_t$ is the inverse subordinator, then the term in \eqref{eq4.1} for $n=1$ gives the stochastic representation of the unique solution (in a suitable sense) to the general Cauchy problem
\begin{equation}
    \begin{aligned}
        \phi(\partial_t - G) q(t,x) &= 0, && x \in \R^d, \, t > 0, \\
        q(0,x) &= u(x), && x \in \R^d.
    \end{aligned}
\end{equation} 
Here $G$ is the infinitesimal generator of the Feller process $M_t$, and the operator $\phi(\partial_t)$ is a non-local in time operator which can be defined in different ways. Still, the most popular is as a fractional-type derivative (see e.g. \cite{chen, kochubei}):
\begin{align}\label{eq4.4}
    \phi(\partial_t) f(t,\cdot) \coloneqq b\partial_t f(t,\cdot) + \partial_t \int_0^t (f(s,\cdot) - f(0,\cdot)) \overline{\nu}(t-s)ds,
\end{align}
{where $b$ and $\overline \nu $ come from \eqref{bernstein-eq0}.}
However, different approaches are also possible, see e.g. \cite{kolokoltsov2015}.

If instead $T_t$ is the undershoot $H_t$, then \eqref{eq4.1} for $n=1$ is the stochastic representation of the unique solution to
\begin{equation}
    \begin{aligned}
        \phi(\partial_t - G) q(t,x) &= \overline{\nu}(t)q(0,x), \quad &&x\in\R^d,\,t>0,\\
        q(0,x) &= u(x), \quad &&x\in\R^d,
    \end{aligned}
\end{equation}
where the non-local operator $\phi(\partial_t - G)$ is non-local in time and space and given by
\begin{align*}
    -\phi(\partial_t - G) q(t,x) \coloneqq \int_0^\infty ( P_s q(t-s,x)\1_{[0,t]}(s) - q(t,x) )\nu(ds).
\end{align*}
Here $P_s$, $s\geq0$, is the semigroup generated by $G$, see \cite{ascione2024}.

Lastly, if $T_t$ is the overshoot $D_t$, then \eqref{eq4.1} for $n=1$ is the stochastic representation of the unique solution to
\begin{equation}
    \begin{aligned}
        \phi(\partial_t - G) q(t,x) &= \int_t^{+\infty}P_s\big(q(0,\cdot)\big)(x)\nu(ds), \quad &&x\in\R^d,\,t>0,\\
        q(0,x) &= u(x), \quad &&x\in\R^d,
    \end{aligned}
\end{equation}
see \cite{harmonic}.

The Monte Carlo approximation of \eqref{eq4.1} that we consider here is the random variable
\begin{align}
    q_N(t_1, \cdots, t_n)\coloneqq\frac{1}{N}\sum_{k=1}^{N} u\Big( M_{T_{t_1}^{k}}^k,\dots, M_{T_{t_n}^{k}}^k\Big),
 \label{monteest}
\end{align}
where $M_t^k$ and $T_t^{k}$ are independent and identically distributed samples of the processes $M_t$ and $T_t$, respectively.
It is clear that by the strong law of large numbers $\pr^{(x,0)}$-a.s. $q_N (t_1, \cdots, t_n) \to q(t_1, \cdots, t_n, x) $ as $N\to+\infty$, if $q(t_1,\dots,t_n,x)$ is well-defined.

For this estimator we study the central limit theorem and Berry-Esseen theorem. First, we provide general conditions for which these theorems can be obtained and after that we show that these conditions are satisfied, for example, by {all isotropic L\'evy processes}. In our theorems we will control the behaviour of $u: \R^{d\times n} \mapsto \R$ at the infinity. In particular, we will adopt the notation
\begin{align}
    u (x) \, \stackrel{\infty}{=} \,  O \l g(x) \r  \quad \overset{\text{def}}{\iff}\quad  \limsup_{|x| \to +\infty }\left|\frac{u(x)}{g(x)} \right| < +\infty.
\end{align}

\begin{theorem}\label{thm1541}
    Let {$u: \R^{d\times n} \mapsto \R$} be locally bounded and suppose that there exists $p\geq0$ such that $u{(x_1,\dots,x_n)} \stackrel{\infty}{=} O(\max_{1\leq i\leq n} |x_i|^p)$. If  $p>0$, suppose additionally that
    \begin{align}
        \int_0^{+\infty} \int_1^{+\infty}  s^{2p-1} \pr^{(x,0)} (|M_w| > s) \,  ds \, \mu_{t}(dw) < +\infty,\quad {t\in\{t_1,\dots,t_n\}},
    \label{tailfeller}
    \end{align}
    for some $x\in\R^d$. For all such $x$ as in \eqref{tailfeller},  let 
    \begin{align}
        S_N \coloneqq \frac{\sqrt{N} (q_{N}(t_1, \cdots, t_n)-q(t_1, \cdots, t_n,x))}{\sigma(t_1, \cdots, t_n , x)},
    \end{align}
    where $\sigma(t_1, \cdots, t_n,x)^{2}=\E^{(x,0)} [u(M_{T_{t_1}}, \cdots, M_{T_{t_n}})-q(t_1, \cdots, t_n,x)]^2$. Then, for $\psi\in C_b(\R^d)$ it holds that
    \begin{equation*}
        \E^{(x,0)}[\psi(S_N)]\to \int_\R \psi(w) \frac{e^{-\frac{w^2}{2} }}{\sqrt{2\pi}} \, dw, \quad\text{as $N\to+\infty$}.
    \end{equation*}
    If, furthermore, \eqref{tailfeller} is replaced by the condition
    \begin{align}
        \int_0^{+\infty} \int_1^{+\infty} s^{3p-1} \pr^{(x,0)} (|M_w| > s) ds \, \mu_t(dw) < +\infty,\quad {t\in\{t_1,\dots,t_n\}},
    \label{tailfeller3}
    \end{align}
    for some $x\in\R^d$, then for all such $x\in\R^d$ and all $\psi\in C^3_b(\R^d)$, it holds that
    \begin{align}
        \Big| \E^{(x,0)} \psi(S_N) - \int_R \psi(w) \frac{e^{-\frac{w^2}{2}}}{\sqrt{2\pi}} dw \Big| \leq  \frac{1}{6}\Big(1+2\sqrt{\frac{2}{\pi}}\Big)\|\psi^{'''}\|_{\infty}\frac{\rho(t_1, \cdots, t_n,x)}{\sqrt{N}\sigma(t_1, \cdots, t_n,x)^3},
    \end{align}
    where $\rho (t_1, \cdots, t_n,x) = \E^{(x,0)} | u(M_{t_1}, \cdots, M_{T_{t_n}}) - q(t_1, \cdots, t_n,x) |^3 $.
\end{theorem}

\begin{proof}
    For the first result, it is enough to check that the second moment of $u(M_{T_1}, \cdots, M_{T_{t_n}})$ is finite and to apply the central limit theorem. If $p=0$, the result is trivial so we proceed with calculations for $p>0$. Note that for any $x_0>0$ we have
    \begin{align}
        \E^{(x,0)} \left| u(M_{T_{t_1}}, \cdots, M_{T_{t_n}}) \right|^2 \leq \,& \E^{(x,0)}[ |u(M_{T_{t_1}}, \cdots, M_{T_{t_n}})|^2 \1_{[\max_{1\leq i\leq n}|M_{T_{t_i}}|\leq x_0]} ]\notag \\
        &+ \E^{(x,0)}[ |u(M_{T_{t_1}}, \cdots, M_{T_{t_n}})|^2 \1_{[\max_{1\leq i\leq n}|M_{T_{t_i}}|> x_0]} ].
    \end{align}
    By the assumption on the growth of $u$ and by suitably choosing $x_0$, we obtain
    \begin{equation}\label{414}
        \begin{split}
            \E^{(x,0)} \left| u(M_{T_{t_1}}, \cdots, M_{T_{t_n}}) \right|^2 &\leq \sup_{\substack{ |x_i|\leq x_0\\ \text{ for any } i}} |u(x_1, \cdots, x_n) |^2\\ 
            &\qquad+ C \E^{(x,0)} \big[\max_{1\leq i\leq n} |M_{T_{t_i}}|^{2p}\big],
        \end{split}
    \end{equation}
    for some $C>0$. We deal with the term in \eqref{414} only since the first one is finite by the local boundedness of $u$.

    Note that 
    for any $C_{0}>0$ we have
    \begin{align*}
        \E^{(x,0)} \max_{1 \leq i \leq n} |M_{T_{t_i}}|^{2p} \leq & C_0 + \int_{C_0}^\infty \pr^{(x,0)} (\max_{1\leq i\leq n} |M_{T_{t_{i}}}|^{2p} > s) ds \notag \\
        \le  \, &  C_0 + \sum_{i=1}^n \int_{C_0}^\infty \pr^{(x,0)} (|M_{T_{t_{i}}}| > s^{1/2p}) ds.
    \end{align*}
    Further, by a simple conditioning argument we have
    \begin{align}
        \pr^{(x,0)} (|M_{T_{t_{i}}}| > s^{1/2p}) \, = \, \int_0^{+\infty} \pr^{(x,0)} (|M_{w}| > s^{1/2p}) \pr^{(x,0)} (T_{t_i} \in dw).
    \end{align}

Now we use a change of variables and the assumption \eqref{tailfeller} to get
 \begin{equation}\label{1552}
    \begin{split}
        &\E^{(x,0)} \max_{1 \leq i \leq n} |M_{T_{t_i}}|^2 \\
        &\quad \leq C_0 + 2p\sum_{i=1}^n \int_0^{+\infty} \int_{C_0^{1/2p}}^{+\infty} y^{2p-1}\pr^{(x,0)} (|M_w| > y) dy \, \pr^{(x,0)} (T_{t_i} \in dw) < +\infty.
    \end{split}
 \end{equation}
 Thus, the second moment is finite and we can apply the central limit theorem to finish the proof of the first claim.
 
 For the second claim, by the very same argument, it is possible to see that
\begin{equation}
    \begin{split}
        &\E^{(x,0)} \max_{1 \leq i \leq n} |M_{T_{t_i}}|^3 \\
        &\quad\leq C_0 + 3p\sum_{i=1}^n \int_0^{+\infty} \int_{C_0^{1/3p}}^{+\infty} y^{3p-1}\pr^{(x,0)} (|M_w| > y) dy \, \pr^{(x,0)} (T_{t_i} \in dw) < +\infty,
    \end{split}
\end{equation}
 so the claim follows by an application of the Berry-Esseen Theorem, see e.g. \cite[page 355]{o2014analysis}.
\end{proof}

Now we show that a class of processes satisfies the condition of the previous theorem.

\subsubsection{Isotropic L\'evy processes} In this subsection we assume that the process $M_t$, $t\geq0$, is an isotropic L\'evy process. This means that the process is characterized by the L\'evy-Khintchine characteristic exponent
\begin{align}
    \E^{(0,0)} [e^{i \xi \cdot M_t}] \, = \, \int_{\R^d} e^{i \xi \cdot x} p_t(dx) \, = \, e^{-t\Psi(\xi)}
\end{align}
for a function $\Psi$ with the representation
\begin{align*}
    \Psi(\xi) = \Sigma^2|\xi|^2 + \int_{\R^{d}} (1-\cos \left\langle \xi,x\right\rangle) \Pi(dx).
\end{align*}
Here $\Sigma\geq0$, the measure $\Pi(\cdot)$ is the L\'evy measure, i.e. it is such that
\begin{align}
    \int_{\R^d} (|x|^2 \wedge 1) \, \Pi (dx)<\infty,
\end{align}
and the symbol $p_t(\cdot)$ stands for the law of $M_t$ under $\pr^{(0,0)}$. A L\'evy process is isotropic since both the L\'evy measure $\Pi (\cdot)$ and the law $p_t(\cdot)$ are isotropic measures, i.e. they are invariant upon linear isometries on $\R^d$.

Stable distributions (with the zero skewness parameter) exemplify isotropic L\'evy measures, and the Cauchy distribution is a special case of these distributions.

In the following, Pruitt's function for our $d$-dimensional isotropic L\'evy process will be relevant, i.e. the function 
\begin{align*}
    h(r) \coloneqq d\frac{ \Sigma^2 }{ r^2 } + \int_{ \R^{d} } \l \frac{ |z|^2 }{ r^2 } \wedge 1\r \Pi(dz),\quad r>0.
\end{align*}
This function is strictly positive and decreasing (see more on\citep{Bogdan2015, Pruitt1981}) and in particular by \cite[Remark 1]{Bogdan2015} we have that
\begin{align}
    \pr^{(x,0)} (|M_{t}|>r) &\leq 24 h(r) t,\quad\text{for}\quad r\geq 2|x|.
    \label{taillevy}
\end{align}
Then we have the following result.
\begin{proposition}
\label{proplevy}
        Let $T_t$ be the inverse, the undershooting, {or the overshooting} of a subordinator. {If $T_t$ is the overshooting, assume that subordinator's L\'evy measure $\nu$ satisfies}
    \begin{align}\label{1610}
        \int_1^{+\infty} h \nu(dh) < +\infty.
    \end{align}
    
    If  the process $M_t$ is an isotropic L\'evy process in $\R^d$, then the condition \eqref{tailfeller} is satisfied for {all $x\in\R^d$ and $t\ge0$}, and for any $0<p<1$ such that
    \begin{align}
        \int_{|z|>1} |z|^{2p} \Pi(dz) < +\infty,
        \label{ass2p}
    \end{align}
    while the condition \eqref{tailfeller3} is satisfied for any $x\in\R^d$ and for any $0<p<1$ such that
    \begin{align}
        \int_{|z|>1} |z|^{3p} \Pi(dz) < +\infty.  \label{ass3p}
    \end{align}
\end{proposition}
\begin{proof}
    Using \eqref{taillevy} we get, for any constant $C_x$ which is enough larger than $2|x| \vee 1$, that
    \begin{align}
        & \int_0^{+\infty} \int_1^{+\infty} s^{2p-1} \pr^{(x,0)} (|M_w|> s) \, ds \,\mu_t(dw)\notag \\
        \leq \, & \frac{C_x-1}{2p} +  \int_0^{+\infty} \int_{C_x}^{+\infty} {s^{2p-1}}\pr^{(x,0)} (|M_w|> s) \, ds \,\mu_t(dw) \notag \\
        \leq \,&\frac{C_x-1}{2p} + 24 \int_1^{+\infty} s^{2p-1}h(s) ds\int_0^{+\infty}  w \mu_t(dw).
        \label{418}
    \end{align}
    The second integral in \eqref{418} is {the expectation of $T_t$. When $T_t$ is the inverse subordinator, it is} finite since the expectation of an inverse subordinator always exists and equals to the renewal measure, see \cite[Section 1.3]{bertoin1996}, while the undershooting $H_t$ is a.s. bounded by $t$. {If $T_t$ is the overshooting $D_t$, its expectation exists under assumption \eqref{1610}. Indeed, from \cite[Proposition 2 in Chapter III]{bertoin1996} it follows that the distribution of $D_t$ on $(t,\infty)$ is given by $\int_{(0,t)}\nu(dw-s)u(ds)$, where $u(ds)$ is the potential measure of the subordinator. Here we do not have to integrate at zero because we assume $\nu(0,\infty)=\infty$ or $b>0$, which implies that $u(\{0\})=0$. Thus
     \begin{align}
         \E^{(x,0)}D_t&=\int_{[t,\infty)} w\pr^{(x,0)} \l D_t \in dw \r \\
         &= t\,\pr^{(x,0)}(D_t=t)+\int_{(0,t)} \int_{(t,\infty)} w\nu (dw-s) u(ds).\label{1603}
     \end{align}
     By the change of the variables in the inner integral of \eqref{1603}, we get
     \begin{align*}
        \E^{(x,0)}&D_t= t\,\pr^{(x,0)}(D_t=t)+ \int_{(0,t)} \int_{(t-s,\infty)} (h+s)\nu (dh) u(ds)\\
        &\le t\,\pr^{(x,0)}(D_t=t)+\int_{(0,t)} \int_{(0,\infty)} h\nu (dh) u(ds)+\int_{(0,t)} s\,\nu (t-s,+\infty) u(ds)\\
        &\le t\,\pr^{(x,0)}(D_t=t)+u((0,t])\int_{(0,\infty)} h\nu (dh)+t\,(1-\pr^{(x,0)}(D_t=t))\\
        &=t+u((0,t])\int_{(0,\infty)} h\nu (dh)<\infty.
     \end{align*}
     Here, the finiteness of the integral $\int_{(0,\infty)} h\nu (dh)$ follows from the fact that $\nu$ is a L\'evy measure, and from the assumption \eqref{1610}. Further, $u((0,t])=\int_{(0,t]}u(ds)$ is the renewal function of the subordinator $\sigma$ and as such it is finite, see \cite[Chapter III]{bertoin1996}. Further, it holds that
     \begin{align*}
         \int_{(0,t)}\nu (t-s,+\infty) u(ds)=(1-\pr^{(x,0)}(H_t=t))=(1-\pr^{(x,0)}(D_t=t))
     \end{align*}
     which follows by \cite[Proposition 2 in Chapter III]{bertoin1996}.

     In other words, the second integral in \eqref{418} is finite for each choice of $T_t$, and all $t>0$.
} 
    
    For the first integral in \eqref{418}, note that
    \begin{align}
        \int_1^{+\infty} {s^{2p-1}} h(s) ds \, \leq \, &\frac{d \Sigma^2}{2-2p}  + \int_1^{+\infty} s^{2p-1}\int_{|z|>s} \Pi(dz) ds \label{420} \\& + \int_1^{+\infty} s^{2p-3} \int_{|z| \leq s} |z|^2 \Pi(dz).
        \label{421}
    \end{align}
    The integral in \eqref{420} is convergent by \eqref{ass2p} so we are left to check the integral in \eqref{421}. We have that
    \begin{align}
        \int_1^{+\infty} s^{2p-3} \int_{|z| \leq s} |z|^2 \Pi(dz) ds \, = \, &\frac{1}{2-2p}\int_{\R^d} (1 \vee |z|)^{2p-2} |z|^2 \Pi(dz) \notag \\
        = \, & \frac{1}{2-2p} \l \int_{|z|<1} |z|^2 \Pi(dz) + \int_{|z|>1} |z|^{2p} \Pi(dz) \r,
        \label{424}
    \end{align}
    where the integrals in \eqref{424} are finite by the definition of L\'evy measure and the assumption \eqref{ass2p}.
    
    The second statement follows by using the very same argument and the assumption \eqref{ass3p}.
\end{proof}

\begin{remark}
    For example, suppose that $M_t$, $t\geq0$, is a one-dimensional isotropic stable process with exponent $\Psi(\xi)=|\xi|^\alpha$, $\alpha \in (0,2)$. Then the conditions \eqref{ass2p} and \eqref{ass3p} are satisfied, respectively, by $p<\alpha/2$ and $p< \alpha /3$. In other words, in the case of the isotropic stable process $M_t$, we can indeed consider exploding functionals $u$ in Theorem \ref{thm1541}.
\end{remark}

\begin{remark}\label{complex_MC}
    Theorem \ref{thm1541}, allows us to state the following convergence in distribution
    \begin{equation*}
        \sqrt{N}(q_N(t_1, \cdots, t_n)-q(t_1, \cdots, t_n,x))\xrightarrow{\text{d}}N(0,\sigma(t_1, \cdots, t_n,x)^2) \quad\text{as}\quad N\to\infty.
    \end{equation*}
    Equations \eqref{414} and \eqref{1552} provide an upper bound to $\sigma(t_1, \cdots, t_n,x)^2$. The asymptotic confidence interval is
    \begin{align}
        q_N(t_1, \cdots, t_n) \pm \frac{\sigma(t_1, \cdots, t_n,x)}{\sqrt{N}}\varphi(\alpha/2),
    \end{align}
    where $\varphi(\alpha)$ is $\alpha$-quantile of the standard normal distribution. Moreover, it is clear that for a desired tolerance error $\varepsilon$, $N$ has to satisfy
    \begin{align}
        N> \varepsilon^{-2}\sigma(t_1, \cdots, t_n,x)^2 \varphi(\alpha/2)^2.
    \end{align}
    Thus, the calculation of such estimator 
    has the expected complexity of
    $$\varepsilon^{-2}\sigma(t_1, \cdots, t_n,x)^2 \varphi(\alpha/2)^2\, (n(\mathcal{K}^{\gamma}(h)+7)+\mathfrak{M}+2)$$
    if $T_t$ is the inverse or the undershooting of a subordinator, or
    $$\varepsilon^{-2}\sigma(t_1, \cdots, t_n,x)^2 \varphi(\alpha/2)^2\, (n(\mathcal{K}^{\Gamma}(h)+4)+\mathfrak{M}+2)$$
    if $T_t$ is the overshooting of a subordinator. Here $\mathfrak{M}$ is the expected complexity of sampling the Markov process $M$ in times $T_{t_0},\dots,T_{t_n}$ (see Remark \ref{complex_alg3}), and $\mathcal{K}^{\gamma}$ and $\mathcal{K}^{\Gamma}$ are expected running times of algorithms called in Algorithm \ref{alg:2} (lines 4 and 11) and Algorithm \ref{alg:3} (line 6), respectively, where $h=\max\{t_k-t_{k-1}:k=1,\dots,n\}$ denotes the maximum step.
\end{remark}

\section{Approximate simulation of time-changed processes}\label{secapproxdiff}
In this section, we deviate from the preceding one by considering Feller processes for which exact sampling is not necessarily available. {First, in Theorems \ref{strerror1} and \ref{thm:1714}, we bound the strong error of the approximation of a time-changed process, and we calculate the complexity of obtaining a \textit{good enough} approximation. Then, in Subsection \ref{sec:MCATTCD} we deal with Monte-Carlo approximations of functionals of these time-changed Feller processes. The main results are Theorem \ref{thm:l2error} which gives the $L^2$-error of the Monte-Carlo approximation, and Theorem \ref{thm:biased} which is a central limit theorem.}

In particular, we focus on $d$-dimensional diffusions of the form
\begin{align}\label{sde}
    X_t = X_{t_0} + \int_{t_0}^t a(s, X_s) ds + \int_{t_0}^t b(s, X_s) dW_s,
\end{align}
where $X_{t_0}\in\R^d$ is the (possible random) initial condition and $W=(W_t,\, t\geq0)$ is an $m$-dimensional Wiener process, whose components are independent Wiener processes. The functions, $a:\R^+ \times\R^d \to\R^d$ and $b:\R^+ \times\R^d \to\R^{d\times m}$ are measurable functions so the {equations \eqref{sde}  are}, component by component,
\begin{align}\label{sde_components}
    X_t^i = X_{t_0}^i + \int_{t_0}^t a^i(s, X_s) ds + \sum_{j=1}^m \int_{t_0}^t b^{i,j}(s, X_s) dW_s^j,\quad i=1,\dots,d.
\end{align}
It is well-known that the stochastic differential equation~\eqref{sde} has a unique (strong) solution whenever $X_{t_0}$ has a finite second moment and the coefficients satisfy the standard Lipschitz conditions, see \citep{kloeden1992, Protter2004}.

For a suitably time-changed diffusion process, $Y_t\coloneqq X_{T_t}$, stochastic calculus was developed in \cite{Kobayashi2011, Kobayashi2012, Kobayashi2016, Kobayashi2019}. This is deeply connected with \eqref{sde} by the so-called duality principle \citep[Theorem 4.2]{Kobayashi2011}, which states that if $X_t$ is the unique solution to \eqref{sde}, then the time-changed process $X_{T_t}$ is the unique solution to
\begin{align}\label{sde-time-changed}
    Y_t = X_{t_0} + \int_{t_0}^t a(T_s, Y_s) dT_s + \int_{t_0}^t b(T_s, Y_s) dW_{T_s}.
\end{align} 
The theory of equations having the form \eqref{sde-time-changed} is valid for suitable time-changes that are in the so-called synchronization, see \cite{Kobayashi2011} for details. In our case the connection with the theory of \eqref{sde-time-changed} is clear in the case $T_t=L_t$, i.e. when the time-change is the inverse of a subordinator (independent from the diffusion). On the other hand, the case when the time-change is the undershooting $H_t$ is not covered by the theory in \cite{Kobayashi2011} since $H_t$ is a left-continuous process (and so is $M_{\sigma^0_{L_t-}}$). In view of the independence between the diffusion and the random time process, the sampling of two components $(M_t,T_t)$ can be done separately as explained in the previous section. However sometimes it happens that the diffusion cannot be sampled exactly. In this case several approximation methods exist. The more classical ones are the Euler-Maruyama method, its semi-implicit version or the Milstein schemes, but many others exist (see, e.g., \cite{elena} for a scheme in the case when the SDE has a distributional drift). With a scheme at hand one can approximate the diffusion part to obtain an approximation of the time-changed process. Hence, in our context, the representation \eqref{sde-time-changed} is not needed to sample the time-changed process nor to evaluate the approximation error.

We focus on evaluating the approximation error when the diffusion process is sampled using the Euler-Maruyama method. In the literature, the approximation error results are known also for different schemes. For example, the semi-implicit version of the Euler-Maruyama method has been studied in \citep{Deng2020}, but only in the case when the time-change is performed with an independent inverse subordinator, although for very general drift coefficients. This method, however, besides the error induced by the approximation of the diffusion includes also the error induced by the approximation method for the paths of the inverse subordinator. Further, the authors of \citep{Deng2020} evaluate the strong convergence order (in the sense of $\E\ls|Y_t-Y_t^h|^2\rs$) which is of the same order as we get in Theorem \ref{strerror1}, cf. Theorems 3.1 and 3.2 in \citep{Deng2020}. Also, the Milstein method, which improves the Euler-Maruyama scheme but requires stronger assumptions on the diffusion coefficients, has been studied for time-changed diffusion via inverse subordinator, see \citep{Kobayashi2021} and also \cite{Liu2023} where a similar equation is considered. In these papers, the authors used an approximation method for the trajectories of inverse subordinators, too. Clearly, approximation of the paths of the random time introduces a further error that with our method is avoided. For a more elaborate statement, see Remark \ref{remarkerror} below. 

Without loss of generality, from now on we assume that $t_0=0$. The Euler-Maruyama approximation of the solution to SDE~\eqref{sde} with the  time-steps $0= t_0<t_1<t_2<t_3<\dots$, where $\lim_nt_n= \infty$ is given by \citep[Section 10.2]{kloeden1992}:
\begin{align}
    X_{t_{n+1}}^h = X_{t_n}^h + a(t_n, X_{t_n}^h) (t_{n+1} - t_n) + b(t_n, X_{t_n}^h) (W_{t_{n+1}} - W_{t_n}),
    \label{emapprox}
\end{align}
where $n\in\N\cup\{0\}$, and $X_{t_0}^h=X_{t_0}$, and $h$ denotes the maximum time step $h=\sup_{n\in\N}(t_n-t_{n-1})$. We will denote the interpolation of $X^h_{t_n}$, $n\in\N$, as $X_t^{i,h}$, i.e. component by component we have
\begin{align}
    X_t^{i,h} = X_{t_n}^{i,h}+ a^i(t_n, X_{t_n}^h) (t-t_n) + \sum_{j=1}^m b^{i,j} (t_n, X_{t_n}^h) (W_{t}^j - W_{t_n}^j), \quad t \in [t_n, t_{n+1}].
    \label{EM}
\end{align}
In essence, the Euler-Maruyama approximation $X_t^h$ depends on all the grid $\{t_i:i\in\N\}$, and not only on the maximum step size $h$, but our results will be uniform for all grids with the prescribed maximum step size $h$, so we proceed with this slightly abbreviated notation.

Usually, for ordinary differential equations, one evaluates the error as $|X_t-X_t^h|$. Accordingly, for a stochastic differential equation, the classical approach is to evaluate the strong error as $\E\ls|X_t-X_t^h|\rs$ which is the straightforward generalization of the deterministic case, see \citep[Section 9.6]{kloeden1992}. Instead, we adopt here the stronger criterion $\E\Big[\sup_{s\in[0,t]}|X_s-X_s^h|^2\Big]$ that bounds the strong error $\E|X_t-X_t^h|$ as follows
\begin{align}\label{criterion}
    (\E[|X_t-X_t^h|])^2 \leq \Big(\E\Big[\sup_{s\in[0,t]}|X_s-X_s^h|\Big]\Big)^2 \leq \E\Big[\sup_{s\in[0,t]}|X_s-X_s^h|^2\Big].
\end{align}
Here, and in the rest of Section \ref{secapproxdiff}, the probability measure $\mathds{P}$ (and the corresponding expectation $\E$) is from the underlying probability space on which the solution to SDE \eqref{sde} lives, and we adopt the convention that here the (independent) subordinator $\sigma$ is such that $\sigma_0=0$.

In \citep[Propositions 3.1 and 3.2]{Kobayashi2016} it is proved that, when $a$ and $b$ satisfy suitable assumptions, the approximation ~\eqref{EM} converges with the strong order of $\min{(\gamma,1/2)}$ (and with the weak order of 1) to the exact solution. The following lemma provides the explicit constants involved in the strong error of the approximation of the solution to~\eqref{sde} by the Euler-Maruyama method. Since the proof is an adaptation of the method from \cite{Kobayashi2016}, we postpone it to the Appendix. Before we present the {explicit} constants, we bring the assumptions on the SDE \eqref{sde} that will be used throughout the section.

\begin{assumption}{A}{}\label{A}
    The initial condition of the SDE \eqref{sde} satisfies $\E|X_{0}|^2<\infty$, while the drift and the diffusion coefficients satisfy
    \begin{align}
        |a^i(t,x)-a^i(t,y)| + |b^{i,j}(t,x)-b^{i,j}(t,y)| &\leq K|x-y|, \label{Lipschitz}\\
        |a^i(t,x)| + |b^{i,j}(t,x)| &\leq K(1+|x|), \label{Linear} \\
        |a^{i}(s,x)-a^i(t,x)| + |b^{i,j}(s,x)-b^{i,j}(t,x)| &\leq K (1+|x|) |s-t|^\gamma
    \label{tdiversi}
    \end{align}
    for all $x,y\in\R^d$, $s,t\geq0$, where $|\cdot|$ denotes the Euclidean norms in appropriate dimensions, and the constants $K>0$ and $\gamma>0$ do not depend on $i,j$.
\end{assumption}

\begin{lemma}\label{constants}
    Let \ref{A} hold and suppose that the Euler-Maruyama approximation $X^h_t$ is defined using an arbitrary grid $\{t_i:\, i\in\N\}$ with the maximum step size $h\in(0,1)$. Then there exist $\mathcal{C}_1,\mathcal{C}_2>0$ such that the strong error of the Euler-Maruyama approximation is finite and satisfies
    \begin{align}
        \E\Big[ \sup_{0\leq s\leq t} | X_s - X_s^h |^2 \Big] \leq \mathcal{C}_1 e^{\mathcal{C}_2t} h^{\min{(2\gamma,1)}}, \quad t>0,\, h\in(0,1),\label{strongerrorkobayashi}
    \end{align}
    where 
    \begin{align*}
        \mathcal{C}_1 &=	A\Big( 1+\frac{\mathcal{C}_2}{\mathcal{C}_2-2c_1}\Big),\quad \mathcal{C}_2= 2d(3K^2((1+4C_1^2)m+1) +1),
    \end{align*}
    and 
    \begin{align*}
    A& =12dK^2 (1+(1+4C_1^2)m) (1+4(1+m^2) dK^2) \Big(1+\frac{c_{0,1}}{2c_1}\Big) ,\\
    C_1& =1.30693...,\enskip c_{0,1}= 1/2+\E[|X_{t_0}|^2], \enskip c_1= 4dK\Big( 1 + \frac{1}{2}mK\Big).
    \end{align*}
\end{lemma}

{It is noteworthy that Lemma~\ref{constants} is not a simple modification of the classical bound for the strong error of the Euler method \cite[Chapter 10]{kloeden1992}, which takes the form $Ce^{C(t^2+t)}$. In contrast, it is an improved version under the assumptions in~\ref{A}. For further details, refer to \cite[Remark 3.1]{Kobayashi2016}.}

Furthermore, by using \cite[Proposition 3.2]{Kobayashi2016} it is also possible to evaluate the weak error, i.e.
\begin{align}
	|\E[u(X_t) -u(X_t^h)]| \leq h C e^{Ct}
\end{align}
for a suitable constant $C>0$. However, this estimates requires several conditions on the growth of the coefficients of the SDE and of the solution to the associated transition density.

\begin{remark}\label{remarkerror}
    In the following theorems, we provide the strong errors with explicit constants (which come from Lemma \ref{constants}). These results not only extend, but also enhance the results available in the existing literature for two reasons. First, there is an increased flexibility in options concerning a time-change. Second, Algorithms \ref{alg:2} and \ref{alg:3} return exact paths, whereas existing literature approximates both the diffusion and the time-change paths, thereby inducing two errors (see the proof of Theorem 3.1 in \citep{Kobayashi2016}):
    \begin{align}\label{eq:2034}
        \E\Big[ \sup_{0\leq s\leq t}|Y_s -Y_s^h|^2 \Big] \leq 2\E\Big[ \sup_{0\leq s\leq t}|X_{L_s} -X_{L_s^h}|^2 \Big] + 2\E\Big[ \sup_{0\leq s\leq t}|X_{L_s^h} -Y_{L_s^h}^h|^2 \Big].
\end{align}
    The first right-hand side term concerns the approximation of the inverse subordinator path, while the second term concerns the approximation of a diffusion process. Obviously, by using the exact sampling method for the inverse subordinator that we derived in Section \ref{secpathinv}, in our approximations there is no need for a division as in \eqref{eq:2034}, provided that one consider an exact simulation whose output can be controlled with machine precision under the L\'evy-Prokhorov metric, as explained in Section \ref{sec:prelim}.
\end{remark}

Here is the algorithm that we use to approximate the process $Y_t = X_{T_t}$.

\begin{algorithm}[H]\label{alg:5}
	\caption{generating vectors $(Y^h_{t_1}, \cdots, Y^h_{t_n})$}
	\KwData{
		 $0=t_0 <t_1 <\cdots <t_{n-1} <t_n=t$, $n\in\N$, $x_0 \in\R^d, h\in(0,1)$,
	}
    Sample $T_{t_i}$ for $i=1,\cdots,n$, with Algorithm \ref{alg:2} or \ref{alg:3} \\
    Sample $X_s^h$ on the grid $\{0,h,2h,3h,\dots\}\cup\{T_{t_1},\dots,T_{t_n}\}$ with the scheme \eqref{EM} \\
    $Y_{t_i}^h \gets X_{T_{t_i}}^h$ for $i=1, \cdots, n$
\end{algorithm}
\noindent We note that the sampling in line 2 of the previous algorithm finishes when it is achieved that for some $k\in\N$ we get $kh>T_{t_n}$.

\begin{theorem}\label{strerror1}
    Let be $Y_t = X_{T_t}$ where $X_t$ is the solution to~\eqref{sde-time-changed} satisfying \ref{A}. For $h\in(0,1)$, let $Y_s^h = X_{T_s}^h$,  $s\in \{ t_1,\cdots,t_n \}$ where $0=t_0 < t_1< \cdots < t_n=t$, be the approximation of $Y_s$ given in Algorithm \ref{alg:5}. Then, for $T_t$ being the undershooting, the strong error is finite and satisfies
    \begin{align}
        \E\Big[ \max_{s\in \{ t_1,\cdots,t_n \}} |Y_s - Y_s^h|^2 \Big] \leq \mathcal{C}_1 e^{ \mathcal{C}_2 t  }h^{\min(2\gamma,1)}.
    \end{align}
    Additionally, for $T_t$ being the inverse of a subordinator with the Laplace exponent $\phi$, the strong error is finite and satisfies
    \begin{align}
    \E\Big[ \max_{s\in \{ t_1,\cdots,t_n \}} |Y_s - Y_s^h|^2 \Big] \leq \mathcal{C}_1 \Big( 1+ \frac{ e^{\mathcal{C}_3 t}}{\phi(\mathcal{C}_3)-\mathcal{C}_2 } \Big)  h^{\min(2\gamma,1)},
	\end{align}
    where $\mathcal{C}_3>0$ is any constant such that $\phi(\mathcal{C}_3)>\mathcal{C}_2$. 
\end{theorem}

\begin{proof}
    By Lemma \ref{constants} and the independence between the $X_t$ and $T_t$, we observe that
    \begin{align}\label{2339}
        \E\Big[ \max_{s\in \{ t_1,\cdots,t_n \}} |Y_s - Y_s^h|^2\Big] &= \E\Big[ \max_{s\in \{ t_1,\cdots,t_n \}}  |X_{T_s} - X_{T_s}^h|^2 \Big]\\
        &\leq\E\Big[ \max_{s\in \{0,h,2h,3h,\dots\}\cup\{T_{t_1},\dots,T_{t_n}\}}  |X_s - X_s^h|^2 \Big]\\
        &\leq \E\Big[ \sup_{0\leq s\leq T_t} |X_s - X_s^h|^2 \Big]\\
        &\leq \mathcal{C}_1 \E[e^{ \mathcal{C}_2 T_t}]\, h^{\min(2\gamma,1)}.\label{eq:1228}
    \end{align}

    Now, when $T_t$ is the undershooting, we use that $T_t\leq t$ to obtain the claim.
    
    If, on the other hand, $T_t$ is the inverse of a subordinator, by Lemma \ref{our} for any $\mathcal{C}_3>0$ such that $\phi(\mathcal{C}_3) >\mathcal{C}_2$ it holds
    \begin{align}
        \E[e^{\mathcal{C}_2 L_t}] \leq 1+ \frac{ e^{\mathcal{C}_3 t}}{\phi(\mathcal{C}_3)-\mathcal{C}_2 },
	\label{momexpl}
    \end{align}
    which gives the claim.
\end{proof}

The time-change with the overshooting requires a separate statement since the strong error is finite under some more restrictive assumptions. Recall that the potential measure $u$ of the subordinator $\sigma$ is defined as $u(A)\coloneqq \E[ \int_0^\infty \1_A(\sigma_x)dx ]=\int_A u(ds)$, $A\in \mathcal{B}([0,\infty))$. Now we have the statement for the strong error of the time-changed process with the overshooting.

\begin{theorem}\label{thm:1714}
    Let $Y_t= X_{D_t}$ where $X_t$ is the solution to~\eqref{sde-time-changed}  satisfying \ref{A}. For $h\in(0,1)$, let $Y_s^h= X_{D_s}^h$, $s\in \{ t_1,\cdots,t_n \}$ where $0=t_0 < t_1< \cdots< t_n=t$, be the approximation of $Y_t$ as in Algorithm \ref{alg:5}. If $M(\mathcal{C}_2)<\infty$, where
    \begin{align}
        \int_1^{+\infty} e^{c s} \nu(ds) \eqqcolon M(c),
    \label{lunchtime}
    \end{align}
    then the strong error is finite and it holds that
    \begin{align}
        \E\Big[ \max_{s\in \{ t_1,\cdots,t_n \}}  |Y_s - Y_s^h|^2 \Big] \leq \mathcal{C}_1 e^{ \mathcal{C}_2 t } \Big( e^{\mathcal{C}_2\max\{1,t\}}+ M(\mathcal{C}_2) u\big((0,t]\big) \Big) h^{\min(2\gamma,h)}.
    \end{align}
\end{theorem}

\begin{proof}
    The proof can be conducted similarly to the proof of Theorem \ref{strerror1}. We have
    \begin{align}
        \E\Big[ \max_{s\in \{ t_1,\cdots,t_n \}}  |X_{D_s} - X_{D_s}^h|^2 \Big] &\leq \E\Big[ \sup_{0\leq s\leq D_t} |X_s - X_s^h|^2 \Big]\\
        &\leq \mathcal{C}_1 \E[e^{ \mathcal{C}_2 D_t }] h^{\min(2\gamma,1)}.\label{eq:1227-a}
    \end{align}
    We are left to bound $\E[e^{ \mathcal{C}_2 D_t }]$. Recall that $D_t$ on $(0,t)$ has the distribution $\int_{(0,t)}\nu(dw-s)u(ds)$ and that $\int_{(0,t)}\nu (t-s,+\infty) u(ds)=(1-\pr (H_t=t))=(1-\pr (D_t=t))$, see \cite{bertoin1996} and the comments in the proof of Proposition \ref{proplevy}. We have
    \begin{align}
        \E[e^{\mathcal{C}_2 D_t}]=  e^{\mathcal{C}_2 t} \pr (D_t=t) + \int_{(0,t)}\int_{(t,+\infty)} e^{\mathcal{C}_2 x} \nu(dx-s) u(ds).
    \label{1250}
    \end{align}
    By Fubini's theorem and by the change of variables in the second integral we get
    \begin{align}
        \E[e^{\mathcal{C}_2D_t}] &= e^{\mathcal{C}_2 t} \pr (D_t=t) +\int_{(0,t)} e^{\mathcal{C}_2s}\int_{(t-s,+\infty)} e^{\mathcal{C}_2z} \nu(dz) u(ds)\\
        &\leq e^{\mathcal{C}_2 t} \pr (D_t=t) + e^{\mathcal{C}_2t}\int_{(0,t)} \int_{(t-s,\infty)} e^{\mathcal{C}_2 z} \nu(dz) u(ds).\label{1703}
    \end{align}
    To bound the integral in \eqref{1703}, we distinguish two cases: $t\leq1$ and $t>1$. If $t\leq1$, then for the integral in \eqref{1703}  we have
    \begin{align}
        \int_{(0,t)} \int_{(t-s,\infty)} e^{\mathcal{C}_2 z} \nu(dz) u(ds)&= \int_{(0,t)} \int_{(t-s,1)} e^{\mathcal{C}_2 z} \nu(dz) u(ds)\nonumber\\
        &\hspace{4em}+\int_{(0,t)} \int_{[1,\infty)}e^{\mathcal{C}_2 z} \nu(dz) u(ds)\nonumber\\
        &\leq e^{\mathcal{C}_2}\int_{(0,t)}\nu (t-s,+\infty) u(ds) + u\big((0,t]\big)M(\mathcal{C}_2)\nonumber\\
        &= e^{\mathcal{C}_2}(1-\pr (H_t=t)) + u\big((0,t]\big)M(\mathcal{C}_2).\label{1949a}
    \end{align}
    Thus, by collecting the bounds \eqref{1703} and \eqref{1949a} we obtain
    \begin{align}\label{1755a}
        \E[e^{\mathcal{C}_2D_t}]\leq e^{\mathcal{C}_2 t}\l e^{\mathcal{C}_2}+ M(\mathcal{C}_2)u\big((0,t]\big)\r,\quad t\leq1.
    \end{align}

    For $t>1$, we have a similar calculation where we separate the integral in \eqref{1703} in the following way:
    \begin{align}
         \int_{(0,t)} \int_{(t-s,\infty)} e^{\mathcal{C}_2 z} \nu(dz) u(ds)&= \int_{(0,t)} \int_{(t-s,t)} e^{\mathcal{C}_2 z} \nu(dz) u(ds)\nonumber\\
         &\hspace{4em}+\int_{(0,t)} \int_{[t,\infty)}e^{\mathcal{C}_2 z} \nu(dz) u(ds)\nonumber\\
        &\le e^{\mathcal{C}_2 t}\int_{(0,t)}\nu (t-s,+\infty) u(ds)+u\big((0,t]\big)M(\mathcal{C}_2)\nonumber\\
        &=e^{\mathcal{C}_2t}(1-\pr (H_t=t))+u\big((0,t]\big)M(\mathcal{C}_2).\label{1949b}
    \end{align}
    Thus, by collecting the bounds \eqref{1703} and \eqref{1949b} we get
    \begin{align}\label{1755b}
        \E[e^{\mathcal{C}_2D_t}]\leq e^{\mathcal{C}_2 t} ( e^{\mathcal{C}_2t}+ M(\mathcal{C}_2)u\big((0,t]\big) ),\quad t>1.
    \end{align}

    By combining bounds in \eqref{1755a} and \eqref{1755b}, we obtain the claim of the proof.
\end{proof}

\begin{remark}
    The Algorithm \ref{alg:5} can be conducted with any grid, not just with the one with the points $\{0,h,2h,3h,\dots\}\cup\{T_{t_1},\dots,T_{t_n}\}$, as long as the maximum step size is bounded by $h$, and the corresponding bounds on the strong error of approximation of $Y_t$, proved in Theorems \ref{strerror1} and \ref{thm:1714} remain the same. This holds since the strong error of Euler-Maruyama is uniformly bounded among all grids with the maximum step $h$, proved in Lemma \ref{constants}.
    
\end{remark}

\begin{proposition}\label{complex_EM}
    Assume \ref{A} and let $\mathcal{K}^{\gamma}$ and $\mathcal{K}^{\Gamma}$ denote the expected running times of the algorithms used in Algorithm \ref{alg:2} (lines 4 and 11) and Algorithm \ref{alg:3} (line 6), respectively. For a desired tolerance $\varepsilon >0$ on the strong error of the approximation of $Y_s$ by $Y^h_s$ on $s\in\{t_0,t_1,\dots,t_n\}$, with $h=\max\{t_k-t_{k-1}:k=1,\dots,n\}$, the expected complexity is 
    \begin{align}
        (\mathcal{K}^{\gamma}(h)+8)n+ \E\, T_{t_n}\l \frac{\mathcal{C}_1 \E[e^{ \mathcal{C}_2 T_{t_n}}] }{\varepsilon} \r^{\max(1,\frac{1}{2\gamma})}+1,\label{eq:1240-a}
    \end{align}
    if $T_t=H_t$, and
    \begin{align}
        (\mathcal{K}^{\Gamma}(h)+5)n+ \E\, T_{t_n}\l \frac{\mathcal{C}_1 \E[e^{\mathcal{C}_2 T_{t_n} }] }{\varepsilon} \r^{\max(1,\frac{1}{2\gamma})}+1,\label{eq:1240-b}
    \end{align}
    if $T_t=L_t$ or $T_t=D_t$ (with the additional assumption $M(\mathcal{C}_2)<\infty$ defined in \eqref{lunchtime} in this case).

    If Algorithm \ref{alg:simplif_2} is used in Algorithm \ref{alg:2} (lines 4 and 11) and if Algorithm \ref{alg:simplif} is used in Algorithm \ref{alg:3} (line 6), the complexity of the described approximation of $Y_s$ by $Y^h_s$ (with the desired tolerance on the strong error) has finite moments of all orders.
\end{proposition}

\begin{proof}
    The Algorithm \ref{alg:5} first samples the time-change $T_{t_i}$, $i=1,\dots,n$, which has the expected cost of $(\mathcal{K}^{\gamma}(h)+7)n$ in the case $T_t=H_t$, and $(\mathcal{K}^{\Gamma}(h)+4)n$ in the case $T_t=L_t$ or $T_t=D_t$, see Proposition \ref{complex_alg3}.
    
    Then comes the sampling of the Euler-Maruyama scheme on the grid $\{0,h,2h,3h,\dots\}\cup\{T_{t_1},\dots,T_{t_n}\}$. If $\varepsilon>0$ is the desired error tolerance on the paths of $X_{T_t}^h$, then $h$ has to satisfy
    \begin{align}\label{1441}
        \mathcal{C}_1 \E[e^{\mathcal{C}_2 T_t }] h^{\min(2\gamma,1)} \leq \varepsilon,
    \end{align}
    see Theorem \ref{strerror1} and \eqref{eq:1228}, and Theorem \ref{thm:1714} and \eqref{eq:1227-a}.
    It is clear that by choosing $h$ small enough, this can be achieved for any $\varepsilon>0$. For $h$ such that the equality in \eqref{1441} holds, the number of grid points for the Euler-Maruyama scheme is $N=\lceil T_t/h\rceil+n$, and they all have the cost of one since they only consist of elementary operations.  By summing up those costs, and by noting that 
    \begin{align*}
        N=\lceil T_t/h\rceil{+n}\leq T_t/h+1+{n}= T_t\l\frac{\mathcal{C}_1 \E[e^{ \mathcal{C}_2 T_t}] }{\varepsilon}\r^{\max(1,\frac{1}{2\gamma})}+1{+n}
    \end{align*}
    the statement on the complexity follows.

    Since $\E[e^{\mathcal{C}_2 T_t}]<\infty$, in all the cases, the last claim follows from Corollaries \ref{cor:complexDassios} and \ref{cor:complexCLM}.
\end{proof}

\subsection{Monte Carlo approximations of time-changed diffusion processes}\label{sec:MCATTCD}
We consider now the Monte Carlo estimator for $\E[u(Y_{t_1}, \cdots, Y_{t_n} )]$ where $Y_t=X_{T_t}$ is a time-changed diffusion process as above and $u$ is a suitable locally bounded functional. The estimator is given by
\begin{align}\label{MCEM}
   q_N^h(t_1, \cdots, t_n)\coloneqq  \frac{1}{N} \sum_{k=1}^N u \Big( X_{T_{t_1}^k}^{h,k}, \cdots, X_{T_{t_n}^k}^{h,k} \Big),
\end{align}
{where  $\Big( X_{T_{t_1}^k}^{h,k}, \cdots, X_{T_{t_n}^k}^{h,k} \Big)$, $k=1,\dots,N,$ are i.i.d. copies of the Euler-Maruyama approximation of the time-changed diffusion $Y_t$, i.e. of the vector given in Algorithm \ref{alg:5}. The approximation \eqref{MCEM}} is the counterpart of \eqref{monteest} in the case when an exact simulation of the Markov process is not at hand. Here we evaluate the error in the approximation.

\begin{theorem}\label{thm:l2error}
    Assume \ref{A} and let $u: \R^{d\times n}\to\R$ be such that $|u(x_1,\cdots,x_n)-u(y_1,\cdots,y_n)|\leq C_u\max_i|x_i-y_i|$, for a positive constant $C_u$ and $x_i,y_i\in \R^d$, $i=1,\dots,n$. Then
    \begin{align}\label{1222}
        &\E\ls q_N^h(t_1, \cdots, t_n) - \E[u(X_{T_{t_1}}, \cdots, X_{T_{t_n}})] \rs^2 \leq \frac{v}{N} + C_u\sqrt{\mathcal{C}_1 \E[e^{\mathcal{C}_2T_{t_n}}] h^{\min(2\gamma,1)}},
    \end{align}
    where
    \begin{align}\label{1155}
        v=\sup_{\substack{|x_i|\leq x_0 \\ \text{for any } i}} |u(x_1, \cdots, x_n)|^2 +4Cn\max(\mathcal{C}_1,c_{0,1}) \E[e^{\mathcal{C}_2 T_{t_n}}],
    \end{align}
    where $C=C(u,x_0)>0$.

    Furthermore, the expectation in \eqref{1155} can be bounded as in Theorems \ref{strerror1} and \ref{thm:1714}, where  for the case $T_t=D_t$ we additionally assume $M(\mathcal{C}_2)<\infty$ defined in \eqref{lunchtime}.
\end{theorem}

\begin{proof}
    First note that
    \begin{align}
        \E[ u(X_{T_{t_1}}, \cdots, X_{T_{t_n}}) ]^2 < +\infty.
	\label{momsec}
    \end{align}
    
    Indeed, by argumentation as in Theorem \ref{thm1541} for \eqref{414} (since  $u \stackrel{\infty}{=} O (\max_{1\leq i\leq n} |x_i|)$ by the Lipschitz assumption on $u$) we get
    \begin{align}
        &\E[ u(X_{T_{t_1}}, \cdots, X_{T_{t_n}}) ]^2 \leq \sup_{\substack{|x_i|\leq x_0 \\ \text{for any } i}} |u(x_1, \cdots, x_n)|^2 +C\E\Big[ \max_{1\leq i\leq n} |X_{T_{t_i}}|^2 \Big] \\
        &\leq \sup_{\substack{|x_i|\leq x_0 \\ \text{for any } i}} |u(x_1, \cdots, x_n)|^2 +C\sum_{i=1}^n\int_0^\infty \E[|X_{w}|^2] \pr(T_{t_i} \in dw)\\
        &\leq \sup_{\substack{|x_i|\leq x_0 \\ \text{for any } i}} |u(x_1, \cdots, x_n)|^2 +Cnc_{0,1} \E[e^{2 c_1 T_{t_n}}],
    \end{align}
    where in the last line we used Lemma \ref{difference} and the fact that $T_t$ is non-decreasing. This proves the finiteness of the first two moments, where for the case $T_t=D_t$ we must note that $2c_1\le \mathcal{C}_2$ so $\E[e^{2c_1 T_{t_n}}]<\infty$ by the assumption $M(\mathcal{C}_2)<\infty$.
    
    Note that 
    \begin{align}
       & L_e^2 \coloneqq \E\ls q_N^h(t_1, \cdots, t_n) - \E[ u(X_{T_{t_1}}, \cdots, X_{T_{t_n}}) ] \rs^2 \notag \\
       = \,& \E\ls q_N^h(t_1, \cdots, t_n) \pm \E[ u(X_{T_{t_1}}^h, \cdots, X_{T_{t_n}}^h) ] - \E[ u(X_{T_{t_1}}, \cdots, X_{T_{t_n}}) ] \rs^2 \notag \\
       = \, & \frac{1}{N}\var[u( X_{T_{t_1}}^h, \cdots, X_{T_{t_n}}^h )] + \l \E[ u(X_{T_{t_1}}^h, \cdots, X_{T_{t_n}}^h ) - u(X_{T_{t_1}}, \cdots, X_{T_{t_n}} ) ] \r^2 \notag \\
       \leq \, & \frac{1}{N}\E[ u(X_{T_{t_1}}^h, \cdots, X_{T_{t_n}}^h )^2 ] + \l \E |u(X_{T_{t_1}}^h, \cdots, X_{T_{t_n}}^h ) - u(X_{T_{t_1}}, \cdots, X_{T_{t_n}})| \r^2.
       \label{here}
    \end{align}
    Under the assumption $|u(x_1,\cdots,x_n)-u(y_1,\cdots,y_n)|\leq C_u\max_i|x_i-y_i|$, the second term on the right-hand side can be bounded by the strong error as in Theorem \ref{strerror1}, as follows
    \begin{align}
        & \E |u(X_{T_{t_1}}^h, \cdots, X_{T_{t_n}}^h) - u(X_{T_{t_1}}, \cdots, X_{T_{t_n}})| \,\leq C_u \E\max_{s\in\{T_{t_1},\dots,T_{t_n}\}} |X_s^h -X_s| \notag \\
        & \qquad \leq  C_u \E\sup_{0\leq s\leq T_t} |X_s^h -X_s|.\label{1726}
    \end{align}
    Therefore, by using \eqref{criterion} and Theorem \ref{strerror1} or Theorem \ref{thm:1714}, the second right term in \eqref{here} is bounded by $C_u\sqrt{\mathcal{C}_1 \E[e^{\mathcal{C}_2T_{t_n}}] h^{\min(2\gamma,1)}}$. 

    Further, for $X_w^h$ we have
    \begin{align}
        \sqrt{\E|X_w^h|^2} &\leq \sqrt{\E|X_w^h-X_w|^{2}}+\sqrt{\E|X_w|^2}\\
        &\leq \sqrt{\mathcal{C}_1 e^{\mathcal{C}_2 w}}+\sqrt{c_{0,1}e^{c_1 w}} \leq 2\sqrt{\max(\mathcal{C}_1,c_{0,1})}e^{\frac{\mathcal{C}_2}{2} w},\label{51605}
    \end{align}
    where we used Lemma \ref{difference} and Lemma \ref{constants}, and that $2c_1\le \mathcal{C}_2$. 
    Hence, by the same reasoning as in the beginning of the proof, we obtain
    \begin{align}
        \E[u(X_{T_{t_1}}^h, \cdots, X_{T_{t_n}}^h )^2] &\leq \sup_{\substack{|x_i|\leq x_0 \\ \text{for any } i}} |u(x_1, \cdots, x_n)|^2 +4Cn\max(\mathcal{C}_1,c_{0,1})\E[e^{\mathcal{C}_2 T_{t_n}}].\label{boundvar}
    \end{align}
    By collecting these two bounds for terms in \eqref{here}, we get the claim.
\end{proof}

\begin{proposition}\label{complex_MCEM}
     Assume \ref{A} and fix the tolerance error $\varepsilon>0$.  {Let $\mathcal{K}^{\gamma}$ and $\mathcal{K}^{\Gamma}$ denote the expected running times of the algorithms used inside of Algorithm \ref{alg:2} (lines 4 and 11) and Algorithm \ref{alg:3} (line 6), respectively}. For the computation of the estimator \eqref{MCEM} with $L^2$-error smaller than $\varepsilon$, the {expected} complexity is bounded by
    \begin{align}\label{eq1306-a}
        &\frac{2v}{\varepsilon} \l (\mathcal{K}^{\gamma}(1/2)+8)n+ {\E T_{t_n}}\left(\frac{4C_u^2\mathcal{C}_1\E[e^{\mathcal{C}_2 T_{{t_n}}}]}{\varepsilon^2}\right)^{\max(1,1/2\gamma)}+1\r,
    \end{align}
    if $T_t=H_t$, and
    \begin{align}\label{eq1306-b}
        &\frac{2v}{\varepsilon} \l (\mathcal{K}^{\Gamma}(1/2)+8)n+ {\E T_{t_n}}\left(\frac{4C_u^2\mathcal{C}_1\E[e^{\mathcal{C}_2 T_{{t_n}}}]}{\varepsilon^2}\right)^{\max(1,1/2\gamma)}+1\r,
    \end{align}
    if $T_t=L_t$ or $T_t=D_t$ (with the additional assumption $M(\mathcal{C}_2)<\infty$ in this case, see \eqref{lunchtime}).
    The constants  $v$ and $C_u$ come from Theorem \ref{thm:l2error}, $\mathcal{C}_1$ and $\mathcal{C}_2$ from Lemma \ref{constants}.
    
    {If Algorithm \ref{alg:simplif_2} is used in Algorithm \ref{alg:2} (lines 4 and 11) and if Algorithm \ref{alg:simplif} is used in Algorithm \ref{alg:3} (line 6), the complexity of the calculation of the desired approximation as above has finite moments of all orders.}
\end{proposition}

\begin{proof}
    Fix $\varepsilon>0$. Choose $N\in\N$ and $h\in(0,1/2)$ such that $v/N<\varepsilon/2$ and $C_u\sqrt{\mathcal{C}_1\E[e^{\mathcal{C}_2T_t}] h^{\min(2\gamma,1)}}\leq \varepsilon/2$. In this way the bound  \eqref{1222} is smaller than $\varepsilon$. Therefore, $N$ must be bigger than $2v/\varepsilon$, and $h$ must be smaller than
    \begin{align}
        \left(\frac{\varepsilon^2}{4C_u^2\mathcal{C}_1\E[e^{\mathcal{C}_2 T_t}]}\right)^{\max(1,1/2\gamma)}.
    \end{align}
    
    Now the complexity of the estimator \eqref{MCEM} is given as $N$
    multiplied by the Euler-Maruyama complexity which is given in Proposition \ref{complex_EM} by using the time step $h$.
    
    The aforementioned substantiates the proposition claim.
\end{proof}

\subsubsection{Numerical examples}
In the following examples we consider diffusions for which the explicit solution is known, approximate it using Algorithm \ref{alg:5}, and compute the error
\begin{align}
    \max_{s\in \{ t_1,\cdots,t_n \}} |X_{{T}_s} - X_{{T}_s}^h|^2\label{31128}
\end{align}
for 500 independent trajectories. The results are shown in Figure \ref{fig:Error}.

\begin{example}
    We consider the time-changed (with inverse subordinator) Ornstein-Uhlenbeck process,
    \begin{align*}
        dY_t = \theta(\mu-Y_t) dL_t + \sigma dW_{L_t},\quad t\geq0,
    \end{align*}
    where $\theta$, $\sigma>0$ and $Y_0$, $\mu\in\R$, which is a subdiffusion and a special case of a fractional Pearson diffusion. For further details about these subdiffusions see \cite{Ascione2021, Leonenko2013}. In particular, let $L_t$ be the inverse of an $\alpha$-stable subordinator, $\alpha=0.8$, $t\in[0,0.1]$, $Y_0=0$, $\theta=\sigma=1/2$, and $\mu=1/4$, $K=5/8$. Fix the tolerance error $\varepsilon = 10^{-1}$ and use formula \eqref{1441} to get $h\leq 1.9\times10^{-4}$. In the bound \eqref{momexpl}, we have chosen $\phi(\mathcal{C}_3) - \mathcal{C}_2 = 6$. Figure \ref{fig:ErrorL} shows independent samples of the error \eqref{31128} for the desired tolerance error. These agree with the findings of Theorem \ref{strerror1}.
\end{example}

\begin{example}
    The process $X_t = X_0 (1+t) e^{\theta W_t - \theta^2 t/2}$, $t\geq0$, with $X_0,\,\theta\in\R\symbol{92}\{0\}$ solves the stochastic differential equation
    \begin{align*}
        dX_t = X_t(1+t)^{-1} dt + \theta X_t dW_t,
    \end{align*}
    see \cite[Chapter 4.4]{kloeden1992} for details. We time-change this diffusion with the undershooting process of an $\alpha$-stable subordinator, $\alpha=0.8$, in $t\in[0,0.1]$, and consider $X_0=1$ and $\theta=-1/5$. Further, we may choose $K=1+|\theta|$. Now we fix the tolerance error $\varepsilon = 10^{-1}$ and compute that $h\leq 8.4\times10^{-9}$  by using \eqref{1441}. By executing Algorithm \ref{alg:5}, the error \eqref{31128} is calculated. The results are shown in Figure \ref{fig:ErrorH}, which corroborate the findings of Theorem \ref{strerror1}.
\end{example}

\begin{figure}[ht]
    \centering
    \begin{subfigure}{.485\textwidth}
        \centering
        \includegraphics[width=.975\linewidth]{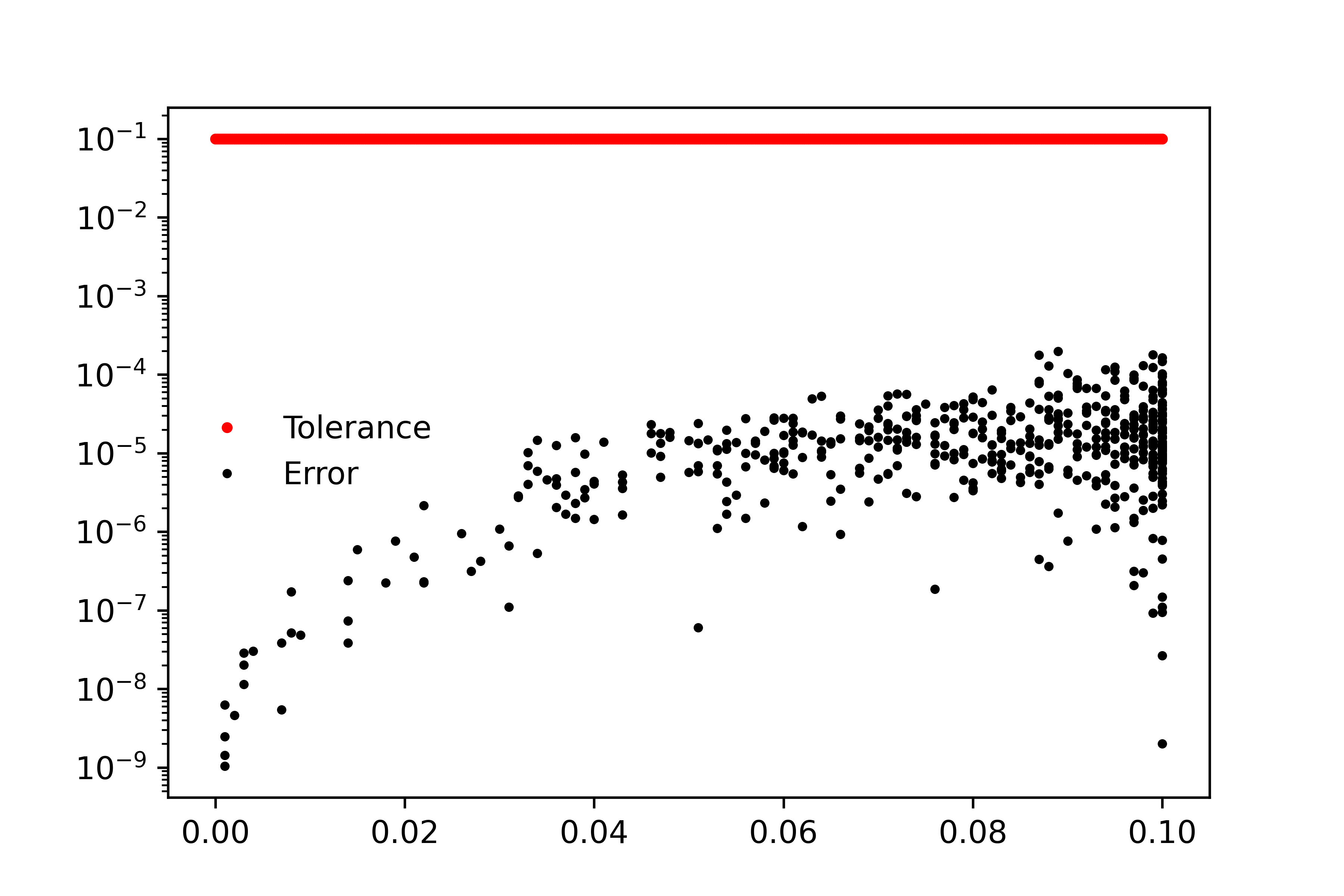}
        \caption{Ornstein-Uhlenbeck, time-changed with the inverse of an $\alpha$-stable subordinator.}
        \label{fig:ErrorL}
    \end{subfigure}\hfill%
    \begin{subfigure}{.485\textwidth}
        \centering
        \includegraphics[width=.975\linewidth]{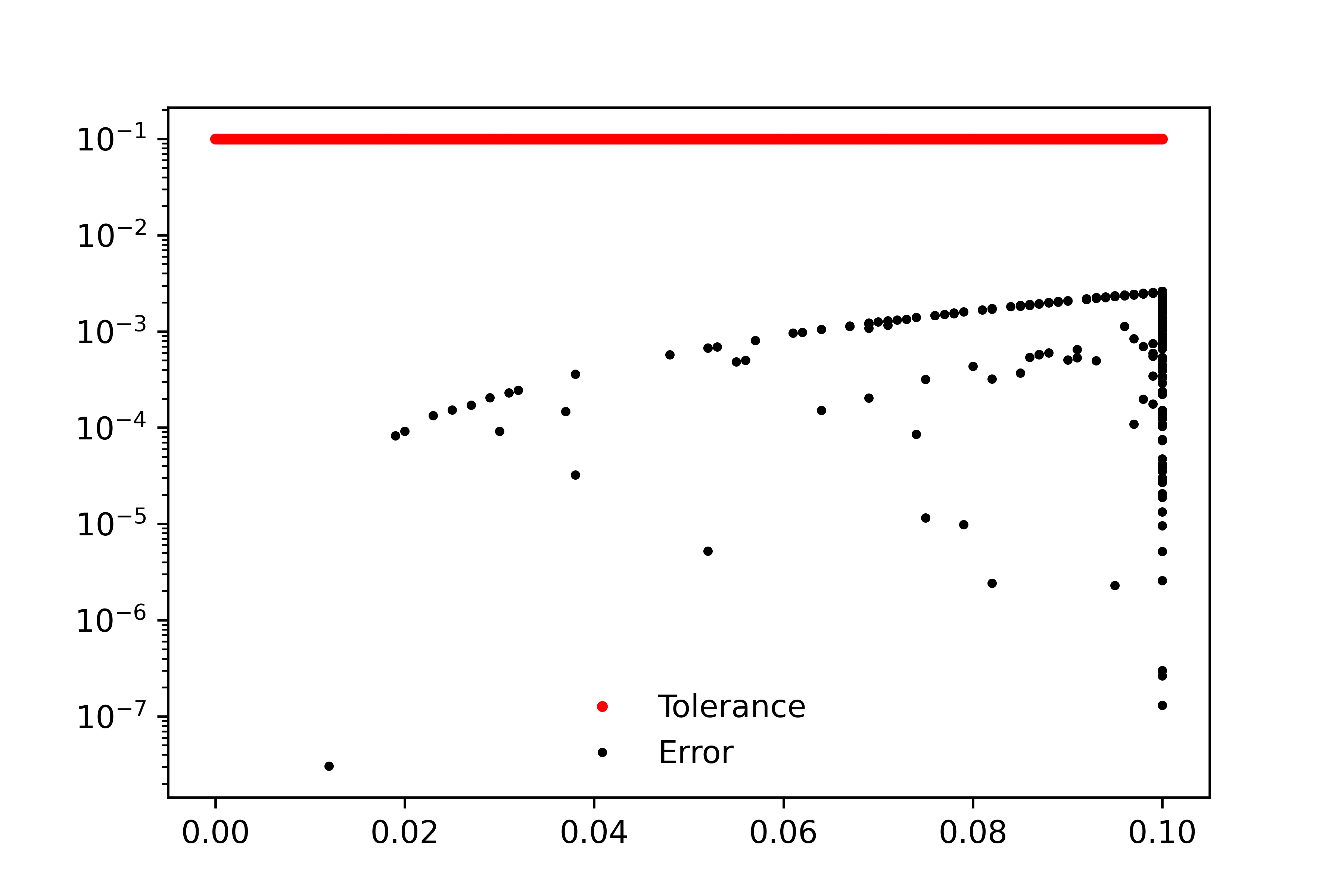}
        \caption{Diffusion, time-changed with the undershooting process of an $\alpha$-stable inverse subordinator.}
        \label{fig:ErrorH}
        \end{subfigure}
    \caption{500 independent error \eqref{31128} samples for two different diffusions, with a desired tolerance error $\varepsilon=10^{-1}$. Each point represents the maximum error obtained by a sample and its corresponding time point at which the error occurred.}
    \label{fig:Error}
\end{figure}

{The parameter $h$ in previous examples is computed using \eqref{1441} which holds true whenever the coefficients of the It\^o diffusion match the assumptions in \ref{A} (as it happens in the previous examples). Hence, the gap between the desired and observed errors in Figure \ref{fig:Error} appears big. By conducting again the computation in the appendix case by case, the inequality \eqref{1441} could be improved and the gap reduced.}

\subsubsection{The limit distribution}
Now we study a limit distribution of the Monte Carlo estimator $q_N^h(t_1, \cdots, t_n)$ defined in \eqref{MCEM}. { This is useful in applications, in order to study Gaussian approximations when $N \to +\infty$; however, the oscillation of $q_N^h(t_1, \cdots, t_n)$ must be around its true value and this is obtained when $h\to0$ and $N\to+\infty$ jointly, i.e., we put $h=h(N)\to0$ as $N\to+\infty$.}
Denote
\begin{align}
    Z= u(X_{T_{t_1}},\cdots,X_{T_{t_n}}) \text{ and }  Z_h^k= u(X_{T_{t_1}^k}^{h,k},\cdots,X_{T_{t_n}^k}^{h,k}),
\end{align}
where $X_{T_{t_1}^k}^{h,k},\cdots,X_{T_{t_n}^k}^{h,k}$ are independent samples of $X_{T_{t_1}}^{h},\cdots,X_{T_{t_n}}^{h}$. Define
\begin{align}
    \xi_{Nj} \coloneqq \frac{Z_h^j -\E[Z]}{\sqrt{N}}.
    \label{1455}
\end{align}
First we need to verify that for a suitable choice of $h=h(N)$ the random variables in \eqref{1455} form a null array.
According to \cite[p. 88]{Kallenberg}, a null array is a triangular array of random variables or vectors $\xi_{nj}$, $1\leq j\leq m_n$, $n\in\N$, such that the $\xi_{nj}$ are independent for each $n$ and satisfy $\xi_{nj}\to 0$, in probability, as $n\to+\infty$, uniformly in $j$.
\begin{lemma}
   Assume \ref{A} and $|u(x_1,\cdots,x_n)-u(y_1,\cdots,y_n)|\leq C_u\max_i|x_i-y_i|$, for $C_u>0$ and $x_i,y_i\in\R^d$. Set $h^{\min(2\gamma,1)}=N^{\delta}$, with $\delta<1$, and if $T_t=D_t$, additionally assume $M(\mathcal{C}_2)<\infty$ defined in \eqref{lunchtime}. Then the random variables \eqref{1455} form a null array. 
\end{lemma}

\begin{proof}
    Using Markov, triangular and Cauchy-Schwartz inequalities, we have
    \begin{align}
        \pr(|\xi_{Nj}|>\epsilon) &= \pr(|Z_h^j -\E[Z]|>\sqrt{N}\epsilon) \leq \frac{\E[|Z_h^j -\E[Z]|]}{\sqrt{N}\epsilon} \notag\\
        &\leq \frac{\E[|Z_h^j -\E[Z_h^j]|] + \E[|Z_h^j -Z|]}{\sqrt{N}\epsilon} \notag\\
        &\leq \frac{\sqrt{\var[Z_h^j]} + \E[|Z_h^j -Z|]}{\sqrt{N}\epsilon}.
    \end{align}
    Note that $\sqrt{\var[Z_h^j]}$ is bounded by using \eqref{boundvar} while $E[|Z_h^j -Z|]$ can be bounded as in \eqref{1726}. Therefore,
    \begin{align}
        \pr(|\xi_{Nj}|>\epsilon) \leq \frac{\sqrt{v} + C_u\sqrt{\mathcal{C}_1 \E[e^{\mathcal{C}_2T_t}] h^{\min(2\gamma,1)}}}{\sqrt{N}\epsilon} \stackrel{N\to\infty}{\longrightarrow} 0,
    \end{align}
    where $v$ is defined in \eqref{1155}.
\end{proof}

Now we can prove the central limit theorem.

\begin{theorem}\label{thm:biased}
    Assume \ref{A} and $|u(x_1,\cdots,x_n)-u(y_1,\cdots,y_n)|\leq C_u\max_i|x_i-y_i|$, for $C_u>0$ and $x_i,y_i\in\R^d$. Set $h=N^\delta$ for $\delta< -1$.  Let 
    \begin{align}
        S_N \coloneqq \frac{\sqrt{N} (q_N^h(t_1, \cdots, t_n)-q(t_1, \cdots, t_n))}{\sigma(t_1, \cdots, t_n , x)},
    \end{align}
    where $\sigma(t_1, \cdots, t_n)^2=\E[(Z-\E[Z])^{2}] $.  If $T_t$ is chosen to be the overshooting, additionally assume $M(\mathcal{C}_2)<\infty$ defined in \eqref{lunchtime}. The for any $\psi\in C_b(\R^d)$ it is true that
    \begin{equation*}
        \E[\psi(S_N)] \to \int_\R \psi(w) \frac{e^{-\frac{w^2}{2} }}{\sqrt{2 \pi}} \, dw, \quad\text{as}\quad N\to+\infty.
    \end{equation*}
\end{theorem}

To prove Theorem \ref{thm:biased}, we use \cite[Theorem 5.15]{Kallenberg} which states

\begin{theorem}[Theorem 5.15 in \cite{Kallenberg}]\label{thm:kallenger}
    Let $(\xi_{nj})$, $1\leq j\leq m_n$, $n\in\N$, be a null array of random variables, and let $N(b,c)$ be a normal distribution with mean $b$ and variance $c$. Then $\sum_{j=1}^{m_n} \xi_{nj}\to N(b,c)$ if and only if these conditions hold:
    \begin{enumerate}
        \item $\sum_{j=1}^{m_n} \pr(|\xi_{nj}|>\epsilon) \to0$ for all $\epsilon>0$ as $n\to+\infty$;
        \item $\sum_{j=1}^{m_n} \E[\xi_{nj}\1_{\{ |\xi_{nj}|\leq1\}}] \to b$ as $n\to+\infty$;
        \item $\sum_{j=1}^{m_n} \var[\xi_{nj}\1_{\{ |\xi_{nj}|\leq1\}}] \to c$ as $n\to+\infty$.
    \end{enumerate}
\end{theorem}

\begin{proof}[Proof of Theorem \ref{thm:biased}]
    The proof of our theorem consists of proving conditions (1), (2), and (3) of Theorem \ref{thm:kallenger} for $\xi_{Nj}$ defined in \eqref{1455}. 

Proof of condition (1).
    Note that
    \begin{align}
        \sum_{j=1}^N \pr(|\xi_{Nj}|>\epsilon) \leq \sum_{j=1}^N \pr(|Z_h^j-Z|>\sqrt{N}\epsilon/2) + \sum_{j=1}^N \pr(|Z-\E[Z]|>\sqrt{N}\epsilon/2).
    \end{align}
    In the same spirit as \cite[Equation~3.41]{kolokoltsov2021}, the second term on the right-hand side goes to zero as $N\to+\infty$:
    \begin{align}
        \sum_{j=1}^N \pr(|Z-\E[Z]|>\sqrt{N}\epsilon/2) \, \leq \, \sum_{j=1}^N \E\left[ \frac{4|Z-\E[Z]|^2}{N \varepsilon^2} \1_{\{ |Z-\E[Z]| > \frac{1}{2} \sqrt{N} \varepsilon \}}\right].
    \end{align}
    Using the Markov inequality, \eqref{1726}, and \eqref{criterion} and Theorems \ref{strerror1} and \ref{thm:1714}, we have
    \begin{align}\label{2155}
        \sum_{j=1}^N \pr(|Z_h^j-Z|>\sqrt{N}\epsilon/2) &\leq \sum_{j=1}^N \frac{\E[|Z_h^j-Z|]}{\sqrt{N}\epsilon/2} \leq  \sqrt{N}\frac{C_u\sqrt{\mathcal{C}_1 \E[e^{\mathcal{C}_2 T_t}] h^{\min(2\gamma,1)}}}{\epsilon/2}
    \end{align}
    which tends to $0$ as $N\to+\infty$.
    
Proof of condition (2).
    In the same fashion as above we have that
    \begin{align}\label{2251}
        \sum_{j=1}^N \E[\xi_{Nj}] &\leq \frac{1}{\sqrt{N}} \sum_{j=1}^N \E[|Z_h^j - Z|] \leq \sqrt{N} C_u\sqrt{\mathcal{C}_1 \E[e^{\mathcal{C}_2 T_t}] h^{\min(2\gamma,1)}}\to0,
    \end{align}
    as $N\to+\infty$. Therefore it is enough to check that
    \begin{align}
        \sum_{j=1}^N \E[\xi_{Nj}\1_{\{ |\xi_{Nj}| > 1 \}}] \to0,\quad\text{as $N\to+\infty$},
    \end{align}
    which is, then, equivalent to proving condition (2).
    
    For a random variable $X$, we denote $X_+ =\max{(X,0)}$. Then
    \begin{align}
        \frac{1}{\sqrt{N}} \sum_{j=1}^N \E& [(Z_h^j -\E[Z])_+ \1_{\{ |\xi_{Nj}| > 1 \}}] =\sqrt{N}\E[(Z_h^1 -\E[Z] )_+ \1_{\{ |\xi_{N1}| > 1 \}}]\\
        &\leq \E[(Z_h^1 -\E[Z] )^2_+ \1_{\{ |\xi_{N1}| > 1 \}}] \\
        &\leq \E[2(Z_h^1 -Z)^2_+ \1_{\{ |\xi_{N1}| > 1 \}}] + \E[2( Z -\E[Z] )^2_+ \1_{\{ |\xi_{N1}| > 1 \}}].
        \label{1521}
    \end{align}
    Theorem \ref{thm:l2error} implies that the first term of \eqref{1521} goes to zero as $N\to+\infty$. The second term in \eqref{1521} goes to zero by the dominated convergence theorem, since $\E[Z^2]<\infty$ by the calculations from the beginning of the proof of Theorem \ref{thm:l2error}.  Hence,
    \begin{align}\label{plus}
        \frac{1}{\sqrt{N}} \sum_{j=1}^N \E[( Z_h^j -\E[Z] )_+ \1_{\{ |\xi_{Nj}| > 1 \}}] \to0,\quad\text{as $N\to+\infty$}.
    \end{align}
        Similarly, denote $X_- =\max{\{-X,0\}}$. Then, by repeating the very same steps as above we have that
    \begin{align}\label{minus}
        \frac{1}{\sqrt{N}} \sum_{j=1}^N \E[( Z_h^j -\E[Z] )_-\, \1_{\{ |\xi_{Nj}| > 1 \}}] \to0,\quad\text{as $N\to+\infty$}.
    \end{align}
    By combining~\eqref{plus} and~\eqref{minus}, condition (2) holds with $b=0$.

Proof of condition (3).
    We have that
    \begin{align}
        \sum_{j=1}^N \var[\xi_{Nj}\1_{\{ |\xi_{Nj}|\leq1 \}}]   = \frac{1}{N} \sum_{j=1}^N \left( \E[( Z_h^j-\E[Z])^2 \1_{\{ |\xi_{Nj}|\leq1 \}}]  - \l \E[( Z_h^j-\E[Z] ) \1_{\{ |\xi_{Nj}|\leq1 \}}] \r^2 \right).
    \end{align}
    The following lines prove that the second term on the right-hand side of the above equation goes to zero as $N\to+\infty$.
    \begin{align}\label{2251a}
        \l \E[( Z_h^j-\E[Z]) \1_{\{ |\xi_{Nj}|\leq1 \}}] \r^2 
        \leq 2 \l \E[(Z_h^j-Z) \1_{\{ |\xi_{Nj}|\leq1 \}}] \r^2  + 2\l \E[(Z-\E[Z]) \1_{\{ |\xi_{Nj}|\leq1 \}}] \r^2.
    \end{align}
    Note that $\E[Z-\E[Z]]=0$ so
    \begin{align}\label{2251b}
        \E[(Z-\E[Z]) \1_{\{ |\xi_{Nj}|\leq1 \}}] =-\E[(Z-\E[Z]) \1_{\{ |\xi_{Nj}|>1 \}}].
    \end{align}
    Similarly,
    \begin{align}\label{2251c}
        \E[(Z_h^j-Z) \1_{\{ |\xi_{Nj}|\leq1 \}}] = \E[(Z_h^j-Z)] - \E[(Z_h^j-Z) \1_{\{ |\xi_{Nj}|>1 \}}].
    \end{align}

    Substituting equations \eqref{2251b} and \eqref{2251c} into \eqref{2251a} and using equations \eqref{plus}, \eqref{minus}, and \eqref{2251}, shows that \eqref{2251a} goes to zero as $N\to+\infty$. Indeed, \eqref{2251b} also tends to zero by dominated convergence theorem, using \eqref{momsec}.

    {Now, we will prove}
    \begin{align}
        \frac{1}{N} \sum_{j=1}^N \E[(Z_h^j-\E[Z])^2 \1_{\{ |\xi_{Nj}|\leq1 \}}] \to \var[Z].
    \end{align}

    Use the Cauchy-Schwarz inequality to note that
    \begin{align}
        &\E[(Z_h^j-\E[Z])^2] - \E[(Z-\E[Z])^2] = \E[( Z_h^j+Z-2\E[Z] )( Z_h^j-Z )] \notag\\
        &\quad\leq \l \E[(Z_h^j+Z-2\E[Z])^2] \E[(Z_h^j-Z)^2] \r^{1/2} \notag\\
        &\quad\leq \l \E[2Z^2 + 2(Z_h^j-2\E[Z])^2] \E[(Z_h^j-Z)^2] \r^{1/2}.
    \end{align}
    Since $\E[Z^2]<+\infty$ (as proved in \eqref{momsec}) and (with the help of \eqref{2251}) $\E[(Z_h^j-Z)^2] \leq \mathcal{C}_1 \E[e^{\mathcal{C}_2 T_t}] h^{\min(2\gamma,1)}$, and furthermore we know that $\E[|Z_h^j|^2]$ have a uniform upper bound for all $N$ by equation \eqref{51605}. Hence, $\E[( Z_h^j+Z-2\E[Z] )^2]$ have a uniform upper bound for all $N$, and so
    \begin{align}
        \frac{1}{N} \sum_{j=1}^N \E[(Z_h^j-\E[Z])^2] - \E[( Z-\E[Z])^2] \to0\quad\text{as $N\to+\infty$}.
    \end{align}
    Therefore, we now only need to show
    \begin{align}
        \frac{1}{N} \sum_{j=1}^N \E[(Z_h^j-\E[Z])^2 \1_{\{ |\xi_{Nj}|>1 \}}] \to0\quad\text{as $N\to+\infty$}.
    \end{align}

    In fact, this is true because
    \begin{equation}
        \begin{split}
            \frac{1}{N} \sum_{j=1}^N \E[(Z_h^j-\E[Z])^2 \1_{\{ |\xi_{Nj}|>1 \}}] \leq \frac{1}{N} \sum_{j=1}^N &\Big( 2\E[(Z_h^j-Z)^2 \1_{\{ |\xi_{Nj}|>1 \}}] \\
            &\qquad+ 2\E[(Z-\E[Z])^2 \1_{\{ |\xi_{Nj}|>1 \}}] \Big),
        \end{split}
    \end{equation}
    which tends to zero as in \eqref{1521}.
\end{proof}

\begin{remark}
    The choice of $h^{\min(2\gamma,1)}=N^\delta$, $\delta< -1$, is not unique in the sense that the proof of the previous claim holds even if $Nh^{\min(2\gamma,1)}\to0$ as $N\to+\infty$ which can be seen by inspecting the proof.
\end{remark}

\begin{remark}\label{Biased_complex_MCEM}
    In the classical case of Monte Carlo approximations for the Euler-Maruyama scheme, one chooses $N$ and then $h$ so that $h^{\min\{2\gamma,1\}}=N^{-2}$, i.e. $\delta=-2$. In this case, the bound of the $L^2$ error of the approximations, given in Theorem \ref{thm:l2error}, is
    \begin{align}
        \frac{v}{N} + \frac{C_u\sqrt{\mathcal{C}_1 \E[e^{\mathcal{C}_2 T_t}]}}{N}.
    \end{align}
    If the tolerance for the $L^2$ error of the approximations is $\varepsilon$, then we can choose $N$ so that $N=\frac{v+C_u\sqrt{\mathcal{C}_1 \E[e^{\mathcal{C}_2 T_t}]}}{\varepsilon}$. In this case, the {expected} complexity (obtained by repeating the reasoning as in Proposition \ref{complex_MCEM}) is bounded by
    \begin{align*}
        &\frac{v+C_u\sqrt{\mathcal{C}_1 \E[e^{\mathcal{C}_2 T_t}]}}{\varepsilon} \l ({\mathcal{K}^{\gamma}(1/2)}+8)n+ {\E T_{t_n}}\l\frac{v+C_u\sqrt{\mathcal{C}_1 \E[e^{\mathcal{C}_2 T_{{t_n}}}]}}{\varepsilon}\r^{\max\{2,1/\gamma\}}+1\r,
    \end{align*}
    if $T_t=H_t$, and
    \begin{align*}
         &\frac{v+C_u\sqrt{\mathcal{C}_1 \E[e^{\mathcal{C}_2 T_t}]}}{\varepsilon} \l ({\mathcal{K}^{\Gamma}(1/2)}+5)n+ {\E T_{t_n}}\l\frac{v+C_u\sqrt{\mathcal{C}_1 \E[e^{\mathcal{C}_2 T_{{t_n}}}]}}{\varepsilon}\r^{\max\{2,1/\gamma\}}+1\r,
    \end{align*}
    if $T_t=L_t$ or $T_t=D_t$ (with the additional assumption $M(\mathcal{C}_2)<\infty$ in this case, see \eqref{lunchtime}).

    The explicit expressions for the constants in the lines above can be found in Theorem \ref{thm:l2error} and Lemma \ref{constants}. {Here we also have that the complexity has finite moments of all orders if we use Algorithm \ref{alg:simplif_2} in Algorithm \ref{alg:2} and Algorithm \ref{alg:simplif} in Algorithm \ref{alg:3}.}
\end{remark}

\section{Subdiffusion and Weak Ergodicity Breaking}
\label{sec:web}
In this section, we show how the processes studied in this paper can exhibit non-ergodic behavior, in a suitable sense. This is experimentally relevant and detectable by our algorithms. Indeed, in order to experimentally study the diffusive behavior of particles undergoing random motion, it is very common to measure the displacement traveled by a single particle in (typically disjoint) intervals and then calculate the corresponding mean square displacement. In other words, if $X_t$ is the position of the particle at time $t>0$, the quantity studied here is
\begin{align}
\frac{1}{n}\sum_{i=1}^n | X_{s_i+t}-X_{s_i} |^2,
\label{msdsper}
\end{align}
where $s_i$, $i=1, 2, \dots, n$, are the instants at which the measurement begins. Usually, in experiments one typically has that $|s_i - s_{i-1}| =t$ and the particle is followed on a time horizon $T=nt$. This occurs, for example, in biological contexts where it is easier to observe the motion of a single (chemically illuminated) particle rather than observing many different ones. The reader can consult, for example, \cite{Bronstein2009, Caspi2000,  Seisenberger2001,  SelhuberUnkel2009, TolicNorrelykke2004, Weber2010} where experiments of this type are conducted to study the (anomalous) diffusive behavior of particles.

The quantity \eqref{msdsper} represents a finite-dimensional (observable) version of the functional
\begin{align}
\delta^2(t,T) \coloneqq \frac{1}{T-t}\int_0^{T-t} |X_{s+t}-X_s|^2 ds.
\label{eq: delta}
\end{align}
The functional \eqref{eq: delta} is called the ``time averaged square displacement" and yields the corresponding ``time averaged \emph{mean} square displacement"
\begin{align}
\langle\delta^2(t,T)\rangle \coloneqq \frac{1}{T-t}\int_0^{T-t} \mathds{E}|X_{s+t} - X_s|^2 ds.
\label{eq: deltaaverage}
\end{align}
On the other hand, the classical mean square displacement, given as
\begin{align}\label{eq: MeanSquaredDisplacement}
\langle X^2_t \rangle \coloneqq \mathds{E}^x |X_t-x|^2,
\end{align}
is, instead, statistically observed just as the sample mean square
\begin{align}
\Delta_N^2(t)\coloneqq \frac{1}{N} \sum_{i=1}^N |X_t^i - x|^2,
\end{align}
where $X_t^i$, $i=1, 2, \cdots, N$, are independent observations. We suggest the instructive discussions on these definitions in \cite{Burov2011, Metzler2014}.

In the case of Brownian Motion an ergodic behavior is observed, i.e., the a.s. convergence
\begin{align}\label{66}
\lim_{T \to +\infty} \delta^2(t,T) = \mathds{E}^x  |X_t-x|^2 =t
\end{align}
is observed. To illustrate, in Figure \ref{fig:Ergodicity} we chosen $T-t$ big enough and we compare the theoretical mean square displacement of the (1-dim) Brownian motion with the empirical means, i.e. we compare the following two quantities
\begin{align}
   \frac{1}{n} \sum_{i=1}^n |B_{t+s_i}-B_{s_i}|^2 \quad \text{and}\quad  \frac{1}{N} \sum_{i=1}^N |B^i_t|^2,
\end{align}
for different values of $t>0$.

However, in general, the behavior as in \eqref{66} is not guaranteed and, furthermore, for finite $T$ one can observe fluctuations of \eqref{eq: delta}, depending on the trajectory. Therefore, in \cite{Burov2011}, the authors considered the average \eqref{eq: deltaaverage} and showed that, by choosing a CTRW $(X_t,t\ge0)$ with i.i.d. heavy-tailed waiting times, e.g., with the density
\begin{align}
f(t) \sim \frac{t^{-\alpha-1}}{\Gamma(1-\alpha)}, \qquad t \to +\infty,
\end{align}
for $\alpha \in (0,1)$, and i.i.d. jumps with finite variance (in our case, for simplicity, we take variance of one), one has
\begin{align}
\langle \delta^2(t,T) \rangle \sim {C(\alpha)}\frac{t}{T^{1-\alpha}},
\label{1111}
\end{align}
in the limit for $t\ll T\to +\infty$, i.e. for fixed $t$ and $T\to +\infty$. Here, $C(\alpha)$ is a positive constant and $f(t)\sim g(t)$ as $t\to +\infty$ means $\lim_{t\to+\infty}f(t)/g(t)=1$. It is not hard to see, see \cite{Meerschaert2004}, that the re-scaled CTRW $c^{-1/2} X_{tc^{1/\alpha}}$ converges, as $c \to +\infty$, weakly in the Skorohod $J$ topology, to the time-changed Brownian motion $B_{L_t}$ where $L_t$ is the inverse of an (independent) $\alpha-$stable subordinator. The theoretical mean square displacement is known explicitly by the easy computation
\begin{align}
    \mathds{E}^{(x,0)} |B_{L_t}-x|^2 = \mathds{E}^{(x,0)}\mathds{E}^{(x,0)}[|B_{L_t}-x|^2 |L_t ] \, = \, \mathds{E}^{(x,0)}L_t = \frac{t^\alpha}{\Gamma(1+\alpha)}.
\end{align}
A natural interesting question therefore is the behavior of \eqref{eq: delta} and \eqref{eq: deltaaverage}, in the limit $t\ll T\to +\infty$,  for the time-changed processes considered in these papers, that arise as scaling limits of CTRWs. Figure \ref{fig:ErgodicityBreaking} shows Weak Ergodicity Breaking for the time-changed Brownian motion $B_{L_t}$. A thorough analytical study of functionals of the form \eqref{eq: deltaaverage} for our time-changed processes should be performed, but this is not an easy task and requires a dedicated effort, which is out of the scope of the current paper.

Now we show how one can use the tools developed in this paper in this context. In particular, we illustrate a possible procedure to study \eqref{eq: deltaaverage} in the case of the time-changed Brownian motion $B_{L_t}$, and compare results with the CTRW results of \cite{Burov2011}. Heuristically, we expect to observe that same behavior \eqref{1111}. Moreover, this method can be immediately generalized to other time-changed Markov processes for further use in this kind of applications. We remark that, for moving particles, CTRWs are meant as jumping approximations of a random motion, and our algorithms permit us to observe the behavior explicitly on the `true' limit process.

In order to find information on the functional \eqref{eq: deltaaverage} one can study its statistical counterpart
\begin{align}\label{functionalaverage}
    \frac{1}{N} \sum_{k=1}^N u (M_{L_{t_1}}^k, \cdots, M_{L_{t_n}}^k), \quad\text{and}\quad u(x_1, \cdots, x_n) \coloneqq \frac{1}{n}\sum_{i=1}^n |x_{i+1}-x_i|^2,
\end{align}
for $|t_{i+1}-t_i| = t$ and $T=nt$ by producing exact samples of the (finite-dimensional) trajectories of the time-changed process. Our Theorem \ref{thm1541} and Remark \ref{complex_MC} apply to the observable quantity \eqref{functionalaverage} (which is that used in experiments) and thus we can determine the confidence interval for the r.v.
\begin{align}\label{613}
    \frac{1}{n} \sum_{i=1}^n |M_{L_{t_{i+1}}}-M_{L_{t_i}}|^2
\end{align}
for different values of $n$, or in other words for  different values of $t=T/n$, as in Figure \ref{fig:Mdelta}, computed for $N=10^4$.

It turns out that the behavior of \eqref{613}, for $t=T/n$ varying, is linear with $t$, see Figure \ref{fig:Mdelta}. Therefore, it seems that the anomalous behavior is lost as we observe linear dependence on $t$. However, the dependence on $t$ is of the form $C(\alpha)t/T^{1-\alpha}$ as Figure \ref{fig:Mdelta} shows. This means that, in order to see the anomalous diffusive behavior, one has to look at the dependence on the time horizon $T$. The heuristic explanation of this phenomenon is as follows. The anomalous behavior is, due to the trapping effect, induced by the intervals of constancy of the inverse stable subordinator: as the time horizon increases, because of the heavy tail of the L\'evy measure (a power law with infinite expectation), the probability to find extremely long intervals (trappings) increases, and thus the slope of the line $t/T^{1-\alpha}$ decreases with $T$, faster for small values of $\alpha$. Nevertheless, for fixed time horizon, the dependence on $t$ is linear and the anomalous behavior is not visible.

\begin{figure}[ht]
    \centering
    \subcaptionbox{Ergodic behavior of Brownian motion $X_t$\label{fig:Ergodicity}}[.32\linewidth]{
        \includegraphics[width=\linewidth]{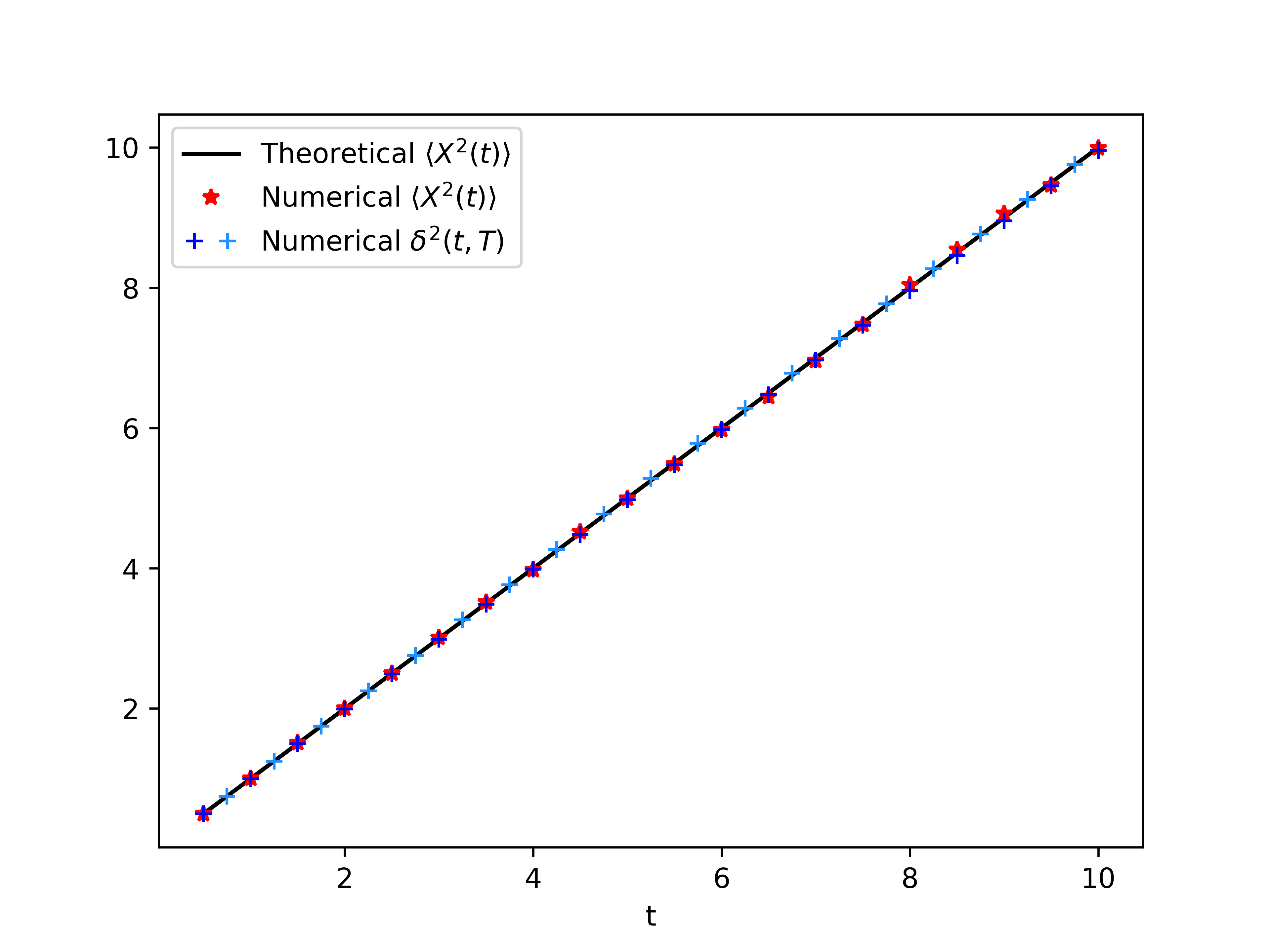}
    }\hfill
    \subcaptionbox{Weak ergodicity breaking for the time-changed Brownian motion $X_t=B_{L_t}$\label{fig:ErgodicityBreaking}}[.32\linewidth]{
        \includegraphics[width=\linewidth]{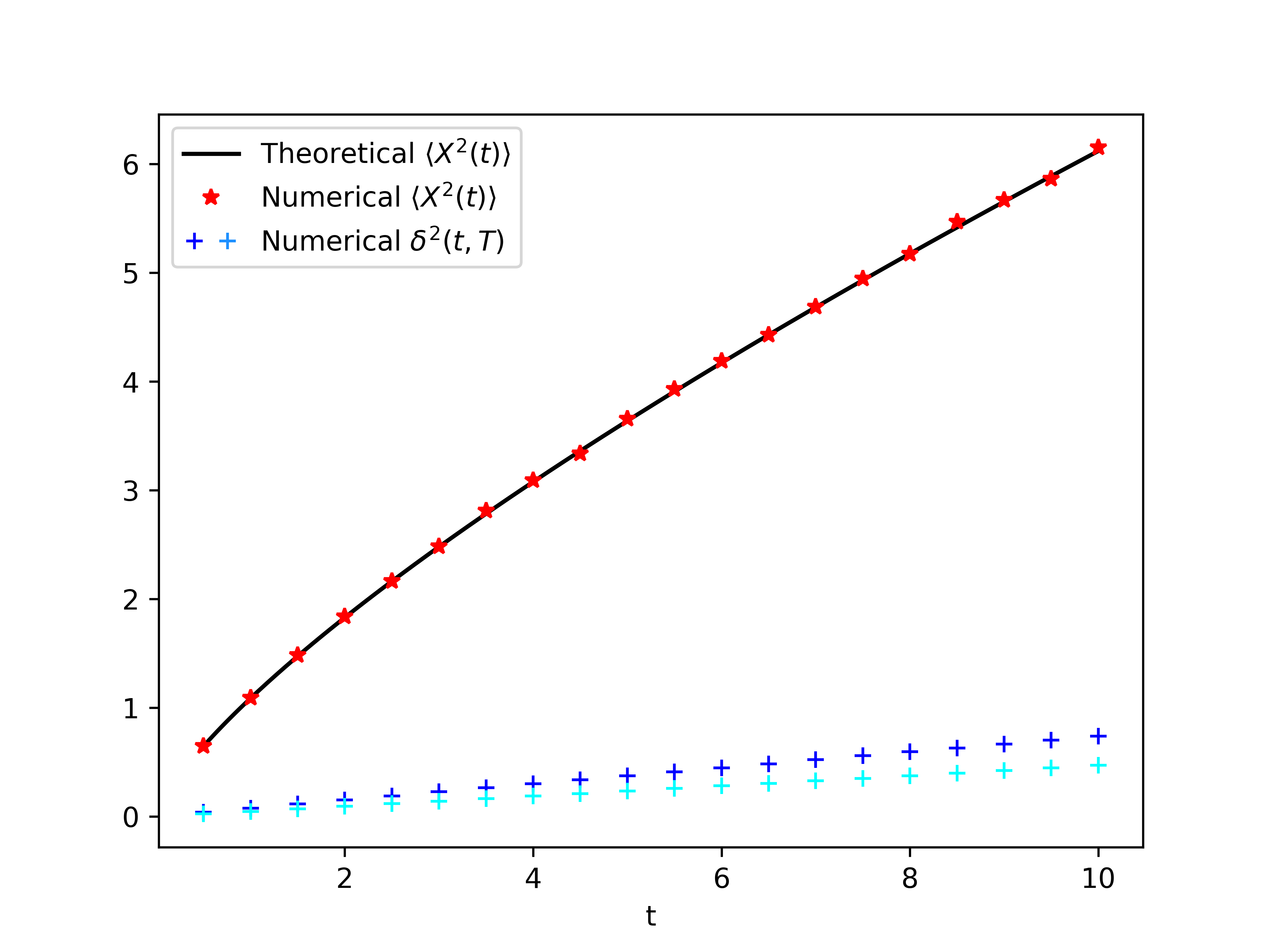}
    }\hfill
    \subcaptionbox{Empirical time averaged \emph{mean} square displacement \eqref{functionalaverage} for two different $T$ and the $95\%$ confidence interval \label{fig:Mdelta}}[.32\linewidth]{
        \includegraphics[width=\linewidth]{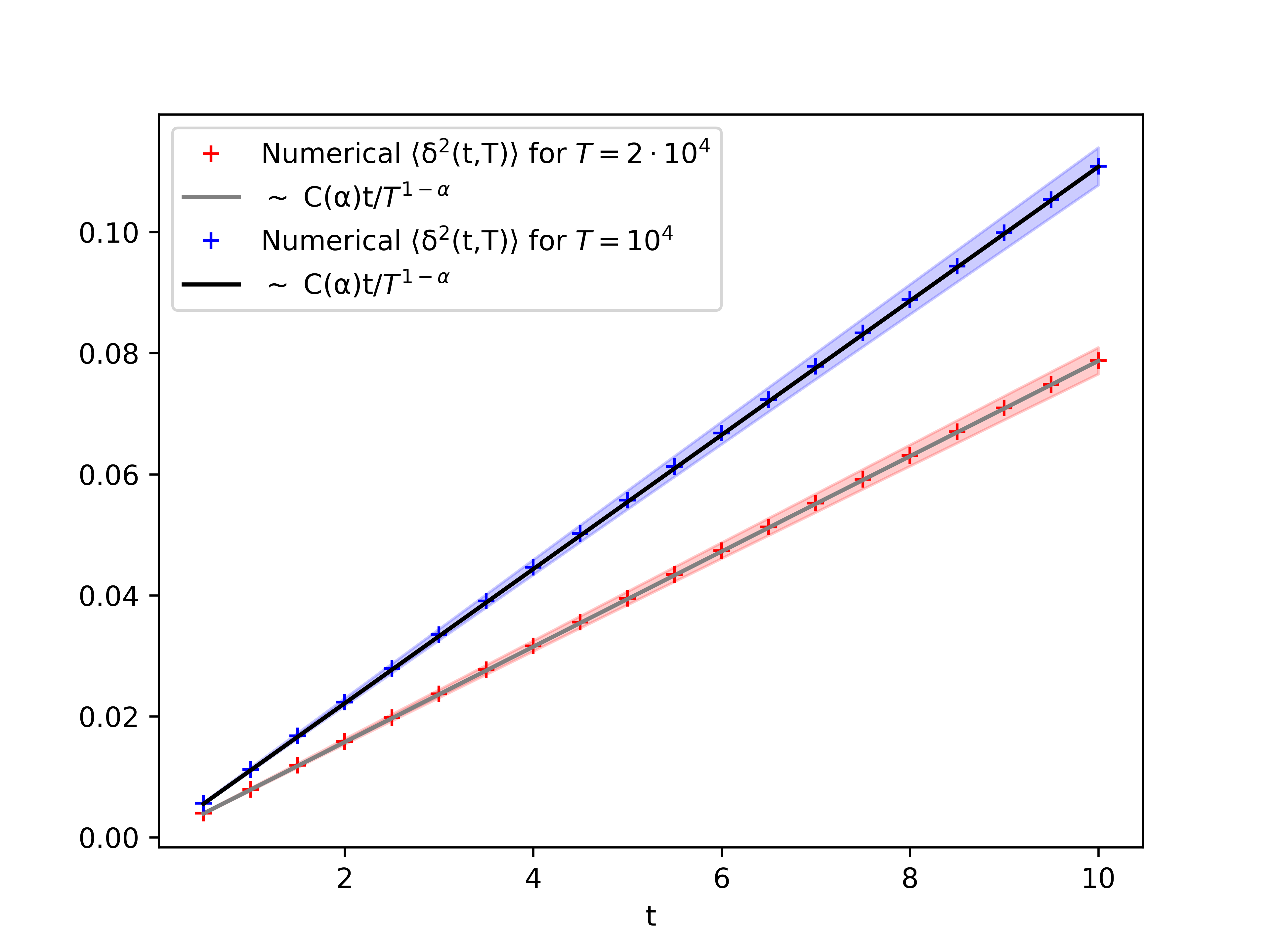}
    }
    \caption{Ergodic behavior for a 1-dimensional, driftless Brownian motion started at zero, and weak ergodicity breaking for a time-changed Brownian motion (the time-changed obtained by a $0.75$-stable inverse subordinator for (A) and (B) while $0.5$ for (C)). The blue and cyan crosses in (A) and (B) stand for independent samples of $\delta^2(t,T)$. In (A) and (B), the mean ensemble is performed with $10^5$ trajectories and $T=10^5$, while in (C), it is performed with {$10^4$} trajectories. {Here we observe that $C(1/2)\approx 1.1\dots$}}
    \label{fig:WEB}
\end{figure}

\appendix
\section{Explicit Euler-Maruyama error}
To prove Lemma \ref{constants}, we need an auxiliary result.
\begin{lemma}\label{difference}
    Suppose \eqref{Lipschitz} and \eqref{Linear} hold, and that $\E|X_{t_0}|^{2n}<\infty$, for some $n\in\N$. Then the solution $X_t$ to~\eqref{sde} satisfies
    \begin{align*}
        \E|X_t|^{2n} \leq c_{0,n} e^{2c_n (t-t_0)},
    \end{align*}
    and
    \begin{align*}
        \E|X_t - X_s|^2 \leq  4(1+m^2)(1+c_{0,1}) dK^2 e^{2c_1t} h,
    \end{align*}
    for any two time points $s$ and $t$ such that $0\leq t-s\leq h$, $h\in(0,1)$, where      $c_n= 4nd(K +\frac{1}{2}mK^2 + (n-1)dmK^2)$ and $c_{0,n}=1/2+\E|X_{t_0}|^{2n}$.
\end{lemma}

\begin{proof}
    It follows from the It\^o formula that $|X_t|^{2n}$ satisfies the SDE
    \begin{align*}
        |X_t|^{2n} = |X_{t_0}|^{2n} &+ \sum_{i=1}^d\int_{t_0}^t 2n|X_s|^{2n-2} X_s^i a^i(s,X_s) ds\\
        &+ \sum_{j=1}^m \int_{t_0}^t \sum_{i=1}^d 2n|X_s|^{2n-2} X_s^i b^{i,j}(s,X_s) dW_s^j\\
        &+ \frac{1}{2} \sum_{i,j=1}^d \int_{t_0}^t 2n(2n-2)|X_s|^{2n-4} X_s^i X_s^j B^{i,j}(s,X_s) ds\\
        &+ \frac{1}{2} \sum_{i=1}^d \int_{t_0}^t 2n |X_s|^{2n-2} B^{i,i}(s, X_s) ds,
    \end{align*}
    where $B^{i,j}(s,X_s) =\sum_{k=1}^m b^{i,k}(s,X_s) b^{j,k}(s,X_s)$. 
	
    Taking the expectation we obtain
    \begin{align*}
        \E|X_t|^{2n} \leq& \E|X_{t_0}|^{2n} + \E\Big[ \sum_{i=1}^d \int_{t_0}^t 2n|X_s|^{2n-2} \Big( |X_s^i||a^i(s,X_s)| + \frac{1}{2}|B^{i,i}(s, X_s)| \\
        &\hspace{53mm} + \frac{1}{2} \sum_{j=1}^d (2n-2) |B^{i,j}(s,X_s)| \Big) ds \Big]\\
        \leq& \E|X_{t_0}|^{2n} + \E\Big[ \sum_{i=1}^d \int_{t_0}^t 2n|X_s|^{2n-2} \Big( |X_s|K(1+|X_s|) \\
        &\hspace{15mm} + \frac{1}{2}mK^2 (1+|X_s|)^2 + (n-1)dmK^2 (1+|X_s|)^2 \Big) ds \Big]\\
        \leq& \E|X_{t_0}|^{2n} + \E\Big[ \sum_{i=1}^d \int_{t_0}^t 2n|X_s|^{2n-2} (1+|X_s|)^2 \\
        &\hspace{37.5mm} \times \l K + \frac{1}{2}mK^2 + (n-1)dmK^2 \r ds \Big]\\
        \leq& \E|X_{t_0}|^{2n} + 4nd\l K + \frac{1}{2}mK^2 + (n-1)dmK^2 \r\\
        &\hspace{42.5mm}\times\E\Big[ \int_{t_0}^t |X_s|^{2n-2} (1+|X_s|^2) ds \Big]\\
        \leq& \E|X_{t_0}|^{2n} + c_n(t-t_0) + 2c_n\E\Big[ \int_{t_0}^t |X_s|^{2n} ds \Big],
    \end{align*}
    where $c_n=4nd\l K + \frac{1}{2}mK^2 + (n-1)dmK^2 \r$.
	
    The first part of the desired result is now obtained by employing the Gronwall inequality, see e.g. \citep[Lemma 4.5.1]{kloeden1992}, and by integration by parts.
	
    Since $|X_t-X_s|^2\leq \sum_{i=1}^d |X^i_t-X^i_s|^2$, to establish the second result, we use the triangular and Cauchy-Schwarz inequalities, respectively, to obtain
    \begin{align*}
        |X_t^i - X_s^i|^2 &\leq 2\Big| \int_s^t a^i(r,X_r) dr \Big|^2 + 2m \sum_{j=1}^m \Big|\int_s^t b^{i,j}(r,X_r) dW_r^j \Big|^2\\
        \E|X_t^i - X_s^i|^2 &\leq 2\E\Big[\Big| \int_s^t a^i(r,X_r) dr \Big|^2\Big] + 2m \sum_{j=1}^m \E\Big[\Big| \int_s^t b^{i,j}(r,X_r) dW_r^j \Big|^2\Big].\\
    \end{align*}
    Using Cauchy-Schwarz inequality, It\^o isometry, and the linear growth condition then yields
    \begin{align*}
        \E|X_t^i - X_s^i|^2 &\leq 2 \int_s^t \E\Big[ (t-s) |a^i(r,X_r)|^2 + m \sum_{j=1}^m |b^{i,j}(r,X_r)|^2 \Big]  dr\\
        &\leq 2 \int_s^t \E\Big[ K^2(1+|X_r|)^2 + m\sum_{j=1}^m K^2 (1+|X_r|)^2 \Big] dr\\
        &=2(1+m^2) K^2 \int_s^t \E[(1+|X_r|)^2] dr\\
        &\leq 4(1+m^2) K^2 \int_s^t \E[1+|X_r|^2] dr.
    \end{align*}
    By using the first part of this lemma, we get
    \begin{align*}
        \E|X_t^i - X_s^i|^2 &\leq 4(1+m^2) K^2 \int_s^t (1+ c_{0,1} e^{ 2c_1 r }) dr\\
        &\leq 4(1+m^2) K^2 (1+c_{0,1}) e^{ 2c_1 t } {h,}
    \end{align*}
    which finishes the proof of the second claim.
\end{proof}

\begin{proof}[{Proof of Lemma \ref{constants}}]
    The proof follows the line of \cite{Kobayashi2016}, and we take care of tracking the constants. Let us define $\tilde{X}_s = X_s^i - X_s^{i,h} $ where $X_s^{i,h}$ is defined in \eqref{EM}. It follows that
    \begin{align}
        X_{t}^{i,h} \, = \, X_{t_0}^{i,h} +\int_{t_0}^t a^i (t_{n_r}, X_{t_{n_r}}^h) dr + \sum_{j=1}^m  \int_{t_0}^t b^{i,j} (t_{n_r}, X_{t_{n_r}}^h) dW^j_r,\label{EMIntegralForm}
    \end{align}
    where $n_r \coloneqq \max\{ n\in\N:\, t_n\leq r \}$. Denoting by {$[X,X]=([X,X]_t,\, t\geq0)$} the quadratic variation process for an arbitrary $X$, for $\tilde{X}$ we have that
    \begin{align}
        \tilde{X}_s^2 &= \int_0^s 2\tilde{X}_r d\tilde{X}_r + [\tilde{X},\tilde{X}]_s.
    \end{align}
    The integral representations~\eqref{sde_components} and~\eqref{EM} then yield
    \begin{align}
        \tilde{X}_s^2 =& 2\int_0^s \tilde{X}_r \big( a^i(r,X_r) -a^i(t_{n_r},X_{t_{n_r}}^h) \big) dr \nonumber\\
        &+ 2 \sum_{j=1}^m \int_0^s \tilde{X}_r \big( b^{i,j}(r,X_r) -b^{i,j} (t_{n_r},X_{t_{n_r}}^h) \big) dW_r^j \nonumber\\
        &+ \sum_{j=1}^m \int_0^s \big( b^{i,j}(r,X_r) -b^{i,j}(t_{n_r},X_{t_{n_r}}^h)   \big)^2 dr.
    \label{a3}
    \end{align}
    We use the representation in \eqref{a3} to evaluate the strong error. We define
    \begin{align}
        Z(t) \coloneqq \E\Big[ \sup_{0\leq s\leq t} \tilde{X}_s^2\Big] = I_1(t) +I_2(t) +I_3(t),
    \end{align}
    where $I_i$ denotes the expectation of the supremum over $s\in[0,t]$ of the $i$-th term in \eqref{a3}.

    The term $I_2(t)$ can be estimated by using the Burkholder-Davis-Gundy inequality \citep[Theorem 48, Chapter IV]{Protter2004}:
    \begin{align*}
        I_2(t) &\leq 2\E\Big[ \sup_{0\leq s\leq t} \Big| \sum_{j=1}^m\int_0^s \tilde{X}_r \big( b^{i,j}(r,X_r) -b^{i,j} (t_{n_r},X_{t_{n_r}}^h) \big) dW_r^j \Big|\Big]\\
        &\leq 2C_1 \E\Big[ \Big( \sum_{j=1}^m \int_0^t \tilde{X}_r^2 \big( b^{i,j}(r,X_r) -b^{i,j}(t_{n_r},X_{t_{n_r}}^h) \big)^2 dr \Big)^{1/2}\Big]\\
        &\leq  \E\Big[ \Big( \sup_{0\leq s\leq t} \tilde{X}_s^2\Big)^{1/2}  \Big( 4C_1^2\sum_{j=1}^m \int_0^t \big( b^{i,j}(r,X_r) -b^{i,j}(t_{n_r},X_{t_{n_r}}^h) \big)^2 dr \Big)^{1/2}\Big]
    \end{align*}
    where $C_1= 1.30693...$ see \cite[p. 591]{Osekowski2010}. Using 
    the inequality $(xy)^{1/2} \leq x/2 +y$ for all $x,\,y\geq0$  yields
    \begin{align}
        I_2(t)\leq \,& \frac{1}{2}Z(t) + 4C_1^2 \E\Big[\sum_{j=1}^m \int_0^t \big( b^{i,j}(r,X_r) -b^{i,j}(t_{n_r},X_{t_{n_r}}^h) \big)^2 dr\Big] \notag \\
        = \, & \frac{1}{2}Z(t) + 4C_1^2 \, I_3(t).
    \end{align}
    Consequently, we have
    \begin{align}
        Z(t) \leq  2 \, I_1 (t)  + 2(1+4C_1^2) \, I_3 (t).
    \end{align}

    To deal with $I_3(t)$ we use \eqref{Lipschitz}, \eqref{Linear} and \eqref{tdiversi} to get 
    \begin{align*}
        &|b^{i,j}(r,X_r) -b^{i,j}(t_{n_r},X_{t_{n_r}}^h)|\\
        &\quad\leq |b^{i,j}(r,X_r) -b^{i,j}(t_{n_r},X_r)| + | b^{i,j}(t_{n_r}, X_r) -b^{i,j}(t_{n_r},X_{t_{n_r}}) |\\
        &\qquad+ | b^{i,j}(t_{n_r},X_{t_{n_r}}) -b^{i,j}(t_{n_r},X_{t_{n_r}}^h) |\\
        &\quad\leq K(1+|X_r|) (r-t_{n_r})^\gamma +K|X_r - X_{ t_{n_r} }| + K|X_{t_{n_r}}-X^h_{t_{n_r}}|,
    \end{align*}
    and equivalently
    \begin{align}\label{1352}
        |a^i(r, X_r) -a^i(t_{n_r}, X_{t_{n_r}}^h)| \leq& K(1+|X_r|) (r-t_{n_r})^\gamma +K|X_r - X_{ t_{n_r} } |\nonumber\\
        &+ K|X_{t_{n_r}}-X^h_{t_{n_r}} |.
    \end{align}
    Then, by using Lemma \ref{difference}
    \begin{align*}
        I_3(t) \leq& 3K^2\E\Big[ \sup_{0\leq s\leq t} \sum_{j=1}^m \int_0^s \Big( 2(1+|X_r|^2) h^{2\gamma} + |X_r - X_{ t_{n_r} }|^2 \\
        &\hspace{63mm} + |X_{t_{n_r}}-X^h_{t_{n_r}} |^2 \Big)dr \Big]\\
        &\leq 3mK^2 \int_0^t ( 2(1+c_{0,1}e^{2c_1 r}) h^{2\gamma} + 4(1+m^2) (1+c_{0,1}) dK^2 e^{2c_1 r} h ) dr\\
        &\qquad +3mK^2 \int_0^t \E|X_{t_{n_r}}-X^h_{t_{n_r}} |^2 dr\\
        &\leq 3mK^2 \Big( 2\Big( t+\frac{c_{0,1}}{2c_1}e^{2c_1 t}\Big) h^{2\gamma} + 4(1+m^2)\frac{(1+c_{0,1})}{2c_1}dK^2 e^{2c_1 t} h \Big)\\
        &\qquad +3mK^2 \int_0^t \E|X_{t_{n_r}}-X^h_{t_{n_r}} |^2 dr\\
        &\leq 3mK^2 \Big( 2\Big(1+\frac{c_{0,1}}{2c_1}\Big) h^{2\gamma} + 4(1+m^2) 2\Big(1+\frac{c_{0,1}}{2c_1}\Big) dK^2 h \Big) e^{2c_1 t}\\
        &\qquad +3mK^2 \int_0^t \E| X_{t_{n_r}}-X^h_{t_{n_r}} |^2 dr\\
        &\leq 6mK^2 \Big(1+\frac{c_{0,1}}{2c_1}\Big) (h^{2\gamma} + { 4(1+m^2)} dK^2) h) e^{2c_1 t}\\
        &\hspace{47.5mm} +3mK^2 \int_0^t \E| X_{t_{n_r}}-X^h_{t_{n_r}} |^2 dr\\
        &\leq {6}mK^2 (1+4(1+m^2) dK^2) \Big(1+\frac{c_{0,1}}{2c_1}\Big) e^{2c_1 t} \max(h^{2\gamma},h)\\
        &\hspace{47.5mm} +3mK^2 \int_0^t \E| X_{t_{n_r}}-X^h_{t_{n_r}} |^2 dr.
    \end{align*}

    By using the inequalities $2xy\leq x^2 + y^2$ and \eqref{1352}
    \begin{align*}
        I_1(t) &\leq \E\Big[ \sup_{0\leq s\leq t} \int_0^s 2|\tilde{X}_r| | a^i(r,X_r) -a^i(t_{n_r},X_{t_{n_r}}^h) |dr\Big]\\
        &\leq \E\Big[ \sup_{0\leq s\leq t} \int_0^s ( |\tilde{X}_r|^2 + |a^i(r,X_r) -a^i(t_{n_r},X_{t_{n_r}}^h) |^2 ) dr\Big]\\
        &\leq \int_0^t Z(r)dr + \E\Big[ \sup_{0\leq s\leq t} \int_0^s | a^i(r,X_r) -a^i(t_{n_r},X_{t_{n_r}}^h) |^2\Big]\\
        &\leq \int_0^t Z(r)dr + 3K^2\E\Big[ \sup_{0\leq s\leq t} \int_0^s \big( (1+|X_r|)^2 h^{2\gamma} + |X_r -X_{ t_{n_r} } |^2 \\
        &\hspace{7cm} + |X_{t_{n_r}}-X^h_{t_{n_r}} |^2 \big)dr \Big] \\
        &\leq { 6}K^2 (1+4(1+m^2) dK^2) \Big(1+\frac{c_{0,1}}{2c_1}\Big) e^{2c_1 t} \max(h^{2\gamma},h)\\
        &\hspace{47.5mm} + (3K^2 +1)\int_0^t \E\Big[\sup_{0\le s\le r} |X_{s}-X^h_s |^2 \Big] dr.
    \end{align*}
    Thus, since
    \begin{align*}
        \E\Big[ \sup_{0\leq s\leq t} |X_s-X_s^h|^2\Big] \leq \sum_{i=1}^d \E\Big[ \sup_{0\leq s\leq t} |X^i_s-X^{i,h}_s|^2 \Big],
    \end{align*}
    and by the estimates of $I_1$, $I_2$, and $I_3$, we get
    \begin{align*}
        \E\Big[ \sup_{0\leq s\leq t} |X_s-X_s^h|^2\Big] &\leq A e^{2c_1 t} + B \int_0^t \E\Big[ \sup_{0\leq s\leq r} |X_s-X_s^h|^2\Big] dr,
    \end{align*}
    $A={ 12}dK^2 (1+(1+4C_1^2)m) (1+4(1+m^2) dK^2) \Big(1+\frac{c_{0,1}}{2c_1}\Big) h^{\min{(2\gamma,1)}}$, $B=2d(3K^2(1+(1+4C_1^2)m) +1)$, which together with Gronwall's inequality (see \citep[Lemma 4.5.1]{kloeden1992}), yields
    \begin{align*}
        \E\Big[ \sup_{0\leq s\leq t} |X_s-X_s^h|^2 \Big] &\leq A e^{2c_1 t} + B \int_0^t \E\Big[ \sup_{0\leq s\leq r} |X_s-X_s^h|^2 \Big] dr,\\
        &\leq Ae^{2c_1 t} + B\int_0^t Ae^{2c_1r} e^{B(t-r)}dr\\
        &= Ae^{2c_1 t} + \frac{ABe^{Bt}(1-e^{-(B-2c_1)t})}{B-2c_1},
    \end{align*}
    {since} $2c_1-B<0$ for all $K  >0$, $m,d\in\N$.
\end{proof}

\begin{remark}
    For time-independent coefficients, the {constants $A$ and $B$ read}
    \begin{align*}
        A &= {8}(1+m^2)\frac{(1+c_{0,1})}{c_1}(1+(1+4C_1^2)m)d^2 K^4h,\\
        B &={2d(2K^2(1+(1+4C_1^2)m) +1).}
    \end{align*}
\end{remark}

\begin{lemma}\label{our} 
    Let $\sigma$ be a subordinator given by \eqref{bernstinsimul}, and $L_t$ its inverse. Then, for any $\mathcal{C}_1>0$, and for any $\mathcal{C}_2>0$ such that $\phi(\mathcal{C}_2)>\mathcal{C}_1$ it holds that
    	$$
    	\E[e^{\mathcal{C}_1 L_t}] \leq 1+\frac{e^{\mathcal{C}_2 t}}{\psi(\mathcal{C}_2)-\mathcal{C}_1},\quad t>0.
    	$$  
\end{lemma}

\begin{proof}
    The proof of this claim follows by repeating exactly the same steps as in \cite[Lemma A.1]{Jorge2023a}. 
\end{proof}

\section*{Acknowledgments}
The authors acknowledge financial support under the National Recovery and Resilience Plan (NRRP), Mission 4, Component 2, Investment 1.1, Call for tender No. 104 published on 2.2.2022 by the Italian Ministry of University and Research (MUR), funded by the European Union – NextGenerationEU– Project Title “Non–Markovian Dynamics and Non-local Equations” – 202277N5H9 - CUP: D53D23005670006 - Grant Assignment Decree No. 973 adopted on June 30, 2023, by the Italian Ministry of University and Research (MUR).

The authors would like to thank the Isaac Newton Institute for Mathematical Sciences, Cambridge, for support and hospitality during the programme Stochastic systems for anomalous diffusion, where work on this paper was undertaken. This work was supported by EPSRC grant EP/Z000580/1.

The author Bruno Toaldo are partially supported by Gruppo Nazionale per l’Analisi Matematica, la Probabilità e le loro Applicazioni (GNAMPA-INdAM).

The authors would like to thank Prof. Aleksandar Mijatović for fruitful discussions on the
topic that improved a previous version of this manuscript.

The authors would like to thank Prof. Eli Barkai for very useful remarks on weak ergodicity breaking.

The authors would like to thank two anonymous referees whose remarks and suggestions considerably improved a previous version of this manuscript.

\bibliographystyle{abbrv}
\bibliography{References}

\begin{thebibliography}{10}

\bibitem{applebaum}
D.~Applebaum.
\newblock {\em L\'evy processes and stochastic calculus}.
\newblock Cambridge University Press, New York, second edition, 2009.

\bibitem{benarous}
G.~B. Arous, M.~Cabezas, J.~\v{C}ern\'y, and R.~Royfman.
\newblock Randomly trapped random walks.
\newblock {\em The Annals of Probability}, 43(5):2405--2457, 2015.

\bibitem{Ascione2021}
G.~Ascione, N.~Leonenko, and E.~Pirozzi.
\newblock Time-{N}on-{L}ocal {P}earson diffusions.
\newblock {\em Journal of Statistical Physics}, 183(48), 2021.

\bibitem{APT20}
G.~Ascione, E.~Pirozzi, and B.~Toaldo.
\newblock On the exit time from open sets of some semi-{M}arkov processes.
\newblock {\em The Annals of Applied Probability}, 30(3):1130--1163, 2020.

\bibitem{tlms}
G.~Ascione, M.~Savov, and B.~Toaldo.
\newblock Regularity and asymptotics of densities of inverse subordinators.
\newblock {\em Transactions of the London Mathematical Society}, 11(1):e70004,
  2024.

\bibitem{ascione2024}
G.~Ascione, E.~Scalas, B.~Toaldo, and L.~Torricelli.
\newblock Time-changed {M}arkov processes and coupled non-local equations.
\newblock {\em arXiv:2412.14956}, 2024.

\bibitem{Asmussen2007}
S.~Asmussen and P.~W. Glynn.
\newblock {\em Stochastic simulation: Algorithms and analysis}.
\newblock Springer, 2007.

\bibitem{Baeumer2001}
B.~Baeumer and M.~M. Meerschaert.
\newblock Stochastic solutions for fractional {C}auchy problems.
\newblock {\em Fractional Calculus and Applied Analysis}, 4:481--800, 2001.

\bibitem{Baeumer2005}
B.~Baeumer, M.~M. Meerschaert, and J.~Mortensen.
\newblock Space-time fractional derivative operators.
\newblock {\em Proceedings of the American Mathematical Society},
  133(8):2273--2282, 2005.

\bibitem{barkai0}
E.~Barkai, Y.~Garini, and R.~Metzler.
\newblock Strange kinetics of single molecules in living cells.
\newblock {\em Physics Today}, 65(8):29--35, 2012.

\bibitem{barkai1}
E.~Barkai, R.~Metzler, and J.~Klafter.
\newblock From continuous time random walks to the fractional {F}okker-{P}lanck
  equation.
\newblock {\em Physiscal Review E}, 61:132--138, 2000.

\bibitem{bertoin1996}
J.~Bertoin.
\newblock {\em L\'evy processes}.
\newblock Cambridge University Press, Cambridge, 1996.

\bibitem{bertoin1999}
J.~Bertoin.
\newblock {\em Subordinators: examples and applications}.
\newblock Springer, Berlin, 1999.

\bibitem{Python}
I.~Bio\v{c}i\'{c}, D.~E. Cedeño-Girón, and B.~Toaldo.
\newblock Github repository.
\newblock
  \url{https://github.com/DaEdCeGi/Sampling-inverse-subordinators-and-subdiffusions}.
\newblock September 2025.

\bibitem{harmonic}
I.~Biočić and B.~Toaldo.
\newblock Harmonic problems arising from continuous time random walks limit
  processes.
\newblock {\em arXiv:2505.14550}, 2025.

\bibitem{Bogdan2015}
K.~Bogdan, T.~Grzywny, and M.~Ryznar.
\newblock Barriers, exit time and survival probability for unimodal {L}\'evy
  processes.
\newblock {\em Probability Theory and Related Fields}, 162:155--198, 2015.

\bibitem{schillinglevy}
B.~B\"ottcher, R.~Schilling, and J.~Wang.
\newblock {\em L\'evy Matters III. L\'evy-Type Processes: Construction,
  Approximation and Sample Path Properties}.
\newblock Springer, 2013.

\bibitem{Bronstein2009}
I.~Bronstein, Y.~Israel, E.~Kepten, S.~Mai, Y.~Shav-Tal, E.~Barkai, and
  Y.~Garini.
\newblock Transient anomalous diffusion of telomeres in the nucleus of
  mammalian cells.
\newblock {\em Physical Review Letters}, 103:018102, 2009.

\bibitem{Burov2011}
S.~Burov, J.-H. Jeon, R.~Metzler, and E.~Barkai.
\newblock Single particle tracking in systems showing anomalous diffusion: the
  role of weak ergodicity breaking.
\newblock {\em Physical Chemistry Chemical Physics}, 13:1800--1812, 2011.

\bibitem{Casella2011}
B.~Casella and G.~O. Roberts.
\newblock Exact simulation of jump-diffusion processes with {M}onte {C}arlo
  applications.
\newblock {\em Methodology and Computing in Applied Probability}, 13:449--473,
  2011.

\bibitem{Caspi2000}
A.~Caspi, R.~Granek, and M.~Elbaum.
\newblock Enhanced diffusion in active intracellular transport.
\newblock {\em Physical Review Letters}, 85:5655, 2000.

\bibitem{cinlarma}
E.~\c{C}inlar.
\newblock Markov additive processes. {I}.
\newblock {\em Zeitschrift f\"{u}r Wahrscheinlichkeitstheorie und Verwandte
  Gebiete}, 24:85--93, 1972.

\bibitem{chen}
Z.-Q. Chen.
\newblock Time fractional equations and probabilistic representation.
\newblock {\em Chaos, Solitons and Fractals}, 102:168--174, 2017.

\bibitem{Chi2016}
Z.~Chi.
\newblock On exact sampling of the first passage event of a {L}\'evy process
  with infinite {L}\'evy measure and bounded variation.
\newblock {\em Stochastic Processes and their Applications}, 26(4):1124--1144,
  2016.

\bibitem{chi2025complexity}
Z.~Chi.
\newblock Complexity of exact sampling of the first passage of a stable
  subordinator.
\newblock {\em arXiv preprint arXiv:2506.03047}, 2025.

\bibitem{cormen}
T.~H. Cormen, C.~E. Leiserson, R.~L. Rivest, and C.~Stein.
\newblock {\em Introduction to Algorithms}.
\newblock The MIT Press, fourth edition, 2022.

\bibitem{Dassios2020}
A.~Dassios, J.~W. Lim, and Y.~Qu.
\newblock Exact simulation of a truncated {L}\'evy subordinator.
\newblock {\em ACM Transactions on Modeling and Computer Simulation}, 30(3):17,
  2020.

\bibitem{elena}
T.~De~Angelis, M.~Germain, and E.~Issoglio.
\newblock A numerical scheme for stochastic differential equations with
  distributional drift.
\newblock {\em Stochastic Process. Appl.}, 154:55--90, 2022.

\bibitem{Deng2020}
C.-S. Deng and W.~Liu.
\newblock Semi-implicit {E}uler-{M}aruyama method for non-linear time-changed
  stochastic differential equations.
\newblock {\em BIT Numerical Mathematics}, 60:1133--1151, 2020.

\bibitem{DiGregorio2024}
A.~{Di Gregorio} and F.~Iafrate.
\newblock Path dynamics of time-changed {L}\'evy processes: A martingale
  approach.
\newblock {\em Journal of Theoretical Probability}, 37:3246--3280, 2024.

\bibitem{Dudley1968}
R.~M. Dudley.
\newblock Distances of probability measures and random variables.
\newblock {\em The Annals of Mathematical Statistics}, 39(5):1563--1572, 1968.

\bibitem{FALCINI2019584}
F.~Falcini, R.~Garra, and V.~Voller.
\newblock Modeling anomalous heat diffusion: Comparing fractional derivative
  and non-linear diffusivity treatments.
\newblock {\em International Journal of Thermal Sciences}, 137:584--588, 2019.

\bibitem{fedotov1}
S.~Fedotov.
\newblock Non-{M}arkovian random walks and nonlinear reactions: Subdiffusion
  and propagating fronts.
\newblock {\em Physical Review E}, 81:011117, 2010.

\bibitem{fedotov3}
S.~Fedotov, N.~Korabel, T.~A. Waigh, D.~Han, and V.~J. Allan.
\newblock Memory effects and {L}\'evy walk dynamics in intracellular transport
  of cargoes.
\newblock {\em Physical Review E}, 98:042136, 2018.

\bibitem{fedotov2}
S.~Fedotov and V.~M\'endez.
\newblock Non-{M}arkovian model for transport and reactions of particles in
  spiny dendrites.
\newblock {\em Physical Review Letters}, 101:218102, 2008.

\bibitem{Fulger2008}
D.~Fulger, E.~Scalas, and G.~Germano.
\newblock Monte {C}arlo simulation of uncoupled continuous-time random walks
  yielding a stochastic solution of the space-time fractional diffusion
  equation.
\newblock {\em Physical Review E}, 77(2):021122, 2008.

\bibitem{Sacerdote2001}
M.~T. Giraudo, L.~Sacerdote, and C.~Zucca.
\newblock A {M}onte {C}arlo method for the simulation of first passage times of
  diffusion processes.
\newblock {\em Methodology and Computing in Applied Probability}, 3:215--231,
  2001.

\bibitem{Jorge2023b}
J.~I. {Gonz\'{a}lez C\'{a}zares}, F.~Lin, and A.~Mijatovi\'{c}.
\newblock Fast exact simulation of the first-passage event of a subordinator.
\newblock {\em {Stochastic Processes and their Applications}}, {183}:{104599},
  2024.

\bibitem{Jorge2023a}
J.~I. {Gonz\'{a}lez C\'{a}zares}, F.~Lin, and A.~Mijatovi\'{c}.
\newblock Fast exact simulation of the first passage of a tempered stable
  subordinator across a non-increasing function.
\newblock {\em Stochastic Systems}, {15}({1}):{50--87}, 2024.

\bibitem{gianni}
R.~Gorenflo, F.~Mainardi, D.~Moretti, G.~Pagnini, and P.~Paradisi.
\newblock Discrete random walk models for space–time fractional diffusion.
\newblock {\em Chemical Physics}, 284(1):521--541, 2002.

\bibitem{Gorenflo2007}
R.~Gorenflo, F.~Mainardi, and A.~Vivoli.
\newblock Continuos-time random walk and parametric subordination in fractional
  diffusion.
\newblock {\em Chaos, Solitons and Fractals}, 34(1):87--103, 2007.

\bibitem{Kobayashi2012}
M.~Hahn, K.~Kobayashi, and S.~Umarov.
\newblock {SDE}s driven by a time-changed {L}\'evy process and their associated
  time-fractional order pseudo-differential equations.
\newblock {\em Journal of Theoretical Probability}, 25:262--279, 2012.

\bibitem{hairer2}
M.~Hairer, G.~Iyer, L.~Koralov, A.~Novikov, and Z.~Pajor-Gyulai.
\newblock A fractional kinetic process describing the intermediate time
  behaviour of cellular flows.
\newblock {\em The Annals of Probability}, 46(2):897--955, 2018.

\bibitem{hairer1}
M.~Hairer, L.~Koralov, and Z.~Pajor-Gyulai.
\newblock From averaging to homogenization in cellular flows - an exact
  description of the transition.
\newblock {\em Annales de L'institut Henri Poincar\'e - Probabilit\'es et
  Statistiques}, 52(4):1592--1613, 2016.

\bibitem{harlamov}
B.~Harlamov.
\newblock {\em Continuous semi-{M}arkov processes}.
\newblock John Wiley \& Sons, 2013.

\bibitem{hernandez2017}
M.~E. Hernández-Hernández, V.~N. Kolokoltsov, and L.~Toniazzi.
\newblock Generalised fractional evolution equations of {C}aputo type.
\newblock {\em Chaos, Solitons and Fractals}, 102:184--196, 2017.

\bibitem{Zucca2019}
S.~Herrmann and C.~Zucca.
\newblock Exact simulation of the first-passage time of diffusions.
\newblock {\em Journal of Scientific Computing}, 9(3):1477--1504, 2019.

\bibitem{Torricelli2020}
A.~Jacquier and L.~Torricelli.
\newblock Anomalous diffusions in option prices: connecting trade duration and
  the volatility term structure.
\newblock {\em SIAM Journal on Financial Mathematics}, 11(4):1137--1167, 2020.

\bibitem{Kobayashi2019}
S.~Jin and K.~Kobayashi.
\newblock Strong approximation of stochastic differential equations driven by a
  time-changed {B}rownian motion with time-space-dependent coefficients.
\newblock {\em Journal of Mathematical Annalysis and Applications},
  479(2):619--636, 2019.

\bibitem{Kobayashi2021}
S.~Jin and K.~Kobayashi.
\newblock Strong approximation of time-changed stochastic differential
  equations involving drifts with random and non-random integrators.
\newblock {\em BIT Numerical Mathematics}, 61:829--857, 2021.

\bibitem{Kobayashi2016}
E.~Jum and K.~Kobayashi.
\newblock A strong and weak approximation scheme for stochastic differential
  equations driven by a time-changed {B}rownian motion.
\newblock {\em Probability and Mathematical Statistics}, 36(2):201--220, 2016.

\bibitem{Kallenberg}
O.~Kallenberg.
\newblock {\em Foundations of Modern Probability}.
\newblock Springer, second edition, 2002.

\bibitem{kloeden1992}
P.~E. Kloeden and E.~Platen.
\newblock {\em Numerical solution of stochastic differential equations}.
\newblock Springer, 1992.

\bibitem{Kobayashi2011}
K.~Kobayashi.
\newblock Stochastic calculus for a time-changed semimartingale and the
  associated stochastic differential equations.
\newblock {\em Journal of Theoretical Probability}, 24:789--820, 2011.

\bibitem{KOCHUBEI2008252}
A.~N. Kochubei.
\newblock Distributed order calculus and equations of ultraslow diffusion.
\newblock {\em Journal of Mathematical Analysis and Applications},
  340(1):252--281, 2008.

\bibitem{kochubei}
A.~N. Kochubei.
\newblock General fractional calculus, evolution equations and renewal
  processes.
\newblock {\em Integral Equations and Operator Theory}, 71:583--600, 2011.

\bibitem{kolokoltsov2015}
V.~N. Kolokoltsov.
\newblock On fully mixed and multidimensional extensions of the {C}aputo and
  {R}iemann-{L}iouville derivatives, related {M}arkov process and fractional
  differential equations.
\newblock {\em Fractional Calculus and Applied Analysis}, 18(4):1039--1073,
  2015.

\bibitem{Kolokoltsov2023}
V.~N. Kolokoltsov.
\newblock The rates of convergence for functional limit theorems with stable
  subordinators and for {CTRW} approximations to fractional evolutions.
\newblock {\em Fractal and Fractional}, 7(4):335, 2023.

\bibitem{kolokoltsov2021}
V.~N. Kolokoltsov, F.~Ling, and A.~Mijatovi\'c.
\newblock Monte {C}arlo estimation of the solution of fractional partial
  differential equations.
\newblock {\em Fractional Calculus and Applied Analysis}, 24(1):278--306, 2021.

\bibitem{Leonenko2022a}
N.~Leonenko and I.~Podlubny.
\newblock Monte {C}arlo method for fractional-order differentiation.
\newblock {\em Fractional Calculus and Applied Analysis}, 25:346--361, 2022.

\bibitem{Leonenko2013}
N.~N. Leonenko, M.~M. Meerschaert, and A.~Sikorskii.
\newblock Fractional {P}earson diffusions.
\newblock {\em Journal of Mathematical Analysis and Applications},
  403(2):532--546, 2013.

\bibitem{Liu2023}
W.~Liu, R.~Wu, and R.~Zuo.
\newblock A {M}ilstein-type method for highly non-linear non-autonomous
  time-changed stochastic differential equations.
\newblock {\em arXiv:2308.13999}, 2023.

\bibitem{Lv2020}
L.~Lv and L.~Wang.
\newblock Stochastic representation and {M}onte {C}arlo simulation for
  multiterm time-fractional diffusion equation.
\newblock {\em Advances in Mathematical Physics}, 2020:1315426, 2020.

\bibitem{Magdziarz2007}
M.~Magdziarz, A.~Weron, and K.~Weron.
\newblock Fractional {F}okker-{P}lanck dynamics: Stochastic representation and
  computer simulation.
\newblock {\em Physical Review E}, 75(1), 2007.

\bibitem{Meerschaert2004}
M.~M. Meerschaert and H.-P. Scheffler.
\newblock Limit theorems for continuous-time random walks with infinite mean
  waiting times.
\newblock {\em Journal of Applied Probability}, 41(3):623--638, 2004.

\bibitem{meerschaert2019stochastic}
M.~M. Meerschaert and A.~Sikorskii.
\newblock {\em Stochastic models for fractional calculus}, volume~43.
\newblock Walter de Gruyter GmbH \& Co KG, 2019.

\bibitem{Meerschaert2014}
M.~M. Meerschaert and P.~Straka.
\newblock Semi-{M}arkov approach to continuous time random walk limit process.
\newblock {\em The Annals of Probability}, 42(4):1699--1726, 2014.

\bibitem{barkai2}
R.~Metzler, E.~Barkai, and J.~Klafter.
\newblock Anomalous diffusion and relaxation close to thermal equilibrium: A
  fractional {F}okker-{P}lanck equation approach.
\newblock {\em Physical Review Letters}, 82(18):3563--3567, 1999.

\bibitem{Metzler2014}
R.~Metzler, J.-H. Jeon, A.~G. Cherstvy, and E.~Barkai.
\newblock Anomalous diffusion models and their properties: non-stationarity,
  non-ergodicity, and ageing at the centenary of single particle tracking.
\newblock {\em Physical Chemistry Chemical Physics}, 16(44):24128--24164, 2014.

\bibitem{METZLER20001}
R.~Metzler and J.~Klafter.
\newblock The random walk's guide to anomalous diffusion: a fractional dynamics
  approach.
\newblock {\em Physics Reports}, 339(1):1--77, 2000.

\bibitem{aleksunder}
A.~Mijatovi\'c and V.~Vysotsky.
\newblock Stability of overshoots of zero mean random walks.
\newblock {\em Electron. J. Probab.}, 25:Paper No. 63, 22, 2020.

\bibitem{Montroll1965}
E.~W. Montroll and G.~H. Weiss.
\newblock Random walk on lattices {II}.
\newblock {\em Journal of Mathematical Physics}, 6(2):167--181, 1965.

\bibitem{giannisim}
A.~Mura and G.~Pagnini.
\newblock Characterizations and simulations of a class of stochastic processes
  to model anomalous diffusion.
\newblock {\em Journal of Physics A: Mathematical and Theoretical}, 41, 2009.

\bibitem{o2014analysis}
R.~O'Donnell.
\newblock {\em Analysis of Boolean Functions}.
\newblock Cambridge University Press, 2014.

\bibitem{Osekowski2010}
A.~Os\c{e}kowski.
\newblock Sharp maximal inequalities for the martingale square bracket.
\newblock {\em Stochastics}, 82(6):589--605, 2010.

\bibitem{povstenko2015fractional}
Y.~Povstenko.
\newblock {\em Fractional Thermoelasticity}.
\newblock Springer International Publishing, 1st edition, 2015.

\bibitem{Protter2004}
P.~E. Protter.
\newblock {\em Stochastic integration and differential equations}.
\newblock Springer, 2nd edition, 2004.

\bibitem{Pruitt1981}
W.~E. Pruitt.
\newblock The growth of random walks and {L}\'evy processes.
\newblock {\em The Annals of Probability}, 9(6):948--956, 1981.

\bibitem{revuz2013continuous}
D.~Revuz and M.~Yor.
\newblock {\em Continuous martingales and {B}rownian motion}.
\newblock Springer Science \& Business Media, 3rd edition, 1999.

\bibitem{Ricciardi1976}
L.~M. Ricciardi.
\newblock On the transformation of diffusion processes into the {W}iener
  process.
\newblock {\em Journal of Mathematical Analysis and Applications},
  54(1):185--199, 1976.

\bibitem{sato1999}
K.-I. Sato.
\newblock {\em L\'evy processes and infinitely divisible distributions}.
\newblock Cambridge University Press, Cambridge, 1999.

\bibitem{savtoa}
M.~Savov and B.~Toaldo.
\newblock Semi-{M}arkov processes, integro-differential equations and anomalous
  diffusion-aggregation.
\newblock {\em Annales de l'Institut Henri Poincaré, Probabilités et
  Statistiques}, 56(4):2640--2671, 2020.

\bibitem{Scher1973a}
H.~Scher and M.~Lax.
\newblock Stochastic transport in a disordered solid. {I}. {T}heory.
\newblock {\em Physical Review B}, 7(10):4491--4502, 1973.

\bibitem{Scher1973b}
H.~Scher and M.~Lax.
\newblock Stochastic transport in a disordered solid. {II}. {I}mpurity
  conduction.
\newblock {\em Physical Review B}, 7(10):4502--4519, 1973.

\bibitem{Scher1975}
H.~Scher and E.~W. Montroll.
\newblock Anomalous transit-time dispersion in amophous solids.
\newblock {\em Physical Review B}, 12(6):2455--2477, 1975.

\bibitem{Schilling2012}
R.~Schilling, R.~Song, and Z.~Vondra\v{c}ek.
\newblock {\em Bernstein functions Theory and applications}.
\newblock De Gruyter, 2nd edition, 2012.

\bibitem{Seisenberger2001}
G.~Seisenberger, M.~U. Ried, T.~Endreß, H.~B\"{u}ning, M.~Hallek, and
  C.~Br\"{a}uchle.
\newblock Real-time single-molecule imaging of cellular entry of an
  adeno-associated virus.
\newblock {\em Science}, 294:1929--1932, 2001.

\bibitem{SelhuberUnkel2009}
C.~Selhuber-Unkel, P.~Yde, K.~Berg-S{\o}rensen, and L.~B. Oddershede.
\newblock Particle diffusion in living cells: Effects of laser tweezers.
\newblock {\em Physical Biology}, 6:025015, 2009.

\bibitem{straka2011lagging}
P.~Straka and B.~I. Henry.
\newblock Lagging and leading coupled continuous time random walks, renewal
  times and their joint limits.
\newblock {\em Stochastic processes and their applications}, 121(2):324--336,
  2011.

\bibitem{toaldo2015}
B.~Toaldo.
\newblock L\'evy mixing related to distributed order calculus, subordinators
  and slow diffusions.
\newblock {\em Journal of Mathematical Analysis and Applications},
  430(2):1009--1036, 2015.

\bibitem{TolicNorrelykke2004}
I.~M. Toli\'{c}-N{\o}rrelykke, E.~L. Munteanu, G.~Thon, L.~Oddershede, and
  K.~Berg-S{\o}rensen.
\newblock Anomalous diffusion in living yeast cells.
\newblock {\em Physical Review Letters}, 93:078102, 2004.

\bibitem{Uchaikin2003}
V.~V. Uchaikin and V.~V. Saenko.
\newblock Stochastic solution to partial differential equations of fractional
  orders.
\newblock {\em Sibirskii Zhurnal Vychislitel'noi Matematiki}, 6(2):197--203,
  2003.

\bibitem{Weber2010}
S.~C. Weber, A.~J. Spakowitz, and J.~A. Theriot.
\newblock Bacterial chromosomal loci move subdiffusively through a viscoelastic
  cytoplasm.
\newblock {\em Physical Review Letters}, 104:238102, 2010.

\bibitem{zaslavsky3}
G.~M. Zaslavsky.
\newblock Fractional kinetic equation for {H}amiltonian chaos.
\newblock {\em Physica D: Nonlinear Phenomena}, 76(1-3):110--122, 1994.

\bibitem{zaslavsky4}
G.~M. Zaslavsky.
\newblock Chaos, fractional kinetics, and anomalous transport.
\newblock {\em Physics reports}, 371(6):461--580, 2002.

\bibitem{zaslavsky5}
G.~M. Zaslavsky.
\newblock {\em Hamiltonian Chaos and Fractional Dynamics}.
\newblock Oxford University Press, 2005.

\end{thebibliography}

\bigskip

{\bf Ivan Bio\v{c}i\'c}

Department of Mathematics, Faculty of Science, University of Zagreb, Zagreb, Croatia,

Department of Mathematics “Giuseppe Peano”, University of Turin, Turin, Italy,

Email: \texttt{ivan.biocic@unito.it}, \texttt{ivan.biocic@math.hr}

\bigskip

{\bf Daniel E. Cedeño-Girón}

Department of Mathematics “Giuseppe Peano”, University of Turin, Turin, Italy,

Email: \texttt{danieleduardo.cedenogiron@unito.it}

\bigskip
{\bf Bruno Toaldo}

Department of Mathematics “Giuseppe Peano”, University of Turin, Turin, Italy,

Email: \texttt{bruno.toaldo@unito.it}
	
\end{document}